\documentclass[11pt]{article}
\usepackage{epsfig,makeidx,amsmath,amsfonts,latexsym,subfigure}
\usepackage{times,pifont}
\usepackage[utf8]{inputenc}
\usepackage{booktabs}       
\usepackage{graphicx}
\usepackage{bm}
\usepackage{xcolor}
\usepackage{enumitem}

\def\smallpar{\smallbreak\@afterindentfalse\@afterheading\ignorespaces}

\newtheorem{theorem}{Theorem}[section]
\newtheorem{proposition}[theorem]{Proposition}
\newtheorem{lemma}[theorem]{Lemma}
\newtheorem{corollary}[theorem]{Corollary}
\newtheorem{definition}[theorem]{Definition}

\newtheorem{remark}[theorem]{Remark}
\newenvironment{proof}{{\em Proof. }}{\hfill $\Box$ \vspace{1em}}
\newenvironment{example}{\noindent {\bf Example. }}
{\hfill $\Box$ \vspace{1em}}
\usepackage{xcolor}

\newcommand{\bS}{{\mathbf S}}
\newcommand{\bU}{{\mathbf U}}
\newcommand{\bV}{{\mathbf V}}

\newcommand{\bW}{{\mathbf W}}

\newcommand{\E}{{\mathbb E}}
\newcommand{\cC}{{\mathcal C}}
\newcommand{\cE}{{\mathcal E}}
\newcommand{\cF}{{\mathcal F}}

\newcommand{\cL}{{\mathcal L}}
\newcommand{\cU}{{\mathcal U}}

\newcommand{\cN}{{\mathcal N}}
\newcommand{\cG}{{\mathcal G}}
\newcommand{\cX}{{\mathcal X}}
\newcommand{\cY}{{\mathcal Y}}

\newcommand{\cJ}{{\mathcal J}}
\newcommand{\cH}{{\mathcal H}}

\newcommand{\cZ}{{\mathcal Z}}
\newcommand{\bY}{{\mathbf Y}}
\newcommand{\bZ}{{\mathbf Z}}
\newcommand{\bX}{\mathbf{X}}

\newcommand{\Leb}{{\mathcal L}}

\newcommand{\Var}{{\mathbb V}{\rm ar}}
\newcommand{\Cov}{{\mathbb C}{\rm ov}}
\newcommand{\RR}{{\mathbb R}}

\newcommand{\Z}{{\mathbb Z}}

\newcommand{\hh}{\mathbf{h}}
\newcommand{\ff}{\mathbf{f}}
\newcommand{\uu}{{\mathbf u}}
\newcommand{\vv}{{\mathbf v}}
\newcommand{\ww}{{\mathbf w}}

\newcommand{\bb}{{\mathbf b}}
\newcommand{\xx}{\mathbf{x}}
\newcommand{\yy}{{\mathbf y}}
\newcommand{\ee}{{\mathbf e}}
\newcommand{\zz}{{\mathbf z}}
\newcommand{\bs}{{\mathbf s}}
\newcommand{\bz}{{\mathbf 0}}

\newcommand{\aaa}{{\mathbf a}}
\newcommand{\btheta}{\bm{\theta}}
\newcommand{\bTheta}{\bm{\Theta}}
\newcommand{\bphi}{\bm{\phi}}

\newcommand{\balpha}{\bm{\alpha}}
\newcommand{\bbeta}{\bm{\beta}}

\newcommand{\bmu}{\bm{\mu}}
\newcommand{\bdelta}{\bm{\delta}}
\newcommand{\bxi}{\bm{\xi}}
\newcommand{\bLambda}{\bm{\Lambda}}

\newcommand{\Pro}{{\mathbb P}}

\newcommand{\one}{{\mathbf 1}}
\newcommand{\taum}{\tau_{{\rm mix}}}
\newcommand{\tauh}{\wh{\tau}}

\newcommand{\n}{\|}
\newcommand{\bi}{\begin{itemize}}
\newcommand{\ei}{\end{itemize}}
\newcommand{\ra}{\rightarrow}

\newcommand{\ep}{\epsilon}
\newcommand{\si}{\sigma}

\newcommand{\iy}{\infty}
\newcommand{\wh}{\widehat}
\newcommand{\beq}{\begin{equation}}
\newcommand{\bsq}{\end{equation}}
\newcommand{\beqa}{\begin{eqnarray*}}
\newcommand{\bsqa}{\end{eqnarray*}}
\newcommand{\btm}{\begin{theorem}}
\newcommand{\etm}{\end{theorem}}
\newcommand{\bpf}{\begin{proof}}
\newcommand{\epf}{\end{proof}}
\newcommand{\bla}{\begin{lemma}}
\newcommand{\ela}{\end{lemma}}
\newcommand{\bdn}{\begin{definition}}
\newcommand{\edn}{\end{definition}}
\newcommand{\bpn}{\begin{proposition}}
\newcommand{\epn}{\end{proposition}}
\newcommand{\bcy}{\begin{corollary}}
\newcommand{\ecy}{\end{corollary}}

\def\ldotsplus{\mathinner{\ldotp\ldotp\ldotp\ldotp}}
\def\fourdots{\relax\ifmmode\ldotsplus\else$\m@th \ldotsplus\,$\fi}

\def \one {{\mathbf 1}}

\begin{document}
\title{The mixing time of Mean-Field Gibbs-Ornstein-Uhlenbeck processes}
\author{Johan Jonasson\thanks{Chalmers University of Technology and University of Gothenburg, S-412 96 Gothenburg, Sweden, jonasson@chalmers.se  }
\thanks{Research partly supported by WASP (Wallenberg Autonomous Systems and Software Program)}}
\maketitle
\begin{abstract}

We introduce mean-field Gibbs--Ornstein--Uhlenbeck (MFGOU) processes, 
a class of discrete-time Markov processes indexed by a large-scale parameter $n$, with increments of order $1/n$, whose conditional expected increments approximate a smooth vector field. 
When the field has a unique attracting zero and satisfies a global Lyapunov condition, we prove sharp mixing-time asymptotics 
\[ \tau_{\mathrm{mix}} = \frac{1}{\lambda_1}n\log n+O(n), \] 
where $\lambda_1$ is the smallest eigenvalue of the Jacobian of the drift at the attracting point. 
An additional locally attracting zero gives instead an exponential lower bound on the mixing time. 

Mean-field block measures form an important subclass of MFGOU processes that 
encompasses both mean-field spin systems and
projected Gibbs samplers for Latent Dirichlet Allocation (LDA) and related mixture
models.
For this subclass, the theory
gives a common spectral formula for the leading mixing constant in terms of
the free-energy Hessian and a covariance matrix associated with the
single-step dynamics.
A block-refinement argument lifts the projected-chain results to the full chain under a mild spectral condition, without changing the leading mixing constant.
For the Curie--Weiss Ising and Potts models and complete
multipartite Ising models, it gives a common derivation of known cutoff
locations. It also gives an external-field formula for Curie--Weiss Ising
and extends the analysis to interactions specified by a regular base graph.
For multinomial mixtures and partially collapsed LDA, it yields explicit
$n\log n$ mixing-time asymptotics under natural identifiability conditions.
For full LDA, it identifies the free-energy landscape as the remaining
problem and determines the mixing behavior conditional on its
stationary-point structure, distinguishing genuine metastability from slow
transitions caused solely by topic-label switching.

The framework also applies to continuous-state Gibbs sampling and constant-step stochastic gradient descent under suitable boundedness and regularity assumptions. For a frequently refreshed Hamiltonian Monte Carlo scheme, the framework gives a rigorous result when the Gaussian noise in the increments is truncated, and suggests the same leading-order asymptotics in the untruncated case.

The proof of the mixing-time asymptotics combines a uniform variance bound and tracking of the mean with a path-coupling argument that allows slight expansion in expected distance. 
After a deterministic rescaling, the coupling distance is a nonnegative supermartingale with sufficiently large quadratic variation, which gives a probability bounded away from zero of coalescence on the relevant time scale.

\end{abstract}


\section{Introduction}

Many Markov chains arising in statistical mechanics, 
Bayesian computation and stochastic optimization have a common mean-field structure. 
Their states record empirical proportions or other aggregate statistics of a system whose size is measured by a large parameter $n$; here $n$ may represent, for example, the number of sites, observations or text tokens, or the inverse of a learning rate. 
In each step, the state changes by an amount of order $1/n$, and its conditional expected change approximates a deterministic vector field.
Near a stable equilibrium of this vector field, 
the resulting dynamics resembles an Ornstein–Uhlenbeck process with small noise: 
the mean contracts on the time scale $n$, while equilibrium fluctuations have size $1/\sqrt{n}$. 
This suggests a mixing time of order $n\log n$, with a leading constant determined by the linearization of the drift at the equilibrium. 
If the drift instead has several stable equilibria, transitions between their neighborhoods may require an exponentially long time.

The aim of this paper is to make this principle precise for a general class of discrete-time Markov chains, which we call mean-field Gibbs–Ornstein–Uhlenbeck processes, or MFGOU processes. 
We prove sharp $n\log n$ mixing-time asymptotics in the unimodal regime and exponentially slow mixing in the presence of an additional locally attracting equilibrium. 
The framework was originally motivated by collapsed Gibbs sampling for Latent Dirichlet Allocation, 
but also covers classical and block-structured mean-field spin systems, constant-step stochastic gradient descent, 
and a frequently refreshed Hamiltonian Monte Carlo scheme.

\smallskip

We now introduce the MFGOU framework. Here and throughout the paper, the following conventions will be used.

\begin{itemize}
\item All constants, including those hidden in $O(\cdot)$, are independent of $n$, unless explicitly stated otherwise.
\item For readability, we will use the letters $Q$ and $q$ for a generic positive finite constant that may change between occurrences.
In general, we will use $Q$ when the point is that there exists a finite constant with some property and $q$ when the point is that the constant is nonzero.
\item The terms discrete, continuous, and semicontinuous refer to the spatial structure of the process, rather than to its time parametrization. 
All MFGOU processes considered here evolve in discrete time.
\end{itemize}
\smallskip

The framework has two parts. 
The common structure is a small-increment Markov process whose conditional drift approximates a smooth vector field with a nondegenerate attracting zero. 
We then add one of two sets of local regularity conditions, according to whether the state space is discrete or continuous. 
Here and throughout, $\ee_j$ denotes the $j$th standard basis vector in $\RR^m$.


\smallskip

Let $\cN$ be either $[N,\iy)$ or $[N,\iy) \cap \Z$ for some positive integer $N$.

\smallskip

{\bf Basic MFGOU:} A basic Mean-Field Gibbs-Ornstein-Uhlenbeck (MFGOU) family $\{\cY^{(n)}\}_{n \in \cN}$ is a family of time-homogeneous Markov processes $\cY^{(n)}=\{\bY^{(n)}_t\}_{t=0}^\iy$ whose common state space is a closed set $\cC \subseteq \RR^m$ with nonempty interior.
Each $\cY^{(n)}$ admits unique stationary probability distribution $\pi^{(n)}$, and has transitions of the following form:
\begin{equation} \label{eq:MFGOU}
    \bY^{(n)}_{t+1}=\bY^{(n)}_t + \frac{1}{n}\bxi^{(n)}_t
\end{equation}
where the $\bxi^{(n)}_t$'s are random variables that satisfy the following conditions.
\begin{itemize}
    \item The range of $\bxi^{(n)}_t$ is contained in a compact $K \subseteq \RR^m$ independent of $n$.  
    \item $\E[\xi^{(n)}_t|\bY^{(n)}_t=\yy]=-\ff(\yy)/2+O(1/n)$ for all $\yy$, where $\ff:\cC \ra \RR^m$ is a $C^2$ vector field for which there exists at least one point $\yy_0 \in \cC^\circ$ such that
    \begin{itemize}
        \item[(a)] $\ff(\yy_0)=\bz$,
        \item[(b)] $\ff(\yy)=\cJ (\yy-\yy_0) + \mathbf{R}(\yy-\yy_0)$, where $\mathbf{R}(\zz)=O(\n\zz\n^2)$ and the Jacobian matrix $\cJ=\nabla \ff(\yy_0)$ is symmetric positive definite,
        \end{itemize}
        and where the $O(1/n)$-term can be taken uniform in $\yy$ on compact sets.
\end{itemize}

\medskip

A basic MFGOU family is said to be  {\em locally regular} at a point $\yy'$ if there exists $\kappa>0$ such that the law $\cL(\bxi^{(n)}(\yy))$ of $\bxi^{(n)}_t \mid \bY^{(n)}_t=\yy$ is Lipschitz continuous in total variation as a function of $\yy$ on
$B(\yy',\kappa)$, with a Lipschitz constant uniform in $n$.

Such a family is called {\em discrete} or {\em semicontinuous} if, for some point $\yy_0$ satisfying (a)–(b), it is locally regular at $\yy_0$ and the corresponding additional conditions below hold.



\begin{itemize}
    \item \textbf{Discrete MFGOU:} The range of $\bY^{(n)}_t$ is $\cC \cap (1/n)\Z^m$, $\bxi^{(n)}_t$ is a random variable with support $K \cap \Z^m$, and
    there exists $\eta > 0$ independent of $\yy$ and $n$ such that 
    \[\Pro(\bxi^{(n)}(\yy)=\bz) \geq 2m\eta\] and 
    \[\Pro(\bxi^{(n)}(\yy)=\xx) \geq \eta,\] for all $\yy \in B(\yy_0,\kappa)$ and $\xx \in \{\pm \ee_1,\ldots,\pm \ee_m\}$.

    \item \textbf{(Semi-)continuous MFGOU:}  
    The laws $\cL(\bxi^{(n)}(\yy))$ can be decomposed as 
    \[\cL(\bxi^{(n)}(\yy)) = a^{(n)}(\yy)\delta_{\bz}+(1-a^{(n)}(\yy))\mu^{(n)}(\yy),\]
    where $\delta_{\xx}$ is the point mass at $\xx$, $\mu^{(n)}(\yy)$ is the law of an absolutely continuous variable, and $a^{(n)}(\yy) \in [0,1-a_0]$ for some $a_0>0$.
    
    Furthermore, there exist $b_0 \in (0,1]$ and a symmetric absolutely continuous probability measure $\mu_s$, both independent of $n$ and $\yy$, such that
    $\mu^{(n)}(\yy)$ admits the decomposition 
    \[\mu^{(n)}(\yy) = b_0 \mu_s +(1-b_0)\mu^{(n)}_r(\yy).\] 

    
    The case $a^{(n)}(\yy)=0$ will be referred to as a {\bf continuous MFGOU family}.
\end{itemize}

\hfill $\Box$

Note that the additional assumptions in these definitions imply that the covariance of $\bxi^{(n)}_t$ is nondegenerate (i.e.\ the covariance matrix is positive definite) uniformly in $n$.
In the semicontinuous case, $a^{(n)}$ accounts for a possible holding probability, $\mu_s$ provides the symmetric noise used to establish the existence of a short-range coupling below, and $\mu_r^{(n)}$ collects the remaining absolutely continuous part.
The Lipschitz condition on the increment law implies that $a^{(n)}(\yy)$ is Lipschitz uniformly in $n$.
This also holds for $\mu^{(n)}_r(\yy)$ unless $b_0=1$, in which the residual component is absent.

We refer to each \(\cY^{(n)}\) as an MFGOU process (at scale $n$). 
When $n$ is understood, we suppress the superscript on $\cY^{(n)}$ and all its associated quantities.  
Thus, a statement that \(\cY\) is an MFGOU process means that \(\cY=\cY^{(n)}\) is the process at scale \(n\) belonging to an MFGOU family.
We retain or restore the superscript when dependence on $n$ is relevant.

\smallskip

Our main theorem is the following. 
\begin{theorem} \label{thm:main_theorem}
Let $\{\cY^{(n)}\}_{n \in \cN}$ be a discrete or semicontinuous MFGOU family, and write \(\cY=\cY^{(n)}\) for its member at scale $n$.
Assume that there is a unique point $\yy_0$ that satisfies (a)-(b), and that there exists a $C^2$ Lyapunov function $V:\cC \ra [0,\iy)$ with compact sublevel sets, $V(\yy)=0$ if and only if $\yy=\yy_0$, $\nabla^2 V(\yy_0)$ positive definite, and such that for some $c>0$, 
\[\nabla V(\yy)^\top \ff(\yy) \geq c(V(\yy) \wedge 1)\] 
for all $\yy \neq \yy_0$.
Let $\lambda_1$ be the smallest eigenvalue of $\cJ$.
Then for any $\ep \in (0,1)$ the mixing time started from $\bY_0=\yy$ satisfies
\[\taum(\cY;\ep,\yy) \leq \frac{1}{\lambda_1}n\log n + O_\ep(n(1+V(\yy))).\]
This bound is sharp as a worst-case bound, for example, when $\yy=\yy^{(n)}$ is a point 
whose projection onto the eigenspace corresponding to $\lambda_1$ is bounded away from zero.

In particular, if the state space $\cC$ of $\cY$ is compact, then
\[\taum(\cY;\ep) = \frac{1}{\lambda_1}n\log n + O_\ep(n),\]
and, consequently, $\cY$ exhibits cutoff at time $\tfrac{1}{\lambda_1}n\log n$.
\end{theorem}

Note that, since $\yy_0 \in \cC^\circ$, the state space $\cC$ is guaranteed to contain a point $\yy^{(n)}$ as in the statement of the theorem.

\begin{remark} \label{remark:stationary_distribution_redundant}
    {\rm
    The stationarity assumption is included in the definition of an MFGOU family so that mixing times are defined throughout the general, possibly unbounded setting.
    If $\cC$ is compact, then under the other conditions of Theorem \ref{thm:main_theorem}, this assumption is redundant.
    This is because the upper-bound part of the proof of the theorem does not presuppose the existence of a stationary distribution: it constructs, uniformly over all pairs of starting states, a coupling that succeeds with probability bounded away from zero by some deterministic time $T=T_n<\infty$. Consequently, for some $\rho<1$, the transition kernel $P$ of the process satisfies
\[
\sup_{\yy,\yy'}\left\|P^T(\yy,\cdot)-P^T(\yy',\cdot)\right\|_{\mathrm{TV}}\leq \rho.
\]
It follows that $\mu\mapsto\mu P^T$ is a strict contraction on the complete metric space of probability measures equipped with total variation distance, and hence has a unique fixed point $\pi$. 
Since
\[(\pi P)P^T=(\pi P^T)P=\pi P,\]
the measure $\pi P$ is also fixed by $P^T$. 
Uniqueness therefore gives $\pi P=\pi$. Thus $\pi$ is stationary for $P$, and it is unique since every stationary distribution for $P$ is also fixed by $P^T$.
    }
\end{remark}

A complementary result covers cases where the drift $\ff$ has more than one locally attracting zero. 
Such points act as local sinks, producing exponential mixing times.
Note that the definition of mixing time in Section \ref{subsec:def_of_total_variation_etc} remains meaningful even when the process does not converge to stationarity in total variation, in which case the mixing time may be infinite.

\begin{theorem} \label{thm:exponential_mixing_time}
Let $\{\cY^{(n)}\}_{n \in \cN}$ be a basic MFGOU family, and write \(\cY=\cY^{(n)}\) for its member at scale $n$.
Assume that there exist two distinct points $\yy_0$ and $\yy_1$ such that $\ff(\yy_i)=\bz$, $i=0,1$, and that there exist nonnegative $C^2$ Lyapunov functions $V_i$ on $\cC$, and a $c>0$ 
such that $V_i(\yy)=0$ iff $\yy=\yy_i$, $\nabla^2 V_i(\yy_i)$ positive definite, and $\nabla V_i(\yy)^\top \ff(\yy) \geq cV_i(\yy)$
throughout some sufficiently small ball $B(\yy_i,\kappa_i)$.
Then there exists $c'>0$ such that for any $0<\ep<1/2$,
\[\taum(\cY;\ep) \geq e^{c'n}.\]
\end{theorem}
If $\ff$ has a potential, so that $\ff=\nabla g$, Theorem \ref{thm:exponential_mixing_time} describes the situation created by separate nondegenerate local minima of $g$.

\smallskip

The precise definition of mixing time will appear in Section \ref{sec:preliminaries}. 
In the setting of Theorem \ref{thm:main_theorem}, we may without loss of generality translate coordinates so that $\yy_0=\bz$, and we make that assumption throughout the proof of that theorem.

The following remark is important.
\begin{remark} \label{remark:Sigma*H}
{\rm
The matrix $\cJ$ in the definition of the MFGOU family equals the Jacobian $\nabla \ff(\bz)$. Sometimes we will encounter situations where the Jacobian is of the form $\cJ= \Sigma \cH $, where $\Sigma$ and $\cH$ are symmetric positive definite, 
but the product typically is not.
In that case, consider the change of coordinates $\tilde{\yy}=\Sigma^{-1/2}\yy$ 
and define the corresponding transformed drift field by $\tilde{\ff}(\tilde{\yy})=\Sigma^{-1/2}\ff(\Sigma^{1/2}\tilde{\yy})$.
Then the leading term of the drift of $\tilde{\cY}$ becomes 
\[\nabla \tilde{\ff}(\bz)\tilde{\yy} = \Sigma^{-1/2}\nabla \ff(\bz)\Sigma^{1/2}\tilde{\yy} = \Sigma^{1/2}\cH\Sigma^{1/2}\tilde{\yy}.\]
The processes $\cY$ and $\tilde{\cY}$ differ only by an invertible linear change of coordinates and consequently have the same mixing time.
Theorem \ref{thm:main_theorem} now applies with $\cJ=\Sigma^{1/2}\cH\Sigma^{1/2}$ and the result becomes
\[\taum(\cY;\ep) = \frac{1}{\lambda_1(\Sigma^{1/2}\cH\Sigma^{1/2})}n\log n + O(n) = \frac{1}{\lambda_1(\Sigma\cH)}n\log n + O(n).\]
In the discrete case, the linear transformation $\tilde{\yy}=\Sigma^{-1/2}\yy$ maps the lattice $(1/n)\Z^m$ to another full‑dimensional lattice of mesh $O(1/n)$. 
This has no effect on any of the arguments, as the coupling and path-coupling constructions are invariant under such linear changes of coordinates.
}   
\end{remark}

\begin{remark} \label{remark:local_dynamics}
{\rm 
The point $\yy_0$ in Theorem \ref{thm:main_theorem} and the points $\yy_0$ and $\yy_1$ in Theorem \ref{thm:exponential_mixing_time} have the same local dynamical character. 
Around each such point, the process has a contracting drift and increments of order $1/n$, and the same local analysis applies until the process leaves the surrounding basin. 
Under the local Lyapunov condition of Theorem \ref{thm:exponential_mixing_time}, it remains near either $\yy_i$ for an exponentially long time with high probability. 
Thus the behavior inside a metastable basin can be analyzed by the same local methods as the behavior near the equilibrium basin, up to the exponentially large exit time. 
In particular, if the local regularity conditions of a discrete or semicontinuous MFGOU also hold around such a point, rapid mixing within the corresponding metastable basin occurs.}
\end{remark} 

\begin{remark} \label{remark:maximal_generality}
{\rm 
The assumption that the Jacobian $\cJ=\nabla \ff(\yy_0)$ is symmetric positive definite is convenient, but it can be relaxed. 
If $\ff$ is $C^1$, $\ff(\yy^*)=\bz$, and all eigenvalues of $\nabla \ff(\yy^*)$ have positive real parts, 
then there exists a positive definite matrix $P$ such that 
\[ V(\zz)=\zz^\top P\zz, \,\, \zz=\yy-\yy^*, \] 
is a local quadratic Lyapunov function satisfying 
\[ \nabla V(\zz)^\top \ff(\yy^*+\zz) \geq cV(\zz) \] 
for all sufficiently small $\zz$. 
This is a classical consequence of Lyapunov's indirect method; see, for example, Khalil~\cite[Theorems 4.6--4.7]{Khalil}. 
One could therefore define a more general MFGOU process by replacing the symmetric positive-definiteness assumption on $\cJ$ with the assumption that all eigenvalues of $\cJ$ have positive real parts. 
We expect Theorems \ref{thm:main_theorem} and \ref{thm:exponential_mixing_time} to extend to this setting, with $\lambda_1$ in Theorem \ref{thm:main_theorem} interpreted as the minimum of the real parts of the eigenvalues of $\cJ$. 
If $\cJ$ is diagonalizable, the remainder should remain $O(n)$, whereas nontrivial Jordan blocks may introduce additional lower-order terms.
We do not pursue this additional generality here. 

\smallskip

A different generalization would be to dispense with the discrete and semicontinuous conditions. 
Much of the analysis applies to any basic MFGOU. In particular, as will be evident from the proof below, the process reaches an $O(1/\sqrt n)$-neighborhood of the attracting point by time 
\[ \frac{1}{\lambda_1}n\log n+O(n) \] 
with high probability, with this time scale being sharp for suitable starting points. 
Once this neighborhood has been reached, only an $O_\ep(n)$ endgame remains, loosely speaking.
On a more general state space, however, convergence in total variation may fail, since nearby copies need not be able to coalesce. 
We expect an analogous mixing-time result to hold in an appropriate weaker metric, under suitable assumptions on the transition kernels. 
Developing such a formulation would require a separate framework, and we do not pursue it here.

}
\end{remark}

\begin{remark} \label{remark:various_things_on_MFGOU_assumptions}
    {\rm 
The following observations clarify the role of the Lyapunov condition and how the global MFGOU family assumptions are used.

    \begin{enumerate}[label=(\roman*)]
    \item If $\ff(\yy)$ has a potential, then $\ff(\yy)$ can be written as $\nabla g(\yy)$ for some $g$ with $g(\yy_0)=\bz$. 
    The Lyapunov inequality of Theorem \ref{thm:main_theorem} then follows if $\n\nabla g(\yy)\n^2 \geq c (g(\yy) \wedge 1)$ by taking $V=g$.
    If in addition $g$ is nonnegative, vanishes only at $\yy_0$, and has compact sublevel sets, then all Lyapunov hypotheses of Theorem \ref{thm:main_theorem} hold.
    \item In some applications, the process $\cY^{(n)}$ under study is not an MFGOU process globally, but satisfies the MFGOU conditions only on a region containing the attracting point. 
    If a smaller subset of that region is reached within time $O(n)$ with high probability and the process is unlikely to leave the larger region before an exponentially long time, then the censoring results of Section \ref{subsec:censoring} allow the process to be treated as confined to the larger region. Theorem \ref{thm:main_theorem} then applies to the censored process, and the time taken to reach the smaller region contributes only an additional $O(n)$ term to the mixing-time upper bound.
    \end{enumerate}
    }
\end{remark}

The proof of Theorem \ref{thm:main_theorem} separates the approach to equilibrium into two scales. 
First, a Lyapunov argument brings the process into any prescribed sufficiently small neighborhood of the attracting point in time $O(nV(\bY_0))$ and keeps it within a slightly larger neighborhood for an exponentially long time.
Inside that neighborhood, we prove that the variance of the process is uniformly of order $1/n$ and that its mean closely follows the linearized deterministic recursion. 
The time required for the mean to come within distance $O(1/\sqrt{n})$ of the origin is then $\tfrac{1}{\lambda_1}n\log n+O(n)$.
At this point, its distance from the origin is of the same order as its fluctuations, and the process is close to stationarity.
For the upper bound, two copies are first run independently until they both reach the $1/\sqrt{n}$-scale equilibrium region. 
The main coupling difficulty is that the natural coupling need not contract the expected distance at every step: it may exhibit an expansion by a factor $1+O(1/n)$. 
After compensating for this expansion by a deterministic rescaling, the distance becomes a supermartingale whose quadratic variation remains sufficiently large. 
Lemma \ref{lemma:path_coupling_with_expansion} shows that such a process has a probability bounded away from zero of hitting zero within the relevant time scale, even in the absence of strict contraction in its conditional mean. 
Applying this to a combination of a short-range and an inert coupling gives coalescence with probability bounded away from zero in time $O(n)$, and repeated attempts complete the proof.

\smallskip

The applications illustrate both the scope of the framework and the role of the local drift geometry. For continuous unimodal densities, we analyze a Barker-type Gibbs sampler and obtain a sharp mixing constant determined by the local curvature of the target log-density at its mode and the proposal covariance. 

For discrete mean-field block measures, we derive a general formula in which the effective curvature is the product of the free-energy Hessian and a covariance matrix associated with the single-step dynamics. 
By partitioning each block according to the initial spin labels, the projected-chain result can be lifted to the full chain under a mild spectral condition, with no change in the leading constant.
This recovers the classical Curie--Weiss Ising and Potts cutoff locations \cite{LLP,Cuff}, recovers known results for complete multipartite Ising models \cite{Kim2022}, and yields an extension to block interactions specified by a regular base graph.

The original motivation for the framework was collapsed Gibbs sampling for Latent Dirichlet Allocation \cite{BNJ}.
After projection onto topic-allocation frequencies, the Gibbs samplers for the LDA variants considered here become instances of the mean-field block-measure framework, and the block-partitioning argument transfers the resulting conclusions back to the full samplers.
For multinomial mixtures and partially collapsed LDA, the theory gives $n\log n$ mixing-time asymptotics under natural identifiability conditions. 
For full LDA, it reduces the mixing-time behavior of both the projected and the full chains to the stationary-point structure of an associated free energy. 
A complete analytic classification of that landscape remains open, but the framework distinguishes genuine metastability, as in \cite{slow}, from slow transitions caused solely by topic-label switching, and the proposed landscape picture is supported by numerical examples. 

Finally, applications to constant-step stochastic gradient descent and frequently refreshed Hamiltonian Monte Carlo show that the same contraction-and-fluctuation mechanism also arises outside single-site Gibbs sampling.

\smallskip

The paper is organized as follows. Section \ref{sec:preliminaries} develops the preliminary hitting-time, censoring, and coalescence tools. Section \ref{sec:mixing_time_proof} proves the main mixing-time results by establishing confinement near the attracting point, a uniform variance bound, tracking of the mean, and the final lower- and upper-bound arguments. Section \ref{sec:applications} contains the applications. Section \ref{subsec:Metropolis--Hastings} treats Gibbs and Metropolis--Hastings sampling from continuous densities, Section \ref{subsec:mean_field_glauber} develops the general block-measure framework and applies it to Ising and Potts models, LDA, and multinomial mixture models, and Sections \ref{subsec:SGD} and \ref{subsec:HMC} discuss stochastic gradient descent and Hamiltonian Monte Carlo.


\section{Preliminaries} \label{sec:preliminaries}

\subsection{Total variation distance and mixing time} \label{subsec:def_of_total_variation_etc}

The total variation distance between two probability measures $\mu$ and $\nu$ on $\Lambda$ is defined as
\[\n \mu - \nu \n_{TV} = \sup\{\mu(A)-\nu(A): A \subseteq \Lambda\},\]
where the supremum is over measurable sets.
If $\Lambda$ is countable, then the right-hand side equals $\tfrac{1}{2}\sum_{x \in \Lambda} |\mu(x)-\nu(x)|$ and if $\Lambda$ is a measurable subset of $\RR^m$ and $\mu$ and $\nu$ are absolutely continuous, 
then the right-hand side equals $\tfrac{1}{2}\int|\psi_\mu(x)-\psi_\nu(x)|dx$ for the respective densities $\psi_\mu$ and $\psi_\nu$. 

For a discrete-time Markov chain $\cX=\{X_t\}_{t=0}^\infty$ on state space $\Lambda$ that admits a unique stationary probability  distribution $\pi$, the $\epsilon$-mixing time, $\epsilon>0$, started from $x$ is given by
\[\taum(\cX;\epsilon,x) = \inf\{t: \n \Pro_x(X_t \in \cdot)-\pi \n_{TV} \leq \epsilon\}.\]
For a subset $\Lambda_0$ of $\Lambda$, 
\[\taum(\cX;\epsilon,\Lambda_0) := \sup_{x \in \Lambda_0} \taum(\cX;\ep,x).\]
When the starting point is entirely unspecified, one refers to
$\taum(\cX;\epsilon) := \taum(\cX;\ep,\Lambda)$.
(In general, these worst-case quantities may be infinite.) 

If $\{\cX^{(n)}\}$ is a sequence of Markov chains such that for each $0 < \ep < 1$,
\[\lim_{n \ra \iy}\frac{\taum(\cX^{(n)};\ep)}{\taum(\cX^{(n)};1-\ep)} = 1,\]
then we say that $\{\cX^{(n)}\}$ exhibits cutoff at $t_n$ for any $t_n$ such that $\taum(\cX^{(n)};a)/t_n \ra 1$ for some (and hence all) $0<a<1$, and if the $n$ in $\cX^{(n)}$ is understood, we simply say that $\cX$ exhibits cutoff.

We also define 
\[\tauh(\cX;\epsilon) = \sup_{x,x' \in \Lambda}\inf\{t: \n \Pro_{x'}(X_t \in \cdot)-\Pro_{x}(X_t \in \cdot) \n_{TV} \leq \epsilon\}.\]
In these definitions, $\Pro_x$ refers to the conditional probability measure given $X_0=x$. Obviously $\taum(\cX;\epsilon) \leq \tauh(\cX;\epsilon)$. 

The quantity $\tauh$ can be characterized by coupling. For $x,x' \in \Lambda$ let 
\[\tauh(x,x';\ep)=\inf\{t:\inf_{\cX^x,\cX^{x'}}\Pro(T(\cX^x,\cX^{x'})>t) \leq \ep\},\]
where the inner infimum is taken over couplings $(\cX^x,\cX^{x'})$ of copies of $\cX$ started from $x$ and $x'$ respectively and $T(\cX^x,\cX^{x'})$ is their coalescence time.
Then
\[
\tauh(\cX;\epsilon) = \sup_{x,x' \in \Lambda}\tauh(x,x';\ep).
\]

A consequence is that $\tauh$ enjoys the contractive property:
\[\tauh(\cX;\epsilon_1\epsilon_2) \leq \tauh(\cX;\epsilon_1) + \tauh(\cX;\epsilon_2).\]

\subsection{A drift-based bound for hitting probabilities} \label{subsec:no_escape_with_negative drift}

For later use in the censoring argument, we need a straightforward estimate for the probability that a process with bounded increments and negative drift reaches a high level before returning to a small one. The statement is standard in nature, but since we have not seen it explicitly in this form, we include a short proof.

\begin{lemma} \label{lemma:no_escape_with_negative drift}

Let $\{X_t\}_{t=0}^\iy$ be a stochastic process on $[0,\iy)$ such that for some $\alpha,\delta,c_0>0$,
\begin{itemize}
    \item $\E[X_{t+1}-X_t \mid X_0^t=x_0^t] \leq -\alpha$ whenever $x_t \geq \delta$,
    \item $|X_{t+1}-X_t| \leq c_0$ almost surely for all $t$.
\end{itemize}
Let $c>0$ and let $\tau_\delta$ be the first time that $X_t$ hits $[0,\delta]$ and $\tau_{cn}$ be the first time it hits $[cn,\iy)$. 
Then if $X_0=x$ and $x \in (\delta,cn/2)$, there is some $q>0$ independent of $n$ such that 
\[\Pro(\tau_{cn} < \tau_{\delta}) \leq e^{-qn}\]
and $\Pro(\tau_{cn} \leq e^{qn/2}) \leq e^{-qn/2}$.
    
\end{lemma}

\begin{proof}
    Fix $\theta>0$, to be chosen below. Assume that $x_t>\delta$.
    Let $\Delta_t=X_{t+1}-X_t$, so that by assumption $\E[\Delta_t|X_0^t=x_0^t] \leq -\alpha$.
    The range of $\Delta_t$ is contained in $[-c_0,c_0]$, so by Hoeffding's lemma for bounded random variables,
    \[\E[e^{\theta X_{t+1}} \mid X_0^t=x_0^t] \leq \exp\left(-\theta\alpha+\frac{\theta^2 c_0^2}{2}\right)e^{\theta X_t}.\]
    Taking $\theta = \alpha/c_0^2$, we get that $\{e^{\theta X_t}\}$ is a supermartingale up to $\tau_\delta$.
    Letting $\tau=\min(\tau_\delta,\tau_{cn})$ and $p = \Pro(\tau_{cn} \leq \tau_\delta)$, optional stopping gives
    \[e^{\theta x} \geq E[e^{\theta X_\tau}] \geq p\,e^{\theta cn}+(1-p),\]
    (where the second term comes from $X_\tau \geq 0$)
    from which it follows that $p \leq 2e^{-\theta(cn-x)} \leq e^{-qn}$.
    The second assertion follows by applying the first bound uniformly to each excursion from $[0,\delta]$ toward $[cn,\infty)$.
    After a return to $[0,\delta]$ the next possible starting point above $\delta$ is bounded by $\delta+c_0$, and then taking a union bound over at most $e^{qn/2}$ attempted excursions.
\end{proof}

\subsection{Censoring and restriction to a smaller domain} \label{subsec:censoring}

The following observations will be used extensively. Together they say essentially that when considering the mixing time of a Markov chain, $\cX$, for which there is a subset of the state space that is (1) quickly reached and (2) with high probability not left for a long time once in it,
one can consider $\cX$ as living only on that subset.

We use the term censored in the broad sense of an auxiliary chain whose kernel agrees with the original kernel on the relevant inner region but is modified outside that region to remain in a prescribed larger region. This is not a standard monotone-censoring construction.

Assume the following setup:
\begin{itemize}
    \item $\cX=\{X_t\}_{t=0}^\iy$ is a Markov chain on a compact subset $\Lambda$ of $\RR^m$,
    \item $\Lambda_0 \subseteq \Lambda_1 \subseteq \Lambda_2 \subseteq \Lambda$,
    \item $\Pro_x(X_{t+1} \in \Lambda_2 \mid X_t=x) = 1$ for all $x \in \Lambda_1$,
    \item $0<\ep<1/7$, $T_1<\iy$ and $\Pro_x(X_0^{T_1} \subseteq \Lambda_1|X_0=x)>1-\ep$ for all $x \in \Lambda_0$. 
\end{itemize}
Here $x_s^t$ denotes $\{x_s,x_{s+1},\ldots,x_t\}$.

Let $\cX^r$ be a censored Markov chain coupled with $\cX$ so that the two chains make identical moves from all states in $\Lambda_1$. 
Outside $\Lambda_1$, any transition that would take $\cX$ outside $\Lambda_2$ is replaced by a prescribed transition staying in $\Lambda_2$.

\begin{lemma} \label{lemma:start_in_subset}
Under the given setup, for any measurable set $A \subseteq \Lambda$, $t \leq T_1$ and $x \in \Lambda_0$,
\[\Pro_x((X^r)_t \in A) \in \Pro_x(X_t \in A) \pm 2\ep.\]
\end{lemma}

\begin{proof}
   By the definition of $\cX^r$, for any $x_0^t \subseteq \Lambda_1$, $\Pro((X^r)_0^t = x_0^t) = \Pro(X_0^t = x_0^t)$. If $\Lambda$ is finite,
   \begin{align*}
   \Pro_x(X_t \in A) &\in \sum_{x_0^t \subseteq \Lambda_1:x_t \in A}\Pro(X_0^t=x_0^t) \, \pm \, \ep \\
   &=  \sum_{x_0^t \subseteq \Lambda_1:x_t \in A}\Pro((X^r)_0^t=x_0^t) \, \pm \, \ep \\
   &= \Pro((X^r)_t \in A) \, \pm \, 2\ep.
   \end{align*}
   If $\cX$ has a continuous kernel, then replace the sum with the corresponding integral.
\end{proof}

Assume in addition that 
\begin{itemize}
    \item $\tauh(\cX^r;\ep) \leq t$ for some $t \leq T_1$,
    \item $\Pro_y(X_{T_0} \in \Lambda_0) \geq 1-\ep$ for any starting point $y \in \Lambda$.
\end{itemize}

\begin{lemma} \label{lemma:mixing_on_fast_hit_restricted_set}
Under the added setup,
\[\tauh(\cX;7\ep) \leq T_0+t.\]
\end{lemma}

\begin{proof}
    By assumption, for any $x,x' \in \Lambda_0$ and any $A$, 
    \[\Pro_x(X^r_t \in A)-\Pro_{x'}(X^r_t \in A) \leq \ep.\]
    Hence, with $u=T_0+t$ and any starting points $y,y' \in \Lambda$,
    \begin{align*}
    \Pro_y(X_u \in A)-\Pro_{y'}(X_u \in A) &\leq 2\ep + \max_{x,x' \in \Lambda_0} \left(\Pro_x(X_t \in A)-\Pro_{x'}(X_t \in A)\right) \\
    &\leq 6\ep + \max_{x,x' \in \Lambda_0} \left( \Pro_x((X^r)_t \in A)-\Pro_{x'}((X^r)_t \in A)\right) \\ 
    &\leq 7\ep.
    \end{align*}
\end{proof}

Taken together, the two lemmas give the following localization principle. Suppose that $\mathcal X$ reaches a set $\Lambda_0$ within time $T_0$ with high probability and, when started there, is unlikely to leave a larger set $\Lambda_1$ before time $T_1$. Then, up to a negligible error and the addition of the entrance time $T_0$, the mixing time of $\mathcal X$ agrees with that of a suitably censored version $\mathcal X^r$ confined to a still larger set $\Lambda_2$, provided that the mixing time of $\mathcal X^r$ is smaller than $T_1$. When this setup has been established, we will say that $\mathcal X$ may be treated as confined to $\Lambda_2$.


\begin{remark}
    {\rm 
In Section  \ref{subsec:to_small_box_in_linear_time}, we will use the freedom in the censoring mechanism to construct a censored chain that preserves the Lyapunov condition of Theorem \ref{thm:main_theorem}. 
The original transition kernel is retained on an inner Lyapunov sublevel set, 
while on the annulus between two larger sublevel sets, the drift is modified so that the censored chain cannot leave the outer set. 
The block-measure applications in Section \ref{subsec:mean_field_glauber} use a different construction adapted to the lattice simplex.    }
\end{remark}

\subsection{Coalescence for nearly contracting coupling distances} \label{subsec:Coalescence for nearly contracting coupling distances}

The following lemma gives a hitting-time estimate for a nonnegative supermartingale with sufficiently large quadratic variation. 
Its main use will be to control coupling distances that are not contracting in expectation but permit an expected relative expansion of order $O(1/n)$ per step. 
After a deterministic rescaling, such a distance becomes a supermartingale, and the remaining fluctuations give a probability bounded away from zero of hitting zero on the natural time scale.
We are not aware of a general path‑coupling theorem that directly covers this regime.
Corollary \ref{cor:path_coupling_with_expansion} below 
translates Lemma \ref{lemma:path_coupling_with_expansion} into a criterion that will later be verified for the coupling constructed in Section \ref{subsec:couplings}.

\begin{lemma} \label{lemma:path_coupling_with_expansion}
  Let $\{W_t\}_{t=0}^\iy$ be a nonnegative stochastic process with increments satisfying $|W_{t+1}-W_t| \leq L < \iy$ a.s., and let $T_0 = \min\{t:W_t =0\}$. 
  Assume that almost surely on the event $\{T_0 > t\}$, $\E[W_{t+1} \mid W_0^t] \leq W_t$ and $\E[(W_{t+1}-W_t)^2 \mid W_0^t] \geq a e^{-ct}$ for some $a,c \in (0,\iy)$. 
  Then
  \[\Pro\left(T_0 \leq \frac{1}{c}\right) \geq \frac{7}{9}\]
  whenever
  \[W_0^2 \vee 4L^2 \leq \frac{a(1-e^{-1})}{1000(1-e^{-c})}.\]
  
\end{lemma}

\begin{proof}
  
Fix $A > W_0$ arbitrarily. We claim that $(A-W_t)^2-a\sum_{s=0}^{t-1}e^{-cs}$ is a submartingale until the first time $T$ such that either $W_T = 0$ or $W_T \geq A$, i.e.\
  \[a e^{-ct} \leq \E[(A-W_{t+1})^2-(A-W_t)^2 \mid W_0^t].\]
   To see that, note first that the supermartingale assumption $\E[W_{t+1} \mid W_0^t] \leq W_t$ implies that $\E[(A-W_{t+1})(A-W_t) \mid W_0^t] \geq (A-W_t)^2$. 
   Therefore
  \begin{align*}
  a e^{-ct} &\leq \E[(W_{t+1}-W_t)^2 \mid W_0^t]\\
  &= \E[((A-W_{t+1})-(A-W_t))^2 \mid W_0^t] \\
  & \leq \E[(A-W_{t+1})^2-(A-W_t)^2 \mid W_0^t],
  \end{align*}
  proving the claim.
  
  Assume now first that $W_0 \geq L$.
  The Optional Stopping Theorem then gives, by taking $A=9W_0$ and noting that $W_T \leq 10W_0$,
  \[\E\left[a\sum_{s=0}^{T-1}e^{-cs}\right] \leq 100W_0^2.\]
  The sum in the left hand expectation equals $(1-e^{-cT})/(1-e^{-c})$, so the inequality implies that
  \[\E[1-e^{-cT}] \leq \frac{100W_0^2(1-e^{-c})}{a},\]
  which by Markov's inequality gives
  \[\Pro\left(e^{-cT} \geq 1-\frac{1000W_0^2(1-e^{-c})}{a}\right) \geq \frac{8}{9}.\]
  By the assumption on $W_0$, this becomes
  \[\Pro(e^{-cT} \geq e^{-1}) \geq \frac{8}{9},\]
  i.e.
  \[\Pro\left(T \leq \frac{1}{c}\right) \geq \frac{8}{9}.\]
  Since $W_t$ itself is a supermartingale and $W_T \geq 9W_0$ on the event $W_T \geq A$, it follows that $\Pro(W_T=0) \geq 8/9$. 
  Combining these,
  \[\Pro\left(T_0 \leq \frac{1}{c}\right) \geq \frac{7}{9}\]
  as claimed.

  Finally, if \(W_0<L\), repeat the preceding argument with \(U_t=W_t+L\),
stopping when \(U_t\le L\) or \(U_t\ge A\), and adding the transition rule
that whenever \(U_t\le L\), then \(U_{t+1}=0\). The same variance lower bound
holds before this stopping time, since then \(U_{t+1}-U_t=W_{t+1}-W_t\).
(At the final artificial step the variance bound is no longer needed.) 
Since $U_0<2L$, the smallness assumption on $U_0^2 \vee 4L^2$ is automatic.
\end{proof}

When applying Lemma \ref{lemma:path_coupling_with_expansion} for Corollary \ref{cor:path_coupling_with_expansion} below, $a$ is of order $1/n^2$, and $c$ and $L$ of order $1/n$, making the $4L^2$ part of the smallness condition on $U_0^2 \vee 4L^2$ automatic for large $n$.

\begin{corollary} \label{cor:path_coupling_with_expansion}
Let $\{D_t\}_{t=0}^\iy$ be a stochastic process on $[0,\kappa]$, let $T=\inf\{t:D_t=0\}$, and assume that there exist $\alpha,K \in (0,\iy)$ such that almost surely on the event $\{T > t\}$,
\begin{itemize}
    \item $\E[D_{t+1} \mid D_0^t] \leq (1+K/n)D_t$,
    \item $\E[(D_{t+1}-D_t)^2 \mid D_0^t] \geq \alpha/n^2$,
    \item $|D_{t+1}-D_t| \leq C/n$.
\end{itemize}
If $\kappa>0$ is sufficiently small, for any $\hat{K} \geq K$,
\[\Pro\left(T \leq \frac{n}{4\hat{K}}\right) \geq \frac{7}{9},\]
provided that 
\[D_0^2 \leq \frac{\alpha}{16000\hat{K}n}.\]
\end{corollary}

\begin{proof}
   First assume $\hat{K}=K$.
   Taking $W_t = (1-K/n)^tD_t$, gives 
   \[\E[W_{t+1} \mid W_0^t] \leq W_t\] 
   and 
   \[\E[(W_{t+1}-W_t)^2 \mid W_0^t] \geq \frac{\alpha(1-2K/n)^t}{n^2} - \frac{4K\kappa^2}{n^2}\geq \frac{\alpha e^{-4Kt/n}}{2n^2},\]
   for $\kappa$ sufficiently small and since $t \leq n/4K$.
   The assumption on $D_0$ ensures that $W_0$ satisfies the smallness assumption of Lemma \ref{lemma:path_coupling_with_expansion}.
   For the process \(W_t\), the increment bound is \(L=O(1/n)\), while
$a=\tfrac{\alpha}{2n^2}$ and $c=4K/n$, so the $4L^2$ part of the assumption
is satisfied for all sufficiently large $n$.
Applying the lemma proves the desired result.
   For $\hat{K}>K$, just observe that the conditions in the corollary statement hold with $K$ replaced by any larger number. 
\end{proof}


\section{The mixing time proofs} \label{sec:mixing_time_proof}

We recall the standing assumption that the attracting point $\yy_0$ is the origin.

\subsection{Couplings for the MFGOU} \label{subsec:couplings}

The mixing-time upper bound of Theorem \ref{thm:main_theorem} will be proved by first running two copies of $\cY$ independently until both have entered an $O(1/\sqrt{n})$ neighborhood of the attracting point, and then coupling them to achieve coalescence in an additional $O(n)$ steps.
This section proves, by construction, the existence of a coupling that provides the second phase and verifies the conditions
needed to apply Corollary \ref{cor:path_coupling_with_expansion}.
The coupling is a variant of path coupling \cite{DG}; see also the original path-coupling paper \cite{BubleyDyer}. 
Unlike the path-coupling argument of Bordewich and Dyer \cite{BD}, it permits a slight expected relative expansion in distance.
The construction combines two different mechanisms. For sufficiently nearby states, a short-range coupling gives a probability bounded away from zero of immediate coalescence, 
while controlling the expected change in distance and ensuring a sufficiently large second moment of that change. 
It will later be combined with an inert coupling for states that are farther apart.

Recall the following measure-theoretic notation. For a measure $\mu$, write $\tau_\hh\mu:=\mu(\cdot-\hh)$.
For probability measures $\mu$ and $\nu$ with densities $f_\mu$ and $f_\nu$ with respect to a given common base measure $\lambda$, the common submeasure is $(f_\mu \wedge f_\nu)d\lambda$.
Its total mass is $1-\n \mu-\nu\n_{TV}$. Whenever we refer to a random variable distributed according to the common part, we mean the normalized measure $(f_\mu \wedge f_\nu)d\lambda/(1-\n \mu-\nu\n_{TV})$.

\begin{proposition} \label{prop:short_distance_coupling} 
Let $\yy_0=\bz$ and $\kappa>0$ be as in the definition of an MFGOU family.
For a semicontinuous MFGOU family, there exists $\alpha>0$ and $\ep>0$ independent of $n$, such that whenever $\yy,\yy' \in B(\bz,\kappa)$ and $d:=\n \yy-\yy'\n_2 \leq \ep/n$, then there exists a coupling for which 
\begin{itemize}
\item[(a)] $\Pro(\bY_1=\bY'_1 \mid \bY_0=\yy,\bY'_0=\yy') \geq \alpha$,
\item[(b)] $\E[\n \bY_1-\bY'_1\n_2 \mid \bY_0=\yy,\bY'_0=\yy'] \leq d+Q/n^2$,
\item[(c)] $\E[(D-d)^2 \mid \bY_0=\yy,\bY'_0=\yy'] \geq \alpha d^2$,
\end{itemize}
where $D=\n\bY_1-\bY'_1\n_2$.

For a discrete MFGOU, the same result holds with the $\ell^2$-norm replaced by the $\ell^1$-norm and the condition $d \leq \ep/n$ replaced by $d=1/n$.
\end{proposition}

\smallskip

\begin{proof}
Consider first a continuous MFGOU process.
Let $I$ be a Bernoulli random variable with $\Pro(I=1)=b_0$.
If $I=1$, select the increments of both processes according to the fixed symmetric component $\mu_s$, and if $I=0$, select them according to their respective residual measures.
In the second case, couple $\mu^{(n)}_r(\yy)$ and $\mu^{(n)}_r(\yy')$ maximally; by the Lipschitz assumption, this coupling fails with probability at most $Q/n$, $Q$ uniform in $n$ and $\yy,\yy' \in B(\bz,\kappa)$.

If $I=1$, write $\bY_1 :=\yy+\bW/n$ and $\bY'_1:=\yy'+\bW'/n$, for random variables $\bW,\bW'$ with law $\mu_s$, whose dependence will be specified below.

Recall the standard result for absolutely continuous probability measures $\psi$, obtainable by approximating the density of $\psi$ in $L^1$ by a continuous density, that the map
    \[\hh \mapsto \tau_\hh\psi\]
    is continuous in total variation.
    Hence there exists $\ep>0$ such that 
    \[\n \tau_{\hh}\mu_s-\mu_s\n_{TV} < \frac{1}{3}\]
    whenever $\n \hh \n_2 \leq \ep$.
    Taking $\hh=n(\yy-\yy')$ then gives
    \[\left\n \cL\left(\yy+\frac{\bW}{n}\right) - \cL\left(\yy'+\frac{\bW'}{n}\right)\right\n_{TV} < \frac{1}{3}\]
    whenever $\n \yy-\yy'\n_2 \leq \ep/n$.

A random variable chosen according to the common part of the laws of $\bY_1$ and $\bY'_1$ can, 
by the symmetry of $\bW$ and $\bW'$, be written as $\Bar{\yy} + \bLambda/n$, where $\bLambda$ is symmetric.
Hence, writing $\psi_{\bLambda}$ for the density of $\bLambda$ and using the symmetry of $\bLambda$, 
one can write the density $\psi_\bW$ of $\bW$ as
\[\psi_\bW(\ww)=\frac{1-\delta}{2}\left( \psi_{\bLambda}(\ww-n(\yy-\Bar{\yy})) + \psi_{\bLambda}(\ww+n(\yy-\Bar{\yy}))\right) + \delta \psi(\ww),\]
where $\psi$ is the density of the remaining part, which, by the symmetry of $\bW$, is symmetric, and $\delta \leq 1/3$.
Equivalently, $\bW$ has the following mixture representation: 
\[
\bW=
\begin{cases}
n(\yy-\Bar{\yy})+\bLambda, \mbox{ with probability } (1-\delta)/2,\\
-n(\yy-\Bar{\yy})+\bLambda, \mbox{ with probability } (1-\delta)/2, \\ 
\bW^r, \mbox{ with probability } \delta,
\end{cases}
\]
where $\bW^r \sim \psi$.
Since $\bW'$ is a copy of $\bW$, this representation holds for $\bW'$ too.
This allows us to define a coupling $(\bX,\bX')$ of $(\bW,\bW')$ by
\[
\begin{cases}
    \bX=n(\yy-\Bar{\yy})+\bLambda \mbox{ and } \bX'=-n(\yy-\Bar{\yy})+\bLambda, \mbox{ with probability } (1-\delta)/2, \\
    \bX=-n(\yy-\Bar{\yy})+\bLambda \mbox{ and } \bX'=n(\yy-\Bar{\yy})+\bLambda, \mbox{ with probability } (1-\delta)/2, \\
    \bX=\bX'=\bW^r, \mbox{ with probability } \delta.
\end{cases}
\]
Then coalescence of $\cY$ and $\cY'$ occurs with probability $(1-\delta)/2$ and the distance between them is doubled, with probability $(1-\delta)/2$.
With probability $\delta$, the distance between them remains unchanged.

Combining the pieces, we obtain a coupling of the next steps of $\cY$ and $\cY'$ such that conditional on $I=0$, its maximal coupling of the residual components fails with probability at most $Q/n$.
On this failure event, the distance increases by at most $Q/n$; otherwise it remains equal to $d$.
Conditional on $I=1$, the expected distance between $\cY$ and $\cY'$ equals $d$,
and the probability of coalescence of $\cY$ and $\cY'$ equals $(1-\delta)/2$, which is at least $1/3$. 
Since $\Pro(I=1)=b_0>0$,
this proves (a) and (b) for the continuous case upon taking $\alpha>0$ small enough. 
Part (c) is an immediate consequence of (a).

For a semicontinuous MFGOU, first maximally couple the indicators deciding between the atom at zero and the absolutely continuous component. 
The indicators disagree with probability $O(\|\yy-\yy'\|)=O(1/n)$ uniformly in $n$. 
If both select the atom, use the identical zero increments; if both select the continuous component, apply the preceding continuous coupling. 
Since the probability that both select the continuous component is bounded away from zero, the coalescence and quadratic-variation bounds are still valid after decreasing $\alpha$. 
On the $O(1/n)$-probability event that the indicators disagree, 
boundedness of the increments gives an $O(1/n^2)$ contribution to the expected distance. 
This proves the semicontinuous case.  

\smallskip

The discrete case is similar but simpler.
The distribution of $\bxi(\yy)$ can by assumption be written as a mixture of the symmetric distribution on 
$\{\bz,\pm\ee_1,\ldots,\pm\ee_m\}$ that takes $\bz$ with probability $1/2$ and is conditionally uniform on the other outcomes if $\bz$ is not chosen, and a remainder. 
For neighboring lattice points $\yy$ and $\yy'$, 
by the Lipschitz condition on the laws $\cL(\xi^{(n)}(\yy))$'s, the remainders can be coupled to be identical with probability at least $1-Q/n$ uniformly in $n$, and the same is true of the mixing indicators.
Given that $\bxi(\yy)$ and $\bxi(\yy')$ are simultaneously chosen according to their identical symmetric parts, couple them as follows:
\begin{itemize}
\item Let $j$ be the coordinate in which $\bY_0$ and $\bY'_0$ differ,
\item choose $I$ uniformly at random from $[m]$,
\item given $I=i$, if $i \neq j$, let $\cY$ and $\cY'$ both stand still with probability $1/2$, both move up in direction $i$ with probability $1/4$, and both move down in direction $i$ with probability $1/4$,
\item given $I=j$, couple the increments so that $(\bxi,\bxi')$ takes on each of the elements of
\[\{(\bz,-\ee_j),(\bz,\ee_j),(-\ee_j,\bz),(\ee_j,\bz)\}\] 
with probability $1/4$. 
\end{itemize}
Conditional on $I=j$, the chains coalesce with probability $1/2$, while otherwise their distance doubles; 
conditional on $I\neq j$, their distance remains $1/n$. 
The symmetric component therefore preserves the expected distance and gives a probability bounded away from zero of coalescence, 
while failure to couple the mixing indicators or residual components contributes only $O(1/n^2)$ to the expected distance.
Thus (a)--(c) follow for some $\alpha>0$, independent of $n$.
\end{proof}

A coupling that satisfies the estimates in the statement of Proposition \ref{prop:short_distance_coupling} will be referred to as a {\em short-range coupling}. We will combine that with an {\em inert coupling} to obtain the existence of an {\em admissible path coupling}.

\smallskip

{\bf Inert coupling:} 
Set, as before, $d=\n \yy-\yy'\n_2$ and $D=\n \bY_1-\bY'_1\n_2$. 
An inert coupling is a coupling that satisfies
    \[\Pro(D \neq d|\bY_0=\yy,\bY'_0=\yy') \leq Qd\]
for $Q$ uniform in $n$ and $\yy,\yy' \in B(\bz,\kappa)$ for some $\kappa>0$.

Inert couplings exist, since the maximal coupling of $\bxi(\yy)$ and $\bxi(\yy')$ results in an inert coupling. 
This follows from the Lipschitz conditions: the probability that the maximal coupling fails is bounded by $Qd$ with $Q$ uniform in $n$.
Note that if an inert coupling fails, then, by the MFGOU framework, the distance between $\cY$ and $\cY'$ changes by at most $Q/n$. Hence
\[\E[D|\bY_0=\yy,\bY'=\yy'] \leq \left(1+\frac{Q}{n}\right)d,\]
for $Q$ uniform in $n$ and $\yy,\yy' \in B(\bz,\kappa)$.

\smallskip

A composite path‑coupling is a path‑coupling construction obtained by combining an inert coupling with a short‑range coupling.

In what follows, the number $\kappa$ is consistently chosen sufficiently small that short-range and inert-coupling properties used are satisfied on $B(\bz,\kappa)$.

{\bf Composite path coupling:} 
Consider one short-range coupling and one inert coupling, $\yy,\yy' \in B(\bz,\kappa)$ and $\ep$ as in the statement of Proposition \ref{prop:short_distance_coupling}.

For the semicontinuous case,
let $d=\n \yy-\yy'\n_2$, assume $d>\ep/n$, and let $\yy^0$ be the point on the line segment $[\yy,\yy']$ at distance $\ep/n$ from $\yy$.
Let $\cY$, $\cY^0$ and $\cY'$ be three copies of the MFGOU with 
\[\bY_0=\yy, \, \bY^0_0=\yy^0,\, \bY'_0=\yy'.\]
Couple $\bY_1$ and $\bY^0_1$ according to the short-range coupling, and couple $\bY^0_1$ and $\bY'_1$ according to the inert coupling. By the gluing lemma, these two pairwise couplings can be realized as a joint coupling of $(\bY_1,\bY^0_1,\bY'_1)$.

For a discrete MFGOU, assume that $d:=\n \yy-\yy'\n_1 \geq 2/n$ and select $\yy^0$ to be a point at $\ell^1$-distance $1/n$ from $\yy$ on a shortest path between $\yy$ and $\yy'$ and then construct the coupling as in the semicontinuous case.

\begin{proposition} \label{prop:long_range_coupling_bounds}
    Consider a composite path coupling as above, and set $D=\n \bY_1-\bY'_1\n$, with the norm being $\ell^2$ in the semicontinuous case and $\ell^1$ in the discrete case.
    Then, for $\yy,\yy' \in B(\bz,\kappa)$ with $d:=\n \yy-\yy'\n \geq \ep/n$ in the semicontinuous case and $d \geq 1/n$ in the discrete case,
    \[\E[D \mid \bY_0=\yy,\bY'_0=\yy'] \leq \left(1+\frac{Q}{n}\right)d,\]
    and for $\chi>0$ sufficiently small
    \[\E[(D-d)^2 \mid \bY_0=\yy,\bY'_0=\yy'] \geq \frac{\chi}{n^2}.\]
\end{proposition}

\begin{proof}
    Consider first the semicontinuous case.
    By the triangle inequality, writing $A:=\{\bY_0=\yy,\bY'_0=\yy'\}$,
    \[\E[D \mid A] \leq \E[\n \bY_1-\bY^0_1\n_2 \mid A] + \E[\n \bY^0_1-\bY'_1\n_2 \mid A].\]
    By Proposition \ref{prop:short_distance_coupling}, the first term on the right is bounded by $\ep/n+Q/n^2$. The second term is bounded by $(d-\ep/n)+Qd/n$ by the inert-coupling property.
    Summing gives an upper bound of $d+Qd/n$, proving the first part.
    For the second part, observe that the probability that the inert part of the coupling succeeds and the short-range part results in coalescence of $\cY$ and $\cY^0$ is at least $\alpha-Qd \geq \alpha/2$ for $\kappa$ sufficiently small and $\alpha$ as in Proposition \ref{prop:short_distance_coupling}.
    If that happens, then the distance between $\cY$ and $\cY'$ changes by at least $\tfrac{\ep}{2n}$ for $n$ large after reducing the constant if necessary, resulting in 
    \[\E[(D-d)^2 \mid A] \geq \frac{\alpha\ep^2}{8n^2}.\]
    In the discrete case, the argument is identical on replacing $\ep$ by $1$; when $d=1/n$, the conclusion follows from the short-range-coupling properties.
\end{proof}

A coupling that satisfies the estimates in the statement of Proposition \ref{prop:long_range_coupling_bounds} will be called an {\em admissible path coupling}. 
Proposition \ref{prop:long_range_coupling_bounds} shows that admissible path couplings exist.

\subsection{Local hitting estimates near the origin} \label{subsec:hitting_probability_of_inner_box}

The following lemma gives bounds for important hitting probabilities for an MFGOU process when it is within distance $O(1/\sqrt{n})$ of the origin.
We use $B_x$ to denote the box $[-x,x]^m$.

\begin{lemma} \label{lemma:hitting_inner_box}
Let $\beta<K$ be constants and consider a discrete or semicontinuous MFGOU $\cY^{(n)}$, started from a point $\aaa \in B_{K/\sqrt{n}}$.
Let $\tau$ be the first time $\cY^{(n)}$ either hits $B_{\beta/\sqrt{n}}$ or exits $B_{2K/\sqrt{n}}$ and let $A_{\beta} = \{\bY_\tau \in B_{\beta/\sqrt{n}}\}$. 
Then for all $c>0$ there is a $p=p(\beta,K,c)>0$ independent of $n$ such that
\[\Pro(A_\beta \cap \{\tau \leq cn\})>p.\]
Moreover, the identical result holds for $(\cY,\cY')$, where these are independent copies of $\cY^{(n)}$, with $B_x$ interpreted as $[-x,x]^{2m}$.
\end{lemma}

\begin{proof}
    Since we are only interested in the evolution of $\cY$ at most until it leaves $B_{2K/\sqrt{n}}$, we may regard $\cY^{\left(n\right)}$ as restricted to that set.
    Write 
    \[\bY^{(n)}_t=\aaa+\frac{1}{n}\sum_{s=0}^t \bxi^{(n)}_s,\]
    where $\bxi^{(n)}_s$ is a realization of $\bxi^{(n)}(\bY^{(n)}_{s-1})$. 
    Next write, trivially, $\bxi^{(n)}(\yy)=\bX^{(n)}+(\bxi^{(n)}(\yy)-\bX^{(n)})$ for a copy $\bX^{(n)}$ of $\bxi^{(n)}(0)$.
    Letting $\bX^{(n)}$ be maximally coupled with $\bxi^{(n)}(\yy)$, we get
    \[\bxi^{(n)}(\yy)=\bX^{(n)} + I(\yy)(\bxi^{(n)}(\yy)-\bX^{(n)}),\]
    where $I(\yy)$ is the indicator that $\bX^{(n)} \neq \bxi^{(n)}(\yy)$.
    By the Lipschitz assumptions on the increment laws, 
    \[\Pro(I(\yy)=1) \leq Q/\sqrt{n}\] 
    uniformly in $n$ and $\yy$.
    
    Now write each $\bxi^{(n)}_t$ in the given form: $\bxi^{(n)}_s=\bX^{(n)}_s+I(\bY^{(n)}_{s-1})(\bxi^{(n)}_s-\bX^{(n)}_s)=:\bX^{(n)}_s+I(\bY_{s-1})\bV^{(n)}_s$.
    
    By Hoeffding/Bernstein bounds, the number of $I_s(\bY_{s-1})$ that are $1$ is whp bounded by $Qt/\sqrt{n}$ for all $t=1,\ldots,cn$, uniformly in $n$. 
    Since each $\bV^{(n)}_s$ is bounded, 
    \[\frac{1}{n}\left\n\sum_{s=1}^t I(\bY_{s-1})\bV_s\right\n \leq \frac{Qt}{n^{3/2}}\]
    for all $t$ whp uniformly in $n$ (where the norm is $\ell^1$ in the discrete case and $\ell^2$ in the semicontinuous case).

    Now consider the part $\bU^{(n)}_t:=\aaa+(1/n)\sum_{s=1}^t \bX^{(n)}_s$. Rescale, writing $u=t/n$, $\bb=\sqrt{n}\aaa \in B_K$, $\bW^{(n)}(u) = \sqrt{n}\bU^{(n)}_{\lfloor nu \rfloor}$, and consider
    \[\bW^{(n)}(u) = \bb + \frac{1}{\sqrt{n}}\sum_{s=1}^{\lfloor nu \rfloor}\bX^{(n)}_s.\]
    Since $\E[\bX_s^{(n)}]=O(1/n)$, centering the increments changes $\bW^{(n)}$ by $O(1/\sqrt{n})$ uniformly on bounded time intervals. 
    The increments are uniformly bounded, 
    and their covariance matrices are therefore uniformly bounded above and, by the MFGOU assumptions, uniformly bounded away from singularity. 
    Hence, by the triangular-array version of Donsker’s theorem and a subsequence argument, every subsequential limit of $\bW^{(n)}$ is a nondegenerate Brownian motion, started from a point in $B_K$, 
    whose covariance matrix belongs to a compact family of positive-definite matrices. 
    It follows by compactness that the probability of hitting $B_{\beta/2}$ before leaving $B_{2K-\beta/2}$ and before time $c$ is bounded below by a positive constant uniformly in the starting point and in $n$.
    
    Scaling back, $\bU^{(n)}_t$ hits $B_{\beta/(2\sqrt{n})}$ before $B_{(2K-\beta/2)/\sqrt{n}}$ and before time $cn$ with probability bounded below by a positive constant uniformly in the starting point and in $n$.
    Taking $c$ sufficiently small ensures that $Qt/n^{3/2} \leq \beta/2\sqrt{n}$ for all $t=1,\ldots,cn$. 
    This proves the lemma for $c$ sufficiently small and since the left-hand side of the statement is increasing in $c$, it holds for all $c$.

    \smallskip

    The argument for $(\cY,\cY')$ is identical up to trivial modifications.

\end{proof}

\subsection{Contraction to a small ball in linear time or entrapment in local modes} \label{subsec:to_small_box_in_linear_time}
We next show that under the Lyapunov condition of Theorem \ref{thm:main_theorem}, an MFGOU process contracts into a small neighborhood of the minimizer in linear time and remains there for an exponentially long period with high probability. 
This will allow us to apply the censoring lemmas and treat the chain as confined to a fixed small ball.
Specifically, we claim that for any $\kappa>0$ and a given basic MFGOU $\cY$, there is a $\kappa_1>0$ such that $\cY$ reaches $B(\bz,\kappa_1)$ in time $O(nV(\bY_0))$ and only with exponentially small probability leaves $B(\bz,\kappa/2)$ before an exponential time has passed. 
In the end, when upper bounding mixing time, 
this means that we may think of $\cY$ as living entirely on $B(\bz,\kappa)$ at the cost of adding a time penalty of $O(nV(\bY_0))$.

\smallskip

At the end of this section, we also prove Theorem \ref{thm:exponential_mixing_time}; this will essentially
come for free.

\smallskip

{\bf Lyapunov-preserving censoring.}
Before proving the contraction result, we clarify what is meant by
“treating the chain as confined to a small sublevel set.”  Since the
censoring mechanism in Section \ref{subsec:censoring} may be chosen freely, we now specify
a version that preserves the Lyapunov condition of Theorem \ref{thm:main_theorem}.

Assume that the Lyapunov condition of Theorem~1.1 holds on the Lyapunov 
sublevel set
\[
L_{K_2} := \{y \in \cC : V(y) \le K_2\},
\]
and let $0 < K_0 < K_1 < K_2$ with $K_2$ sufficiently small that $L_{K_2}$ is compactly contained in $\cC^\circ$. 
When we say that the process $\cY$ can be treated as confined to $L_{K_2}$ after 
having reached $L_{K_0}$, we mean the following construction.

Define a censored chain $\cY^{(r)}$ as follows.    
Let $a:[0,K_2] \ra [1,M]$ be a $C^2$ function with $a(x)=1$ for $x \in [0,K_1]$ and $a(x)=M$ in a neighborhood $[K_2-\delta,K_2]$ of $K_2$.
Since, by the Lyapunov condition, $\nabla V^\top \ff >0 $ is bounded away from $0$ on any compact domain that does not contain $\bz$, we may pick a constant $M$, such that 
\[\frac{Mb}{2} \geq 2\sqrt{m}B,\]
where
\[B=\max_{\yy \in L_{K_2} \setminus L_{K_1}}\n \nabla V(\yy)\n,\,\, b=\min_{\yy \in L_{K_2} \setminus L_{K_1}} \nabla V(\yy)^\top \ff(\yy).\]
On $L_{K_2}$, define
\[\tilde{\ff}(\yy) := a(V(\yy))\ff(\yy).\]  
For $\yy \in L_{K_1}$, let the censored chain make the identical move as $\cY$.
For all sufficiently large $n$, such a transition remains in $L_{K_2}$, since the increments of $\cY$ are uniformly $O(1/n)$ and $K_2-K_1>0$.

From a state $\yy$ in the annulus $L_{K_2} \setminus L_{K_1}$, let the censored chain perform a step such that
\begin{equation} \label{eq:annulus_drift}
\E[\bY^{(r)}_{t+1} \mid \bY^{(r)}_t=\yy] = \zz:= \yy - \frac{1}{2n} \tilde{\ff}(\yy)
\end{equation}
in the following way.
First consider $\yy$ in the compact boundary neighborhood $\{\yy: K_2-\delta\le V(\yy) \le K_2\}$. 
Taylor expansion gives, for $n$ sufficiently large, and any $\xx$ such that $\n \xx \n \leq \sqrt{m}$,
\begin{align*} 
V\left(\yy-\frac{1}{2n}\tilde{\ff}(\yy) + \frac{1}{n}\xx \right) &= V\left(\yy-\frac{M}{2n}\ff(\yy) + \frac{1}{n}\xx\right) \\
&\leq V(\yy)-\frac{M}{2n} \nabla V(\yy)^\top \ff(\yy) + \frac{1}{n}\n \nabla V(\yy) \n \n \xx \n +O\left(\frac{1}{n^2}\right) \\
&\leq V(\yy) - \frac{\sqrt{m}B}{n} + O\left(\frac{1}{n^2}\right) < V(\yy). 
\end{align*}
This proves that any box of side length $1/n$ that contains $\zz$ is contained in $L_{K_2}$, in particular the lattice box whose corners are lattice points.
For points outside the boundary neighborhood, the same conclusion follows for sufficiently large $n$, since they range over a compact set separated from the boundary of $L_{K_2}$, while all constructed steps have size $O(1/n)$.

Now, in the discrete case, let $\bY^{(r)}_{t+1}$ have support on the corners of the lattice box that contains $\zz$, with probabilities that satisfy (\ref{eq:annulus_drift}).
In the semicontinuous case, let $\bY^{(r)}_{t+1}$ have uniform distribution on the ball $B(\zz,1/n)$.

In summary, the censored chain cannot leave $L_{K_2}$ and agrees with $\cY=\cY^{(n)}$ on $L_{K_1}$. 
Applying this construction for each $n$, the censored processes form an MFGOU family: the drift, increment, and local regularity conditions hold uniformly in $n$, while Remark 1.2 gives each process a unique stationary probability distribution.

Thus, for the purpose of mixing-time estimation, once $\cY$ has entered $L_{K_0}$, we may replace it with this 
Lyapunov-preserving censored chain $\cY^{(r)}$ for some choice of 
$K_0 < K_1 < K_2$.

\begin{remark}
{\rm The nondegenerate minimum of $V$ at the origin implies that $V$ is locally comparable to a positive-definite quadratic form, while compactness of its sublevel sets and uniqueness of its zero exclude distant points with arbitrarily small values of $V$. 
Hence each sufficiently small sublevel set contains a ball centered at the origin and is contained in a slightly larger such ball. We therefore freely speak of “small balls” or “small sublevel sets” near the origin.
}
\end{remark}

 \begin{lemma} \label{lemma:to_small_box_in_linear_time}
 Assume that $V$ satisfies the Lyapunov-function hypotheses of Theorem \ref{thm:main_theorem}.
Let $\cY$ be a basic MFGOU process and let $\tau$ be its hitting time of $B(\bz,\eta)$.
For each $\rho>0$ there is a $K=K(\rho,\eta)<\iy$ such that if $\bY_0 \not \in B(\bz,\eta)$
 \[\Pro\left(\tau>K V(\bY_0) n\right)<\rho.\]
 Moreover, there exists $C<\iy$ such that for all sufficiently small $\eta>0$, there exists a $\beta=\beta(\eta)>0$ such that if $\bY_0 \in B(\bz,\eta)$, then
 \[\Pro\left(\exists t \leq e^{\beta n}:\bY_t \not \in B(\bz,C\eta)\right) \leq e^{-\beta n}.\]
     
 \end{lemma}

 Taking $\eta=\kappa_1$, with $\kappa_1\leq\tfrac{\kappa}{2C}$, shows that after reaching $B(\bz,\kappa_1)$, 
 the process remains in $B(\bz,\kappa/2)$ for an exponentially long time with high probability, as claimed at the beginning of the section.
 By Lemmas \ref{lemma:start_in_subset} and \ref{lemma:mixing_on_fast_hit_restricted_set}, we may, and henceforth will, treat $\cY$ as started in and restricted to $B(\bz,\kappa)$, with only the cost of adding a term $O(nV(\bY_0))$ to the final mixing time upper bound.

 \smallskip

 \begin{proof}
     By the MFGOU definition, we may write 
     \[\bY_{t+1}=\bY_t-\frac{1}{2n}\ff(\bY_t)+\frac{1}{n}\bS_t+O\left(\frac{1}{n^2}\right),\] 
     where $\bS_t$ has conditional mean $\bz$ given $\bY_t$ and is uniformly bounded. 
     By Taylor expansion and the MFGOU transition rule, for some $Q<\infty$, 
     \[ V(\bY_{t+1}) \leq V(\bY_t) -\frac{1}{2n}\nabla V(\bY_t)^\top \ff(\bY_t) +\frac{1}{n}\nabla V(\bY_t)^\top\bS_t +\frac{Q}{n^2}. \]
     Choosing an $\alpha \leq c/4$ and taking conditional expectation given $\bY_t$ gives
     \[\E[V(\bY_{t+1}) \mid \bY_t] \leq V(\bY_t)-\frac{2\alpha}{n}+
     \frac{Q}{n^2} \leq V(\bY_t)-\frac{\alpha}{n}\]
     whenever $n$ is sufficiently large and $V(\bY_t) \geq 4\alpha/c$.
     By Lemma 2.1 and the censoring construction above, we may restrict attention to a fixed compact Lyapunov sublevel set containing $B(\bz,\eta)$.
     Now fix $\alpha$ sufficiently small that $A=A_\alpha=\{\yy:V(\yy) \leq 4\alpha\}$ is contained in $B(\bz,\eta)$ 
     and let $\tau_A$ be the hitting time of $A$. Then $\tau_A \geq \tau$, so that an upper bound for $\Pro(\tau_A > T)$ is an upper bound also for $\Pro(\tau>T)$.  
     
     Since we are interested in the time taken to reach $A$, we may regard $\bY_t$ as absorbed when hitting $A$.
     Then the conditional expectation upper bound may be written as
     \[\E[V(\bY_{t+1}) \mid \bY_t] \leq V(\bY_t)-\frac{\alpha}{n}I_{\{\tau_A > t\}}\]
     which, upon taking expectation, gives
     \[\E[V(\bY_{t+1})] \leq \E[V(\bY_t)]-\frac{\alpha}{n}\Pro(\tau_A > t).\]
     If $T$ is such that $\Pro(\tau_A > T) \geq \rho$, this entails that
     \[0 \leq \E[V(\bY_T)] \leq V(\bY_0) - \frac{\rho\alpha T}{n},\]
     from which it follows that 
     \[T \leq \frac{V(\bY_0)}{\rho\alpha}n,\]
     as desired. 
    
     The second assertion follows from Lemma \ref{lemma:no_escape_with_negative drift} applied to the process $\{nV(\bY_t)\}$.
     
\end{proof}

We now turn to the metastable case Theorem \ref{thm:exponential_mixing_time}, whose proof uses only Lemma \ref{lemma:to_small_box_in_linear_time}.

{\em Proof of Theorem \ref{thm:exponential_mixing_time}.}   
For each $i =0,1$, if $\cY$ is started from $\yy_i$, then by Lemma \ref{lemma:no_escape_with_negative drift} applied to 
$\{nV_i(\bY_t)\}$ locally around $\yy_i$, the time to move distance $\kappa_i$ from that $\yy_i$ is whp exponential.
Taking the two neighborhoods $B(\yy_i,2\kappa_i)$ to be disjoint, at least one of them must contain no more than half of the probability mass at stationarity. 
Starting $\cY$ from the corresponding point $\yy_i$ therefore gives that $\taum(\cY;\ep)$ is exponential for any $\ep < 1/2$.
\hfill $\Box$

\subsection{A variance bound on MFGOU states} \label{subsec:variance_bound}

It will be of crucial importance to monitor the speed of contraction once a small ball has been hit. To that end, it is vital to control the variance of $\bY_t$. The following lemma states this; the variance of $\bY_t$ is $O(1/n)$ uniformly in $t$.

\begin{lemma} \label{lemma:variance_bound}
    Let $\cY$ be a basic MFGOU process restricted to $B(\bz,\kappa)$ for some $\kappa>0$ sufficiently small, started from a fixed state $\bY_0$. Then there is a finite constant $C$ such that for all $t$ and with $\bmu_t:=\E[\bY_t]$,
    \[\Var(\bY_t) := \E\left[\n\bY_t-\bmu_t\n_2^2\right] \leq \frac{C}{n}.\]
\end{lemma}

\begin{proof} 
Write 
$M=I-\tfrac{\cJ}{2n}$ so that 
\[\bY_{t+1}= M\bY_t - \frac{1}{2n}\mathbf{R}(\bY_t) + \frac{1}{n} \bS_t + O\left(\frac{1}{n^2}\right).\] 
 
Recall that $\Var(\bY_t)$ can be expressed as
\[\Var(\bY_t) = \frac12 \E\left[\n \bZ_t-\bZ'_t \n^2\right],\]
where $\bZ_t$ and $\bZ'_t$ are two independent copies of $\bY_t$. 
Let $\zz_t := \E[\n \bZ_t- \bZ'_t \n^2]$.
We have
\[\E[\n\bZ_{t+1}-\bZ'_{t+1}\n^2 \mid \bZ_t,\bZ'_t] = \left\n M(\bZ_t-\bZ'_t)-\frac{1}{2n}\left(\mathbf{R}(\bZ_t)-\mathbf{R}(\bZ'_t)\right)\right\n^2 + O\left(\frac{1}{n^2}\right),\]
where the $O(1/n^2)$ term is the sum of the $O(1/n^2)$ term in the definition of the MFGOU and the contribution to the expectation of $(1/n)(\bS_t-\bS'_t)$, 
taking into account that $\E[\bS_t \mid \bZ_t]=\E[\bS'_t \mid \bZ'_t]=\bz$.
The expectation of the right-hand side is
\[\E[\n M(\bZ_t-\bZ'_t) \n^2] -\frac{1}{n} \E[(\bZ_t-\bZ'_t)^\top M (\mathbf{R}(\bZ_t)-\mathbf{R}(\bZ'_t))] +\frac{Q}{n^2}.\]
 By the multivariate Mean-value Theorem applied to each component of $\mathbf{R}$ and using that $\mathbf{R}$ is $C^1$ (since $\ff$ is $C^2$) and $\mathbf{R}(\yy)=O(\n\yy\n^2)$,
\[\n \mathbf{R}(\bZ_t)-\mathbf{R}(\bZ'_t)\n^2 \leq c^2\n \bZ_t - \bZ'_t\n^2\]
for 
\[c^2=\max_{\yy \in B(\bz,\kappa),i \in \{1,\ldots,m\}}\n \nabla R_i(\yy) \n^2.\]
By Schwarz inequality and since $\n M\vv \n \leq \n \vv \n$ for any $\vv$, the second term of the expectation is thus bounded by $\tfrac{c}{n}\E[\n \bZ_t-\bZ'_t \n^2]$.
The first term is bounded by 
\[\left(1-\frac{\lambda_1}{n}\right)\E[\n \bZ_t-\bZ'_t \n^2] + O\left(\frac{1}{n^2}\right).\]
By taking $\kappa$ sufficiently small, $c<\lambda_1/2$, and adding all up, we get
\[\zz_{t+1} \leq \left(1-\frac{\lambda_1}{2n}\right)\zz_t+\frac{Q}{n^2}.\]
Since $\zz_0=0$, it follows that $\zz_t \leq Q/n$ for all $t$.

\end{proof}

\subsection{Tracking the mean} \label{subsec:tracking_the_mean}

Lemma \ref{lemma:variance_bound} is our key result in two ways; (a) it establishes that at any given time, $\bY_t$ is whp within distance $O(1/\sqrt{n})$ of its mean and (b) allows for tracking the evolution of the mean.
We will use these facts to show that virtually all the mass of the stationary distribution is located within distance $O(1/\sqrt{n})$ of the origin. Entering this region is therefore a necessary condition for mixing; 
conversely, once it has been reached, mixing is near, as will be established by the coupling argument below.

Let us now analyze how $\bmu_t:=\E[\bY_t]$ evolves with time.
We have
\[\bmu_{t+1}=M\bmu_t - \frac{1}{2n}\E[\mathbf{R}(\bY_t)] + O\left(\frac{1}{n^2}\right).\]
For the second term, observe that 
\[\n\E[\mathbf{R}(\bY_t)]\n \leq \E[\n \mathbf{R}(\bY_t)\n] \leq Q \E[\n \bY_t\n^2] \leq Q \left(\n\bmu_t\n^2+\frac{1}{n}\right),\]
where the last inequality is by Lemma \ref{lemma:variance_bound}.
Inserting into the recursion, it follows that for sufficiently large constant $C<\iy$,
\[\n \bmu_{t+1}-M\bmu_t \n \leq \frac{C\n M \n \n \bmu_t \n^2}{n}+\frac{C}{n^2}. \]
Here the factor $\n M \n$ in the first term on the right-hand side is inserted for convenience.

Recall the definition of the usual matrix norm: 
\[\n A \n := \max_\uu \n A\uu \n,\] 
where the maximum is taken over unit vectors $\uu$. 
If $A$ is symmetric and positive definite, the maximum is reached for the eigenvector for the largest eigenvalue, from which it follows that
\[\left\n \sum_{s=0}^k A^s \right\n = \sum_{s=0}^k \n A^s \n = \sum_{s=0}^k \n A\n^s\] 
for any $k \geq 0$.

The following lemma tells us that, provided that $\kappa$ is sufficiently small, $\{\bmu_t\}$ behaves essentially as if the recursion had simply read $\bmu_{t+1}=M\bmu_t$; the small wiggle room does not cause it to stray far from the main contractive path.
For convenience we compare $\xx_{t+1}$ to $A^2\xx_t$ (rather than $A\xx_t$).

\begin{lemma} \label{lemma:drop_the_squre_in_the_recursion}
Assume that $E$ is an $m \times m$ symmetric positive definite matrix and let $A=I-E/n$. 
Let $S_t=\sum_{s=0}^{t-1}A^s$ and write $e_1$ for the smallest eigenvalue of $E$.
Assume that $\{\xx_t\}_{t=0}^\iy$ is a sequence in $\RR^m$ that recursively satisfies
\[\n\xx_{t+1}-A^2\xx_t\n \leq C\left(\frac{\n A\n^2 \n\xx_t\n^2}{n} + \frac{1}{n^{2}}\right).\]
Then there exist $N,C_1<\iy$ such that  if $\n \xx_0 \n \leq \tfrac{e_1}{4C}$ and $n \geq N$, then
\[\n \xx_t-A^{2t}\xx_0\n \leq C_1\n S_t\n\left(\frac{\n A \n^{2t} \n\xx_0\n^2}{n} + \frac{1}{n^{2}}\right)\]
for all $t \geq 1$.

\smallskip

In particular, for $C_2$ sufficiently large,
\[\n\xx_t-A^{2t}\xx_0\n \leq C_2\left(\n A\n^{2t} \n\xx_0\n^2 + \frac{1}{n}\right).\]
\end{lemma}

\begin{proof}
Fix $C'<\iy$. 
If $\|\xx_0\|<C'/n$, then the assumed recursive inequality and a simple induction give
\[\|\xx_t\|\le \frac{C'}{n}\]
for all $t$, provided $C'$ is sufficiently large in comparison with $C$ and $n$ is sufficiently large. 
Applying the triangle inequality and then the recursive inequality to each term gives, for sufficiently large $n$,
\begin{align*}
    \n \xx_t - A^{2t}\xx_0\n & \leq \sum_{s=0}^{t-1}\n A^{2s}\n \n \xx_{t-s}-A^2\xx_{t-s-1} \n \\
    & \leq \sum_{s=0}^{t-1} \n A\n^{2s} \frac{2C}{n^2} \leq C_1 \frac{\n S_t\n}{n^2},
\end{align*}
assuming wlog that $C_1 \geq 2C$.
This establishes the assertion of the lemma for $\n \xx_0 \n \leq C'/n$.

    The proof for $\n \xx_0 \n > C'/n$ goes by induction, starting with the observation that the case $t=1$ is immediate.
    To this end, observe first that, by the recursive inequality and the triangle inequality,
    \begin{equation} \label{eq:bound_on_norm_of_x1}
    \n \xx_1 \n \leq \n A \n^2 \n \xx_0 \n \left(1+\frac{C \n \xx_0\n}{n}\right) + \frac{C}{n^2}.
    \end{equation}
    Since the right-hand side is increasing in $\n \xx_0\n$, substituting with $\n \xx_0\n=\tfrac{e_1}{4C}$ shows that $\n \xx_1\n \leq \tfrac{e_1}{4C}$.
    Thus we may apply the induction hypothesis to the shifted sequence $\{\xx_{t+1}\}_{t=0}^\iy$,
    and doing so, it follows that
    \begin{align*}
        \n \xx_{t+1}-A^{2(t+1)}\xx_0 \n &\leq \n \xx_{t+1}-A^{2t}\xx_1 \n + \n A\n^{2t} \n\xx_1-A^2\xx_0\n \\
        &\leq C_1\n S_t\n \left(\frac{\n A \n^{2t} \n\xx_1\n^2}{n}+ \frac{1}{n^2}\right) \\
        &+ C\n A\n^{2t} \left(\frac{\n A\n^{2}\n\xx_0\n^2}{n} + \frac{1}{n^2}\right).
    \end{align*} 
Since $C \leq C_1$, the sum of the second and fourth terms is then 
\[\frac{C\n A\n^{2t}+ C_1\sum_{s=0}^{t-1}\n A\n^s}{n^2} <  \frac{C_1\n S_{t+1} \n}{n^2}.\] 
For the other two terms, observe further that
\[\n A \n^{1/2}\left(1+\frac{C\n\xx_0\n}{n}\right) \leq \left(1-\frac{e_1}{2n}\right)\left(1+\frac{C\n\xx_0\n}{n}\right) < 1-\frac{e_1}{4n}\]
for $\n\xx_0\n \leq \tfrac{e_1}{4C}$, which on insertion in (\ref{eq:bound_on_norm_of_x1}) gives
\[\n \xx_1 \n \leq \left(1-\frac{e_1}{4n}\right) \n A \n^{3/2}\n\xx_0\n + \frac{C}{n^2}.\]
Since $\n\xx_0\n \geq C'/n$, the right-hand side is bounded by $\n A \n^{3/2}\n \xx_0\n$, provided that $C'$ is sufficiently large (if not, increase it sufficiently and increase $C_1$ accordingly if necessary). 
The first term is therefore bounded by
\[\frac{C_1\n A \n^{2(t+1)} \n AS_t\n \n \xx_0 \n^2}{n} = \frac{C_1 \n A \n^{2(t+1)} \n S_{t+1}-I \n \n \xx_0 \n^2}{n}.\]
Since $C \leq C_1$, this in turn bounds the sum of the first and third terms by 
\[C_1\n A \n^{2(t+1)} \n S_{t+1}\n \n \xx_0 \n^2/n.\]

Combining these bounds gives
\[\n \xx_{t+1}-A^{2(t+1)}\xx_0 \n \leq  \frac{C_1\n A \n^{2(t+1)} \n\xx_0\n^2 \n S_{t+1}\n}{n} + \frac{C_1\n S_{t+1} \n}{n^2}\]
as desired.

Finally, since
\[\n S_t \n \leq \sum_{s=0}^\iy \n A \n^{s} = \frac{n}{e_1},\]
the last assertion follows for $C_2$ sufficiently large.
\end{proof}

\subsection{Proof of the mixing time bounds} \label{subsec:proof_of_mixing_time}

{\bf Proof of the upper bound in Theorem \ref{thm:main_theorem}.}

\smallskip

We stress that the coupling construction below uses only the transition, drift, Lyapunov, and local regularity assumptions of an MFGOU process and does not require the prior existence of a stationary distribution.

As observed in Section \ref{subsec:to_small_box_in_linear_time}, by adding a time of $C_\kappa nV(\bY_0)$ to the upper bound, we may assume that $\cY$ is restricted to $B(\bz,\kappa)$.
Pick $\kappa$ sufficiently small that all above results that hinge on $\kappa$ being small hold, pick $C_\kappa$ accordingly and make said assumption.

Let $\cY$ and $\cY'$ be two copies of a discrete or semicontinuous MFGOU on $B(\bz,\kappa)$.
The goal is to construct a coupling such that the coupling time $\tau=\min\{t:\bY_t=\bY'_t\}$ satisfies 
\[\Pro\left(\tau>\frac{1}{\lambda_1}n\log n + C_1 n\right) < \delta\]
for a sufficiently large constant $C_1=C_1(\delta)$.
The coupling is conceptually simple: let $\cY$ and $\cY'$ run independently until they both reach within $O(1/\sqrt{n})$ of the origin, and then use a sequence of attempts to bring them simultaneously into $B(\bz,\beta/\sqrt{n})$,
where $\beta>0$ will be chosen sufficiently small that Corollary \ref{cor:path_coupling_with_expansion} applies to a given admissible path coupling.

Since \(M=I-\tfrac{1}{2n}\cJ\), Lemma \ref{lemma:drop_the_squre_in_the_recursion} applies to \(\bmu_t\) with
\(E=\tfrac{1}{4}\cJ\). 
Indeed, then
\(A=I-\tfrac{1}{4n}\cJ\) and
\[A^2=M+O\left(\frac{1}{n^2}\right),\]
and the additional \(O(1/n^2)\) term is absorbed into the error term in the
recursion.
Hence, for a given constant $C_2$, setting $t:=\tfrac{1}{\lambda_1}n\log n + C_1 n$ and taking $C_1$ sufficiently large, Lemma \ref{lemma:drop_the_squre_in_the_recursion} shows that $\n \bmu_t \n \leq C_2/\sqrt{n}$. 
Since $\n \bY_t \n \geq 4C_2/\sqrt{n}$ implies $\n \bY_t - \bmu_t \n \geq 3C_2/\sqrt{n}$,
Markov's inequality and Lemma \ref{lemma:variance_bound} give 
\[\Pro(\bY_t \in B(\bz,4C_2/\sqrt{n})) > 1-\frac{\delta}{4}\]
provided that $C_2$ is sufficiently large.
Of course, the analogous statements for $\cY'$ also hold.
Make such a choice of $C_2$ and pick $C_1$ large enough accordingly.
Let $\tau_0$ be the first time that we find both $\cY$ and $\cY'$ in $B(\bz,4C_2/\sqrt{n})$. 
It follows from a union bound that
\[\Pro\left(\tau_0 > \frac{1}{\lambda_1}n\log n + C_1 n\right) < \frac{\delta}{2}.\]

{\bf The attempt scheme.} The rest of the analysis proceeds in {\em attempts}. An attempt consists of up to three stages according to the following algorithm.
\begin{itemize}
\item[(a)] Wait until the next time $\cY$ and $\cY'$ are both in $B(\bz,C_2/\sqrt{n})$. 
\item[(b)] After (a), wait for a maximum of $n$ steps until either $\cY$ and $\cY'$ are both in $B(\bz,\beta/\sqrt{n})$ or one of them exits $B(\bz,2C_2/\sqrt{n})$.
If the latter happens or time runs out before either occurs, declare the attempt a failure and begin a new attempt. If the former occurs, declare this step a success and go to (c).
\item[(c)] Apply a given admissible path coupling for a maximum number of $n/4\hat{K}$ ($\hat{K}$ chosen below) steps until either $\cY$ and $\cY'$ coalesce or one of them exits $B(\bz,2C_2/\sqrt{n})$.
If coalescence occurs, declare the attempt successful, otherwise a failure. In case of failure, begin a new attempt.
\end{itemize}
Notice that every failure in (b) or (c) results in a restart with both processes inside $B(\bz,4C_2/\sqrt{n})$ and frequently with both already in $B(\bz,C_2/\sqrt{n})$, making Step (a) vacuous for the next attempt. 
In some cases, both processes are even in $B(\bz,\beta/\sqrt{n})$, making Step (b) vacuous as well.

\smallskip

{\bf A smoothing argument.}
To bound the expected duration of Step~(a), we need a smooth approximation to the distance from the origin whose conditional drift is bounded above by 
\[ -\frac{c\|\bY_t\|}{n}+\frac{Q}{n^{3/2}}. \]
The Euclidean norm is not smooth at the origin, so we use the following approximation.
\[ \tilde Z_t=\phi(\bY_t), \,\, \phi(\yy):=\sqrt{\|\yy\|^2+\frac{Q_0}{n}}, \] 
for a constant $Q_0>0$.
We have 
\[ \nabla\phi(\yy) = \frac{\yy}{\sqrt{\|\yy\|^2+Q_0/n}} \] 
and 
\[ \nabla^2\phi(\yy) = \frac{I}{\sqrt{\|\yy\|^2+Q_0/n}} - \frac{\yy\yy^\top}{(\|\yy\|^2+Q_0/n)^{3/2}}. \] 
In particular, 
\[ \|\nabla\phi(\yy)\|\le 1, \,\, \|\nabla^2\phi(\yy)\|\le \sqrt{\frac{n}{Q_0}}.\] 
A second-order Taylor expansion of \(\phi\), 
together with the MFGOU transition rule and the conditional centering of $\bS_t$, therefore gives, uniformly on \(B(\bz,\kappa)\), 
\[ E[\tilde Z_{t+1}-\tilde Z_t\mid \bY_t=\yy] \le -\frac{\yy^\top \ff(\yy)} {2n\sqrt{\|\yy\|^2+Q_0/n}} +\frac{Q}{n^{3/2}}. \] 
For \(\kappa\) sufficiently small, \(\yy^\top \ff(\yy)\ge c\|\yy\|^2\) for a constant $c>0$. 
Since \[ \frac{\|\yy\|^2}{\sqrt{\|\yy\|^2+Q_0/n}} \ge \|\yy\|-\sqrt{\frac{Q_0}{n}}, \] 
it follows that 
\[ E[\tilde Z_{t+1}-\tilde Z_t\mid \bY_t=\yy] \le -\frac{c\|\yy\|}{n}+\frac{Q}{n^{3/2}}. \]

{\bf Duration of Step (a).} By increasing $C_2$ further if necessary, the smoothed process $\tilde{\cZ}$ satisfies
\[ E[\tilde Z_{t+1}-\tilde Z_t\mid \bY_t=\yy] \le -\frac{2Q}{n^{3/2}} \]
whenever $\n\yy\n \geq \tfrac{C_2}{2\sqrt{n}}$.
It follows that if either of $\n\yy\n$ or $\n \yy'\n$ exceeds $C_2/\sqrt{n}$,
\[ E[\tilde Z_{t+1}-\tilde Z_t + \tilde Z'_{t+1}-\tilde Z'_t \mid \bY_t=\yy,\bY'_t=\yy'] \le -\frac{Q}{n^{3/2}}. \]
At the beginning of Step (a), both processes are in $B(\bz,4C_2/\sqrt n)$, so the initial value of $\tilde Z_t+\tilde Z'_t$ is $O(1/\sqrt{n})$.
From optional stopping, it then follows that the expected duration of Step (a) is bounded by $Qn$.

\smallskip

{\bf Success probability and duration of Step (b).} Lemma \ref{lemma:hitting_inner_box} applied with $c=1$ implies that when started from $B(\bz,C_2/\sqrt{n}) \times B(\bz,C_2/\sqrt{n})$, Step (b) is a success with probability $p=p(\beta,C_2)>0$ for any given $\beta>0$.
By definition, its duration is bounded by $n$.

\smallskip

{\bf Choice of $\beta$ and $\hat{K}$, and analysis of Step (c).} We will appeal to Corollary \ref{cor:path_coupling_with_expansion}. Before doing so, however, it needs to be established that for $\beta$ sufficiently small, Step (c) is unlikely to fail through an early exit. 
For that, consider again $\tilde{Z}_t$. The outward drift is bounded by $Q/n^{3/2}$, so the process $U_t:=\tilde{Z}_t-Qt/n^{3/2}$ is a supermartingale.
It follows from the Maximal Azuma-Hoeffding inequality (see \cite{Williams}) that if $U_0=a/\sqrt{n}$ and $\tau$ is the first time $U_t$ exceeds $2a/\sqrt{n}$, then $\Pro(\tau \leq \ell n) \leq 1/9$ for $\ell>0$ sufficiently small.

Now with $Z_0 \leq \beta/\sqrt{n}$ and $\beta \leq \sqrt{Q_0}$, we have $\tilde{Z}_0=U_0 \leq \sqrt{2Q_0/n}$.
Choosing again $C_2$ larger if necessary, we have $2C_2-Q > 2\sqrt{2Q_0}$ and thus whenever $\ell$ is sufficiently small
\[\Pro(\tau_{\rm out} \leq \ell n) \leq \frac{1}{9}\]
for $\tau_{\rm out}$ being the first time that $\cU$ exceeds $(2C_2-Q)/\sqrt{n}$.
Furthermore, provided that $\ell \leq 1$, then if $U_t$ does not exceed $(2C_2-Q)/\sqrt{n}$, then $\tilde{\cZ}_t$ does not exceed $2C_2/\sqrt{n}$ and since $\tilde{Z_t} \geq \n\bY_t\n$, neither does $\{\n \bY_t\n\}$.

\smallskip

Now let $D_t=\n\bY_t-\bY'_t\n$.
By Proposition \ref{prop:long_range_coupling_bounds}, there exists an admissible path coupling. 
Fix one such coupling to use in what follows.

Consider first the discrete case.
The admissible path coupling satisfies the conditions of
Corollary \ref{cor:path_coupling_with_expansion} with $\alpha=\chi$ and some $K<\iy$.
Applying the corollary with $\hat{K}=K \vee (1/4\ell) \vee \sqrt{Q}$,
it follows that if started with 
\[D_0^2 \leq \frac{\alpha}{16000\hat{K}n},\]
coalescence occurs under the admissible path coupling within time $\tfrac{n}{4\hat{K}} \leq \ell n$ with probability at least $7/9$.
Choose
\[\beta=\sqrt{\frac{\alpha}{64000\hat K}},\]
so that if both chains are in \(B(\bz,\beta/\sqrt n)\), 
then
$D_0^2\le \tfrac{\alpha}{16000\hat K n}$.
Since \(\Pro(\tau_{\rm out} \leq \ell n) \leq 1/9\), the probability that either
\(\cY\) or \(\cY'\) exits \(B(\bz,2C_2/\sqrt n)\) within time \(\ell n\) is at most
\(2/9\). Combining this with the \(7/9\) coalescence probability from
Corollary \ref{cor:path_coupling_with_expansion}, Step (c) succeeds with probability at least \(7/9-2/9=5/9\).

In the semicontinuous case, the same argument brings the distance below $\epsilon/n$ within $\ell n$ steps with probability at least $7/9$. 
Conditional on this event, Proposition \ref{prop:short_distance_coupling} gives a probability at least $\alpha_0>0$ of coalescence in the next step. 
Thus, apart from the possibility of an early exit, coalescence occurs during Step~(c) with probability at least $\gamma:=7\alpha_0/9$. By choosing $\ell>0$ sufficiently small and then choosing $\hat K$ and $\beta$ accordingly, 
the probability that either process exits during Step~(c) can be made smaller than $\gamma/2$. 
Hence Step~(c), extended by one step if necessary, succeeds with probability $q=\gamma/2$ bounded away from zero also in the semicontinuous case.


\smallskip

{\bf Conclusion.} Let $\cF_i$ denote the history of $(\cY,\cY')$ up to the start of the $i$'th attempt.  Since the preceding estimates hold uniformly over all states from which an attempt may begin, we have, conditional on $\cF_i$, that (a) attempt $i$ succeeds with probability at least $pq$ and (b) the expected time taken by an attempt is bounded by $Qn$.

More formally, let $T_i$ be the duration of attempt $i$, let $G$ be the number of attempts that are required until $\cY$ and
$\cY'$ finally coalesce and $F_i=\{G>i\}$. 
Iterated conditioning gives
\[\Pro(F_i) \leq(1-pq)^i,\] 
so $G$ is stochastically dominated by a geometric random variable with parameter $pq$.
Moreover, 
\[\E[\tau_c-\tau_0] \leq \sum_{i=1}^\iy \E[T_i\one_{F_{i-1}}] = \sum_{i=1}^\iy \E[\one_{F_{i-1}}\E[T_i|\cF_i]] 
\leq Qn\sum_{i=1}^\iy (1-pq)^{i-1} \leq C_3n.\]
Hence by Markov's inequality
\[\Pro(\tau_c-\tau_0>C_4 n) < \frac{\delta}{2}\]
for $C_4=2C_3/\delta$.
Adding the bound on $\tau_0$, it follows that
\[\Pro\left(\tau_c > \frac{1}{\lambda_1}n\log n + C_5 n\right) < \delta\]
for $C_5=C_1+C_4$.
The Coupling inequality now gives 
\[\tauh(\cY;\delta) \leq \frac{1}{\lambda_1}n\log n + C_5 n.\]
Recalling that this bound was obtained under the assumption of a starting point in $B(\bz,\kappa)$, we add the $O(nV(\bY_0))$ entrance-time term to finish the proof.
\hfill $\Box$

\smallskip

{\bf Proof of the lower bound in Theorem \ref{thm:main_theorem}.}

\smallskip


Let $\bY_0=\bmu_0 = \kappa\uu$ for a sufficiently small constant $\kappa$ and $\uu$ the unit eigenvector of $M$ for the eigenvalue $1-\tfrac{\lambda_1}{2n}$.
For $t \leq (1/\lambda_1)n\log n - C_1n$ and $C_1$ large enough, it follows from the Lemma \ref{lemma:drop_the_squre_in_the_recursion} with $A^2=M+O(1/n^2)$,
that $\n \bmu_t \n \geq 3C_2/\sqrt{n}$ for any large given constant $C_2$.
On the event $\{\n \bY_t \n \leq 2C_2/\sqrt{n}\}$, we have $\n \bY_t-\bmu_t\n \geq C_2/\sqrt{n}$.
Hence by Markov's inequality and Lemma \ref{lemma:variance_bound},
\[\Pro\left(\n\bY_t\n \leq \frac{2C_2}{\sqrt{n}}\right) < \rho\]
for arbitrarily small $\rho>0$ upon taking $C_2$ sufficiently large and then $C_1$ large enough accordingly.

In contrast, for all sufficiently large $s$, $\n\bmu_s\n \leq C_2/\sqrt{n}$, so an identical application of Markov's inequality gives 
\[\Pro\left(\n\bY_s\n \geq \frac{2C_2}{\sqrt{n}}\right) < \rho.\]
By the upper bound just proved, the law of $\bY_s$ converges in total variation to $\pi$ as $s\to\infty$.
Hence, with $\bY_\iy$ having law $\pi$,
\[\Pro\left(\n\bY_\iy\n \geq \frac{2C_2}{\sqrt{n}}\right) < \rho.\]
Consequently,
\[\n\Pro(\bY_t \in \cdot)-\pi\n_{TV} \geq 1-2\rho.\]
This proves the lower bound.
\hfill $\Box$

\section{Applications} \label{sec:applications}

In this section, we develop the applications outlined in the introduction. 
They have somewhat different logical status. 
The Gibbs-sampling and mean-field block-measure results are direct applications of Theorems \ref{thm:main_theorem} and \ref{thm:exponential_mixing_time} once their stated hypotheses have been verified. 
The conclusions for full LDA are conditional on the stationary-point structure of the associated free energy, whose complete analytic description remains open. 
The stochastic-gradient-descent applications require the stated boundedness and regularity assumptions, while the Hamiltonian Monte Carlo application uses a Gaussian-increment extension of the MFGOU analysis. 
The comparison with Metropolis--Hastings is heuristic and is not claimed as a consequence of Theorem \ref{thm:main_theorem}.

\subsection{Gibbs/Metropolis--Hastings sampling from a unimodal density} \label{subsec:Metropolis--Hastings}

Assume that $h:\RR^m \ra \RR$ is a three times continuously differentiable function with a unique stationary point, 
which is its global minimum. Assume wlog that the minimum is at the origin and $h(\bz)=0$, and that $\n \nabla h(\yy)\n$ is bounded away from $0$ on $B(\bz,\kappa)^c$ for each $\kappa>0$. 
Note that with these assumptions, there is a $q>0$ such that 
\[\n \nabla h(\yy)\n^2 \geq q(h(\yy) \wedge 1).\] 
We consider sampling from the probability density $H(\yy) \propto e^{-nh(\yy)}$. 

\smallskip

{\bf Gibbs sampling:}
As we will formalize it, Gibbs sampling works as follows. Specify once and for all a symmetric probability measure $\zeta$ whose density has compact support, with the additional property that there exist $\sigma,q>0$ such that $\zeta(\bs)>q$ for all $\bs$ in the
ball $B(\bz,\sigma)$.
The density is with respect to Lebesgue measure or counting measure, depending on the application at hand. In the discrete case, we assume that $\sigma$ can be taken larger than $1$.
In the continuous case, $\zeta(\cdot+\xx)$ is assumed Lipschitz in total variation as a function of $\xx$.

Fix $\gamma \in(0,\Gamma^{-1})$, where $\Gamma=\sup_\kappa \Gamma_\kappa$ and
\[\Gamma_\kappa = \Gamma_\kappa^{(n)} :=\sup\left\{\frac{H(\yy+\bs/n)}{H(\yy)+H(\yy+\bs/n)}\right\},\] 
where the supremum is over $\yy \in B(\bz,\kappa)$ and $\bs$ in the support of $\zeta$. 
Note that $\Gamma \in (1/2,1]$, so that $\Gamma^{-1} \in [1,2)$. Moreover, $\Gamma^{(n)}_\kappa=1/2+O(\kappa+1/n)$ and values of $\gamma$ arbitrarily close to $2$ become permissible when considering the process restricted to $B(\bz,\kappa)$ and $n$ large.

Given $\bY_t=\yy$, sample $\bS \sim \zeta$ and let $\yy'=\yy+\bS/n$.
Then let $\bY_{t+1}=\yy'$, with probability $\gamma H(\yy')/(H(\yy)+H(\yy'))$, otherwise let $\bY_{t+1}=\yy$.
The acceptance probability is a scaled version of Barker's rule \cite{Barker1965}, and may equivalently be viewed as a Gibbs or heat-bath choice between the current state and the proposed state.
The Markov process so defined is readily verified to be irreducible, aperiodic, and reversible, and the stationary distribution is $H$ by the detailed balance equations.

Since $h$ is $C^3$ and $\bS$ has compact support, the acceptance probability of the proposal is Lipschitz in $\yy$ on every compact set, with a Lipschitz constant that is uniform in $n$, and
bounded away from $0$ uniformly in $n$ and over the support of $\bS$.
Consequently, a fixed positive multiple of the symmetric proposal law $\zeta$ can be extracted from the absolutely continuous part of every increment law and used as symmetric component $\mu_s$. 
The remaining holding probabilities and residual laws are Lipschitz in total variation, uniformly in $n$.
Thus the resulting increment variables $\bxi^{(n)}$ have compact support and satisfy the discrete or semicontinuous local regularity assumptions.

For the drift in the MFGOU representation, note that 
\[\frac{H(\yy')}{H(\yy)+H(\yy')}= \frac{1}{1+e^{\nabla h(\yy)^\top\bs}}+O\left(\frac{1}{n}\right).\]
Thus, up to an $O(1/n^2)$ term,
\begin{align*}
    \E[\bY_{t+1}|\bY_t=\yy] - \yy &= \frac{\gamma}{2n}\int\bs\left(\frac{1}{1+e^{\nabla h(\yy)^\top\bs}} - \frac{1}{1+e^{-\nabla h(\yy)^\top\bs}}\right) \zeta(d\bs) \\
    &= -\frac{\gamma}{2n}\int\bs\frac{e^{\nabla h(\yy)^\top\bs}-e^{-\nabla h(\yy)^\top\bs}}{(1+e^{\nabla h(\yy)^\top\bs})(1+e^{-\nabla h(\yy)^\top\bs})} \zeta(d\bs) \\
    &=-\frac{\gamma}{2n} \E\left[ \bS\frac{e^{\nabla h(\yy)^\top\bS}-e^{-\nabla h(\yy)^\top\bS}}{(1+e^{\nabla h(\yy)^\top\bS})(1+e^{-\nabla h(\yy)^\top\bS})} \right],
\end{align*}
with the expectation taken with respect to $\bS$.
Hence the drift field is 
\[\ff(\yy) = \gamma\E\left[ \bS\frac{e^{\nabla h(\yy)^\top\bS}-e^{-\nabla h(\yy)^\top\bS}}{(1+e^{\nabla h(\yy)^\top\bS})(1+e^{-\nabla h(\yy)^\top\bS})}
\right],\]
which is $C^2$, since $h$ is $C^3$.
We next verify the hypotheses needed to apply Theorem \ref{thm:main_theorem}. First, we show that its Lyapunov condition holds with $V=h$, i.e.\ that for some $c>0$,
\[\E\left[ \nabla h(\yy)^\top\bS\frac{e^{\nabla h(\yy)^\top\bS}-e^{-\nabla h(\yy)^\top\bS}}{(1+e^{\nabla h(\yy)^\top\bS})(1+e^{-\nabla h(\yy)^\top\bS})} \right] \geq c(h(\yy) \wedge 1).\]
Note that the quantity in the expectation is nonnegative and remains the same if $\nabla h(\yy)^\top\bS$ is replaced with $|\nabla h(\yy)^\top\bS|$.
Observe also the algebraic fact that $(e^{|x|}-e^{-|x|})/(1+e^{|x|})(1+e^{-|x|}) \geq \tfrac{1}{4}(|x| \wedge 1)$.
Hence
\[\nabla h(\yy)^\top \ff(\yy) \geq \frac{\gamma}{4}\E\left[ |\nabla h(\yy)^\top\bS| \wedge |\nabla h(\yy)^\top\bS|^2 \right].\]
By assumption, there exist $r,q>0$ such that the annulus $A:=B(\bz,2r)\setminus B(\bz,r)$ satisfies $\zeta(\bs)>q$ for all $\bs \in A$. (Here $r$ can be taken to be any number in $(1/2,1)$ in the discrete case.)
Let 
\[A_0=A_0(\yy)=\left\{\bs \in A: \nabla h(\yy)^\top\bs \geq \frac{1}{\sqrt{m}}\n \nabla h(\yy)\n \, \n \bs \n\right\}.\]
Then $\zeta(A_0) \geq q\mu(A_0) \geq q_0 >0$ uniformly in $\yy$ (where $\mu$ is either Lebesgue measure or counting measure).
On $A_0$, $|\nabla h(\yy)^\top\bS| \geq \tfrac{r}{2\sqrt{m}}\n\nabla h(\yy)\n$.
Since $\n \nabla h(\yy) \n^2 \geq q (h(\yy) \wedge 1)$, by considering separately the regimes $h(\yy) \geq 1$ and $h(\yy)<1$, 
this shows that for some $a>0$, 
\[|\nabla h(\yy)^\top\bs| \wedge |\nabla h(\yy)^\top\bs|^2 \geq a(h(\yy) \wedge 1)\] 
on $A_0$.
Taking $c=\gamma a q_0/4$ proves the claim.

It remains to identify the linearization of $\ff$ at the origin.
Let $\cH$ be the Hessian $\nabla^2h(\bz)$. Since $\nabla h(\bz)=\bz$, Taylor's formula gives $\nabla h(\yy) = \cH \yy + O(\n \yy \n^2)$.

Write $\Sigma$ for the covariance matrix $\E[\bS\bS^\top]$. 
Inserting this into the expression for $\ff$  gives
\[\ff(\yy) = \frac{\gamma}{2} \Sigma\cH \yy + O(\n \yy \n^2),\]
i.e.\ $\ff$ is of the form covered by Remark \ref{remark:Sigma*H}.
Theorem \ref{thm:main_theorem} gives for all $\ep \in (0,1)$,
\[\taum(\cY;\ep,B(\bz,K)) = \frac{2}{\gamma \lambda_1(\Sigma\cH)}n\log n + O(n),\]
for all $K<\iy$.

\smallskip

The more commonly used {\bf Metropolis--Hastings algorithm} \cite{MetropolisEtAl1953,Hastings1970} instead accepts the proposed move with probability 
\[ 1\wedge \frac{H(\mathbf y')}{H(\mathbf y)}. \]
After an initial $O(n)$ period, both Gibbs sampling and Metropolis–Hastings remain, with overwhelming probability, in an arbitrarily small ball $B(\boldsymbol 0,\kappa)$ up to an exponentially large time horizon.
(Although we do not provide a full proof of the corresponding statement for MH,
one can readily show that $h(\bY_t)$ decreases rapidly to a small ball and remains with overwhelming probability within that region for an exponential time.)

Considering Gibbs sampling restricted to $B(\bz,\kappa)$, one can take the 
parameter $\gamma$ arbitrarily close to $2$ by taking $\kappa$ correspondingly small, making the acceptance probability arbitrarily close to $1$. 
For MH, the probability of a move automatically becomes arbitrarily close to $1$ by taking $\kappa$ sufficiently small. 
Further, it is straightforward to verify that the conditional transition probabilities differ between the two processes only by $O(\kappa+n^{-1})$ in total variation.
These considerations strongly suggest that the mixing time for MH satisfies the same asymptotic behavior:
\[\taum(\cY;\ep,B(\bz,K))=\frac{1}{\lambda_1(\Sigma \cH)}n\log n + O(n),\]
although establishing this rigorously would require additional technical work. Given how closely MH and Gibbs agree in the relevant regime, we do not pursue this extension with full rigor here.

\begin{remark}
{\rm  If $h$ has an additional nondegenerate local minimum, a local Lyapunov condition as in the hypothesis of Theorem \ref{thm:exponential_mixing_time} can be established exactly as above, locally around each such minimum. 
Hence the mixing time is exponential in that case. 
    }
\end{remark}

\bigskip

\subsection{Mean-field block measures and Glauber dynamics} \label{subsec:mean_field_glauber}

In this section we show how the Glauber dynamics for a general mean‑field discrete model can be placed within the MFGOU framework.
The steps are as follows:
\begin{itemize}
\item[(i)] express the one‑step drift of the projected chain $\cY$ in terms of the derivatives of the potential $g$,
\item[(ii)] identify the drift vector field $\ff$ and show that it satisfies the discrete MFGOU conditions on compact interior subsets,
\item[(iii)] construct a censoring mechanism that keeps the chain in such a subset while preserving the drift,
\item[(iv)] compute the Jacobian of $\ff$ at the unique minimizer of the associated free energy,
\item[(v)] deduce a general mixing‑time theorem for these models from Theorem \ref{thm:main_theorem}.
\end{itemize}

The framework is intended to cover, among other examples, the classical Curie--Weiss Ising and Potts models, for which Glauber dynamics has been analyzed in detail \cite{LLP,Cuff}, as well as block-structured variants such as the Potts model on the complete bipartite graph \cite{HernandezKovchegovOtto2017} and the Ising model on complete multipartite graphs \cite{Kim2022}.
It also encompasses multinomial mixture models and the projected Gibbs samplers arising from the LDA models considered below.
The main purpose of the general formulation below is to capture the common free-energy and covariance structure behind these models.

The conclusion of the general analysis can be described informally as follows. 
If $\aaa$ is the unique nondegenerate minimum of the free energy, then the drift linearized around $\aaa$ is given by $2B\mathcal H$, 
where $\mathcal H$ is the Hessian of the free energy at $\aaa$ and $B$ is a block-diagonal covariance matrix determined by the single-step dynamics at equilibrium.  
In one dimension, this reduces simply to the product of the free-energy curvature at equilibrium and the variance factor coming from the single-step dynamics.
The slowest eigenmode of the product determines the mixing-time constant. 
The formal setup below is needed to derive this statement for general block models. It will then be made explicit in the Ising, Potts, and LDA examples.

\smallskip

A $q+1$-state mean-field block measure is one that can be constructed as follows.
Let $r_1,r_2,\ldots,r_d$ be positive numbers summing to $1$. 
For $r > 0$, write $\Delta^{(\ep)}_r = \{x_1,\ldots,x_q \geq \ep:\sum_{k=1}^q x_k \leq r-\ep\}$, so that $\Delta_r:=\Delta^{(0)}_r$ is the $q$-dimensional $r$-simplex; the $0$-state is implicitly included as $x_0=r-\sum_{\sigma=1}^q x_\sigma$.
Let $\cC=\Delta_{r_1} \times \Delta_{r_2} \times \cdots \times \Delta_{r_d}$ and let
$g$ be a potential, i.e.\ a three times continuously differentiable function on $\cC$.
Let $n$ be a large natural number (in terms of which mixing times will later be studied).

Next, let $N=[n]$ and let $\{N_1,\ldots,N_d\}$ be a partition that satisfies $|N_k|=nr_k$ (where we understand $nr_k$ as rounded to a neighboring integer if it is not already one, and where the rounding rule satisfies that the rounded $nr_i$'s sum to $n$).
For $\zz \in \{0,1,\ldots,q\}^N$, let for $k=1,\ldots,d$, $\sigma=0,1,\ldots,q$ 
\[p_{k,\sigma}(\zz)=\tfrac{1}{n}|\{v \in N_k:z_v=\sigma\}|,\] 
i.e.\ the proportion of $v$'s in block $N_k$ that have value $z_v=\sigma$. 
Note that
\[p_{k,0}(\zz)=r_k-\sum_{\sigma=1}^q p_{k,\sigma}(\zz).\]
Define the {\em block measure} $\nu:=\nu^{(g)}$ on $\{0,1,\ldots,q\}^N$ as the probability measure given by
\begin{equation} \label{eq:block_measure}
   \nu(\zz) \propto \exp\{-ng(p_{1,1}(\zz),\ldots,p_{d,q}(\zz))\}. 
\end{equation}
Defining $\yy = \yy(\zz) = (p_{1,1}(\zz),\ldots,p_{d,q}(\zz))^\top \in L:= \cC \cap \tfrac{1}{n}\Z$ and projecting $\nu$ onto $L$ gives rise to the {\em mean-field block measure} $\pi$:
\begin{equation} \label{eq:projected_block_measure}
\pi(\yy) \propto \left(\prod_{k=1}^d\,\prod_{\sigma=0}^q \binom{nr_k}{nu_{k,\sigma}}\right)e^{-ng(\yy)},\, \yy \in L,
\end{equation}
By a Stirling approximation, it is, up to small discrepancies, proportional to $e^{-nh(\yy)}$, where $h$ is the free energy
\[h(\yy):=g(\yy)+\sum_{k=1}^d\sum_{\sigma=0}^q y_{k,\sigma}\log y_{k,\sigma}.\]
While this approximation of $\pi$ will not be used per se, the free energy $h$ will play a fundamental role below, and the approximation gives an intuitive explanation of why that is. 
For simplicity, we will often denote the double sum on the right by $e(\yy)$ and refer to it as the entropy part of $h$. Its Hessian will be denoted $\cE(\yy)$. Letting $\cH(\yy)$ and $\cG(\yy)$ denote the Hessians of $h$ and $g$, respectively, we then have the relation $\cH=\cG+\cE$.

The standard way to sample from a probability distribution like $\nu$ is Glauber dynamics/Gibbs sampling; i.e.\ with the Markov chain $\cZ$ on $\{0,1,\ldots,q\}^N$ 
with the updating rule that given $\bZ_t=\zz$, pick a $v \in N$ uniformly at random
and then take $\bZ_{t+1}(u)=\zz(u)$ for all $u \neq v$ and
\[\Pro(\bZ_{t+1}(v) = \sigma \mid v) \propto \exp\{-ng(p_{1,1}(\zz^{(\sigma)}),\ldots,p_{d,q}(\zz^{(\sigma)}))\},\]
where $\zz^{(\sigma)}$ is the configuration that agrees with $\zz$ for $u \neq v$ and has $\zz^{(\sigma)}(v)=\sigma$.

This Markov chain is lumpable and admits the projected Glauber dynamics $\cY$, 
\[\bY_t:=\bY(\bZ_t) = (p_{1,1}(\bZ_t),\ldots,p_{d,q}(\bZ_t)) \in L,\]
as the update rule depends on $\zz$ only through the local fractions $y_{k,\sigma}=p_{k,\sigma}(\zz)$.
It is sometimes more convenient to extend the $\yy$-vectors to also contain the coefficients
$p_{k,0}(\zz)$, 
so that we then think of $\bY_t$ as
\[\bY_t=(p_{1,0}(\bZ_t),p_{1,1}(\bZ_t),\ldots,p_{1,q}(\bZ_t),p_{2,0}(\bZ_t),p_{2,1}(\bZ_t)\ldots,p_{d,q}(\bZ_t)),\]
keeping in mind that $\bY_t$ is in reality $dq$-dimensional and that the $y_{k,0}$'s are helpful sham variables.
Then a transition changes two coordinates within the same block or keeps $\yy$ unchanged.  
More precisely, writing $\delta^{(k)}_{i \ra j} = (\ee_{k,j}-\ee_{k,i})/n$ (so that consequently $\delta^{(k)}_{i \ra i}=\bz$),
\begin{equation} \label{eq:proj_block_trans_prob}
\Pro(\bY_{t+1} = \yy+\delta^{(k)}_{i \ra j}|\bY_t=\yy) =\frac{y_{k,i}\exp\{-n(g(\yy+\delta^{(k)}_{i \ra j})-g(\yy))\}}
{\sum_{\sigma=0}^q \exp\{-n(g(\yy+\delta^{(k)}_{i \ra \sigma})-g(\yy))\}}.
\end{equation}

\smallskip

We begin by deriving the drift structure of the projected chain $\cY$ on the simplex. A boundary‑repulsion condition (PID), stated later, will ensure that $\cY$ remains in the interior where these expansions hold and where censoring can be applied.

Recalling that $y_{k,0}=r_k-\sum_{\sigma=1}^q y_{k,\sigma}$, it follows that for $k=1,\ldots,d,\, \sigma=1,\ldots,q$,
\[h'_{k,\sigma}(\yy) = \log\frac{y_{k,\sigma}}{y_{k,0}}+g'_{k,\sigma}(\yy),\]
which can also be expressed as
\[e^{-g'_{k,\sigma}(\yy)} = \frac{y_{k,\sigma}}{y_{k,0}}e^{-h'_{k,\sigma}(\yy)}.\]
For any compact set $K \subset \cC^\circ$ and $\yy \in K$, the transition probability from $\yy$ can then, using Taylor's formula and the convention $g'_{k,0}(\yy)=h'_{k,0}(\yy)=0$, be written
\begin{equation} \label{eq:trans_prob_in_terms_of_h}
\Pro(\bY_{t+1}=\yy+\delta^{(k)}_{i \ra j}|\bY_t=\yy)
= \frac{y_{k,i}y_{k,j}e^{-h'_{k,j}(\yy)}}{W_k(\yy)} + O\left(\frac{1}{n}\right),
\end{equation}
where 
\[W_k(\yy):=\sum_{\sigma=0}^q y_{k,\sigma}e^{-h'_{k,\sigma}(\yy)}\] 
and the $O(1/n)$ term is uniform in $\yy \in K$.
Since $g$ is $C^3$, the transition probabilities are Lipschitz on $K$ uniformly in $n$. 
Moreover, since $K$ is compactly contained in $C^\circ$, the probabilities of standing still and of each of the elementary moves $\pm \ee_j/n$ are bounded away from zero uniformly in $n$.
Hence this describes a family of processes that sits squarely within the discrete MFGOU framework if it can be established that the drift field is of the proper form.

\smallskip

Note that $\sum_\sigma y_{k,\sigma}=r_k$.
Summing the contributions to a move $-1/n$ and a move $1/n$ respectively for $y_{k,i}$, we get 
\[\E[Y_{t+1}(k,i)|\bY_t=\yy] = y_{k,i}-\frac{1}{n}y_{k,i}\left(1-\frac{r_k e^{-h'_{k,i}(\yy)}}{W_k(\yy)}\right) + O\left(\frac{1}{n^2}\right).\]
In the MFGOU framework, this means that $f_{k,i}(\yy)$, $k \in \{1,\ldots,d\}$, $i \in \{1,\ldots,q\}$ needs to be 
\[f_{k,i}(\yy) = 2y_{k,i}\left(1-\frac{r_k e^{-h'_{k,i}(\yy)}}{W_k(\yy)}\right).\] 
We will need the following local linearization of $\ff$. Let $\aaa \in \cC^\circ$ be a stationary point of $h$ and let $\cH=\cH(\aaa)$.
Eliminating the sham quantities $y_{k,0}$ and $h'_{k,0}$,
\begin{align*}
W_k(\yy) &= r_k-\sum_{\sigma=1}^q y_{k,\sigma}\left(1-e^{-h'_{k,\sigma}(\yy)}\right) \\
&= r_k-\sum_{\sigma=1}^q y_{k,\sigma}(\cH\uu)_{k,\sigma} + O(\n\uu\n^2),
\end{align*}
where $\yy=\aaa+\uu$.
Inserting into the expression for $f_{k,i}$,
this gives
\begin{align*}
   \frac{1}{2}\,f_{k,i}(\yy) &= y_{k,i} \left(1-\frac{r_k e^{-h'_{k,i}(\yy)}}{W_k(\yy)}\right) \\
   &= y_{k,i}\left(1-\frac{r_k(1-(\cH\uu)_{k,i})}{r_k-\sum_\sigma y_{k,\sigma}(\cH\uu)_{k,\sigma}}\right) \\
   &= y_{k,i} \left( (\cH\uu)_{k,i} - \frac{1}{r_k}\sum_\sigma y_{k,\sigma}(\cH\uu)_{k,\sigma}\right) + O(\n \uu \n^2) \\
   &= (a_{k,i}+ u_{k,i}) \left( (\cH\uu)_{k,i} - \frac{1}{r_k}\sum_\sigma (a_{k,\sigma}+u_{k\sigma})(\cH\uu)_{k,\sigma}\right) + O(\n\uu\n^2)\\
   &= a_{k,i} \left( (\cH\uu)_{k,i} - \frac{1}{r_k}\sum_\sigma a_{k,\sigma}(\cH\uu)_{k,\sigma}\right)+ O(\n\uu\n^2).
\end{align*}
In vector form, this becomes
\[\ff(\yy)=2B\cH\uu+O(\n \uu \n^2)\]
where 
\[B=\mathrm{blockdiag}(B_1,B_2,\ldots,B_d),\]
\[B_k = \mathrm{diag}(\aaa_k)-\frac{1}{r_k}\aaa_k\aaa_k^\top,\]
and \[\aaa_k=(a_{k,1},\ldots,a_{k,q})^\top.\]
Notice that $B_k$ is $r_k$ times the covariance matrix of $(\one_{\{U=1\}},…,\one_{\{U=q\}})$,
where $U$ takes values in $\{0,1,…,q\}$ with probabilities $\Pro(U=\sigma)=a_{k,\sigma}/r_k$.  
Since $a_{k,\sigma}>0$ for each $\sigma=0,1,\ldots,q$, each $B_k$, and hence also $B$ is positive definite.

Next, we claim:

\begin{lemma} \label{lemma:proj_block_measure_lyapunov_on_compact_interiors}
If $h$ has a local minimizer $\aaa \in \cC^\circ$ with $\cH=\cH(\aaa)$ positive definite, 
then, taking $V(\yy)=h(\yy)-h(\aaa)$, the local Lyapunov condition of Theorem \ref{thm:exponential_mixing_time} around $\aaa$ holds, i.e.
\[\nabla V(\yy)^\top \ff(\yy) \geq cV(\yy)\]
for all $\yy$ in a sufficiently small neighborhood of $\aaa$.

\smallskip

If in addition, $\aaa$ is the unique local minimizer, and hence also the unique global minimizer, and $\n \nabla h\n$ is bounded away from $0$ on each $B(\aaa,\kappa)^c$, $\kappa>0$,
then, taking $V(\yy)=h(\yy)-h(\aaa)$, the global Lyapunov condition
\[\nabla V(\yy)^\top \ff(\yy) \geq c(V(\yy) \wedge 1)\]
of Theorem \ref{thm:main_theorem} holds on any compact $K \subseteq \cC^\circ$ that contains $\aaa$.

\end{lemma}

\begin{proof}
We must prove two things, namely 
\begin{itemize}
\item[(i)] if $\aaa \in \cC^\circ$ is a local minimizer and $\|\yy-\aaa\|$ is sufficiently small, then $\nabla h(\yy)^\top \ff(\yy) \geq c \, (h(\yy)-h(\aaa))$, 
\item[(ii)] if $\aaa$ satisfies the conditions of the second assertion and $h(\yy) - h(\aaa)\geq \delta$ for some given $\delta>0$, then $\nabla h(\yy)^\top \ff(\yy) \geq c>0$.
\end{itemize}

For (i), the linearization of $\ff$, with $\uu=\yy-\aaa$, gives
\[\nabla h(\yy)^\top \ff(\yy) = 2\uu^\top \cH B\cH \uu + O(\n \uu\n^3),\]
and since $\cH$ and $B$ are positive definite, the desired result follows.

For (ii), fix $\delta>0$ and consider the compact set
\[K_\delta=\{\yy \in K :h(\yy) - h(\aaa) \geq \delta\}.\]
By assumption, $\nabla h(\yy) \neq \bz$ on $K_\delta$. 
Let $(X,U) \in \{1,\ldots,d\} \times \{0,1,\ldots,q\} $ be a random variable with $\Pro((X,U)=(k,\sigma))=y_{k,\sigma}$.
Then
\begin{align*}
\nabla h(\yy)^\top \ff(\yy) &= 2\sum_{k=1}^d r_k \sum_{i=1}^q \frac{y_{k,i}}{r_k} h'_{k,i}(\yy)\left(1-\frac{r_k e^{-h'_{k,i}(\yy)}}{W_k(\yy)}\right) \\
&= 2\sum_{k=1}^d r_k \sum_{i=0}^q \frac{y_{k,i}}{r_k} h'_{k,i}(\yy)\left(1-\frac{r_k e^{-h'_{k,i}(\yy)}}{W_k(\yy)}\right) \\
&= 2\,\E\left[\Cov\left(h'_{X,U}(\yy),1-\frac{r_X e^{-h'_{X,U}(\yy)}}{W_X(\yy)} \,\Big| \, X\right)\right],
\end{align*}
where the second equality is the convention $h'_{k,0}=0$ and the third follows by noting that the second of the two random variables in the conditional covariance has mean $0$.
Conditional on $X=k$, the second variable is a strictly increasing function of the first, so each conditional covariance is nonnegative.
Also, since $h'_{k,0}=0$, the conditional distribution of $h'_{X,U}(\yy)$ given $X=k$ must be nondegenerate for at least one $k$, so at least one of the conditional covariances must be strictly positive and we get
\[\nabla h(\yy)^\top \ff(\yy) >0\]
on $K_\delta$.
Here the strict positivity of the covariance follows from the symmetry in the identity
\[\Cov(X,F(X))=\frac{1}{2}\,\E[(X_1-X_2)(F(X_1)-F(X_2))],\] 
with $X_1,X_2$ two independent copies of $X$, and applying with $F$ strictly increasing.
Since $\nabla h(\yy)^\top \ff(\yy)$ is continuous and $K_\delta$ is compact, $c:=\min_{\yy \in K_\delta}\nabla h(\yy)^\top \ff(\yy)>0$.

Taking $\delta$ sufficiently small that (i) applies whenever $h(\yy)-h(\aaa)<\delta$, and combining (i) with the preceding uniform lower bound on $K_\delta$, proves, after decreasing $c$ if necessary, that
\[\nabla V(y)^\top \ff(\yy) \geq c(V(\yy) \wedge 1)\]
throughout $K$.
\end{proof}

Equipped with this lemma, the MFGOU machinery is available on compact interior sets.
If we can also prove that $\cY$ in linear time moves into such a compact set and stays there, and that censoring can be applied in a way that preserves the Lyapunov property, mixing times for the censored process apply to $\cY$ itself with the addition of an $O(n)$-term. 
To this end, observe that $\cY$ satisfies the following property.

\smallskip
\noindent {\bf Pointwise inward drift at the boundary:} There exist $\alpha,\delta>0$ such that for each $(k,\sigma)$, including $\sigma=0$, and each $\yy$ with $y_{k,\sigma}<\delta$, $\E[Y_{t+1}(k,\sigma)|\bY_t=\yy] \geq y_{k,\sigma}+\alpha/n$.
\smallskip

The PID property holds for the following reasons.
Since $g$ is differentiable on the whole of $\cC$, its gradient is bounded. 
Let 
\[\hat{g} = \max_{k,\sigma,\yy}|g'_\sigma(\yy)|\] 
and assume that $\yy$ has $y_{k,\sigma} \leq \delta$.
If $\bY_t=\yy$, then, looking at the exact transition probabilities (\ref{eq:proj_block_trans_prob}) for $\cY$, 
the probability that $y_{k,\sigma}$ decreases in the next step is at most $y_{k,\sigma} \leq \delta$.
The probability of an increase is at least $(r_k-\delta) e^{-4\hat{g}}/(q+1)$.
PID follows.

Important consequences of pointwise inward drift (PID) are:
\begin{itemize}
\item {\bf The convex hull property:} There exists an $\ep_0>0$ such that whenever $\ep \leq \ep_0$ and $\yy \in L^{(\ep)}$, then $\yy^*:=\yy-\tfrac{1}{2n}\ff(\yy)$ is in the convex hull 
of $L^{(\ep)}$.

This property follows from PID, since for any $\ep<\ep_0$ and since each step of $\cY$ moves no coordinate by more than $1/n$, 
the only points $\yy \in L^{(\ep)}$ for which $\yy^*$ could be outside $L^{(\ep)}$ are those with a coordinate equal to $\ep$.
For such a point and such a coordinate $y_{k,\sigma}$, $y^*_{k,\sigma} \geq \ep+\alpha/n$, which implies that $\yy^*$ must be in $L^{(\ep)}$ as $y^*_{\ell,j}$ is obviously also at least $\ep$ if $y_{\ell,j} > \ep$. 

By the simplex form of $L^{(\ep)}$, the convex hull property automatically means that $\yy^*$ is in the convex hull of the set $A$ of those points of $L^{(\ep)}$ that are within $L^\iy$-distance $1/n$ of $\yy^*$. 

\item {\bf $\cY$ can be censored:} By mimicking the proof of Lemma \ref{lemma:to_small_box_in_linear_time} and using Lemma \ref{lemma:no_escape_with_negative drift}, it follows from PID that for $\ep$ sufficiently small and $Y_0(k,\sigma) < \ep$, $Y_t(k,\sigma)$ crosses $3\ep$ within $O(n)$ and remains above $2\ep$ for an exponential time.
Since there are only finitely many coordinates, the whole process enters $L^{(3\ep)}$ in linear time and stays in $L^{(2\ep)}$ for an exponential time.

\end{itemize}

\noindent {\bf Censoring under the convex hull property:} Pick $\ep<\ep_0$ as in the convex hull assumption and define $\tilde{\cY}=\tilde{\cY}^{(\ep)}$ as follows.
\begin{itemize}
    \item From any position $\yy \in L^{(\ep)}$ where there are possible steps of $\cY$ that lead outside $L^{(\ep)}$, make a jump that lands in one of the points of $A$ according to a probability distribution that gives $\E[\tilde{\bY}_{t+1}|\tilde{\bY}_t=\yy]=\yy^*$, i.e.\ whose drift is the same vector field as for $\cY$.
    \item From a position $\yy$ from which $\cY$ cannot jump outside $L^{(\ep)}$, use the same transition probabilities for $\tilde{\cY}$.
\end{itemize}
Since $\tilde{\cY}$ and $\cY$ have the same drift field, $\tilde{\cY}$ inherits Lyapunov properties from $\cY$.

By the PID property, the mixing times for $\tilde{\cY}$ and $\cY$ differ by at most $O(n)$.
Hence we may identify $\cY$ with $\tilde{\cY}$ and focus on the latter,
keeping in mind that we need to add an $O(n)$ term to the mixing time result in the end.
For convenience, we allow ourselves to abuse notation by dropping the $\sim$ while doing so.
By Lemma \ref{lemma:proj_block_measure_lyapunov_on_compact_interiors}, Theorem \ref{thm:main_theorem} or Theorem \ref{thm:exponential_mixing_time} applies, 
according to the stationary-point structure of $h$,
provided that the Jacobian $\nabla \ff(\aaa)$ of $\ff$ at $\aaa$ is of suitable form.
However, the linearization of $\ff$ around $\aaa$ immediately gives
\[\nabla \ff(\aaa) = 2B(\aaa)\cH(\aaa),\]
which is covered by Remark \ref{remark:Sigma*H}.
If $\aaa$ is the unique local minimizer, the mixing time is thus
\[\taum(\cY;\ep) = \taum(\tilde{\cY};\ep)+O(n) = \frac{1}{2\lambda_{1}(B(\aaa)\cH(\aaa))}n\log n + O(n)\]
by Theorem \ref{thm:main_theorem}, 
whereas it is at least exponential in the presence of another nondegenerate interior local minimizer.

Before summarizing, we make the following helpful observation:
\begin{lemma} \label{lemma_BE=I}
The Hessian $\cE=\cE(\aaa)$ of the entropy part $e(\yy)=\sum_{k=1}^d\sum_{\si=0}^q y_{k,\si}\log y_{k,\si}$ of $h$, is the inverse of $B$.
\end{lemma}

\begin{proof}
Taking the second derivatives gives
\[\cE=\cE(\aaa)=\mathrm{blockdiag}(E_1,\ldots,E_d),\]
where 
\[E_k = \mathrm{diag}\left(\frac{1}{\aaa_k}\right)+\frac{1}{a_{k,0}}\one\one^\top.\]
It is then straightforward to compute the product $B\cE$ and verify that $B\cE=I$.
\end{proof}

An immediate consequence of the lemma is that $\lambda_1(B\cH)=1+\lambda_1(B\cG)$. The latter term is usually negative.

We summarize the above conclusions in the following theorem.

\begin{theorem} \label{thm:projected_block_measure}
    Let $\cZ$ be Glauber dynamics for a block measure $\nu^{(g)}$, let $\cY$ be the projection of $\bZ$ onto $\cC$ and let $h$ be the free energy corresponding to $g$. Assume that 
    \begin{itemize}
        \item $h$ has a unique minimizer $\aaa \in \cC^\circ$ with $\cH=\cH(\aaa)$ positive definite,
        \item $\n \nabla h(\yy) \n$ is bounded away from $0$ on each $B(\aaa,\kappa)^c$, $\kappa>0$.
    \end{itemize}
    Write $\aaa_k=(a_{k,1},\ldots,a_{k,q})^\top$,
    \[B_k = \mathrm{diag}(\aaa_k)-\frac{1}{r_k}\aaa_k\aaa_k^\top\]
    and
    \[B=\mathrm{blockdiag}(B_1,B_2,\ldots,B_d).\]
    Then the mixing time of $\cY$ satisfies
    \[\taum(\cY;\ep)=\frac{1}{2\lambda_{1}(B\cH)}n\log n + O(n) =
    \frac{1}{2(1+\lambda_{1}(B\cG))}n\log n + O(n).\]
    In contrast, if $h$ has more than one nondegenerate interior local minimum, then for some $c>0$ and every $\ep < 1/2$,
    \[\taum(\cY;\ep) \geq e^{cn},\,\, \taum(\cZ;\ep) \geq e^{cn}.\]
\end{theorem}
Thus, in the unique minimizer setting of the theorem, the mixing time depends only on the smallest eigenvalue of the matrix $B\cH$, which plays the role of the effective curvature of the free energy at equilibrium.

\medskip

\noindent {\bf Lifting from the projected chain to the full chain}

\smallskip

Assume that the conditions of Theorem \ref{thm:projected_block_measure} hold.
Since $\cY$ is the projection of $\cZ$, the mixing time of $\cY$ only gives a lower bound for that of $\cZ$.
If the starting configuration of $\cZ$ is one where all vertices have the same spin, then at any point $t$ in time, given the number of vertices of
each possible spin $\sigma$, 
the vertex sets $A_{t,\sigma} := \{v \in V:\bZ_t(v)=\sigma\}$, are uniformly distributed among all partitions of $V$ into subsets of the given sizes.
Hence the total-variation distance of $\cL(\bZ_t)$ from its stationary distribution equals that of $\cL(\bY_t)$ from its stationary distribution.  
Thus the mixing time of the two processes agree.
This no longer holds for other starting configurations.
However, the block-measure framework holds an opportunity that addresses
this issue in a clean way, as follows.


The key observation is that the framework allows us to divide each block into sub-blocks.
To describe this, assume for simplicity that $d=1$ for the time being. 
Then, concretely, we partition $V$ into the parts $V_1,\ldots,V_\ell$, and replace $\cY$ with the process $\hat{\cY}$ that projects $\bZ_t$ onto the process $\hat{\bY}_t$ that counts the number of vertices of each spin in each $V_j$.
Compared with $\bY_t$, each $Y_t(\si)$ then gets replaced in $\hat{\bY}_t$
by $(\hat{Y}_t(1,\si),\ldots,\hat{Y}_t(\ell,\si))$, where $\hat{Y}_t(j,\si)$
is the number of vertices of $v \in V_j$ with $Z_t(v)=\si$.

The process $\hat{\cY}$ thus created is simply a refined projection of $\cZ$ and fits into the above block-measure framework. 
If $g$ is the potential under the $\cY$-projection,  the potential of $\hat{\cY}$ becomes 
\[\hat{g}(y_{1,1},\ldots,y_{1,q},\ldots,y_{\ell,1},\ldots,y_{\ell,q})
= g\left(\sum_{j=1}^\ell y_{j,1},\ldots,\sum_{j=1}^\ell y_{j,q}\right),\]
and its corresponding free energy is given by
\[\hat{h}(\hat{\yy}) = 
\sum_{j=1}^\ell\sum_{\si=0}^q y_{j,\si}\log y_{j,\si}
+\hat{g}(\hat{\yy})\]
with the usual convention $y_{j,0}=s_j-\sum_{\si=1}^qy_{j,\si}$, where $s_j$ is the sub-block-$j$ proportion $|V_j|/n$.
The entropy term in $\hat{h}$ is uniquely minimized by distributing each spin proportion among the sub-blocks proportionally to their sizes. 
It follows that the stationary points of $\hat{h}$ are precisely these: if \(y\) is stationary for \(h\), then the corresponding stationary point, $\hat{\yy}$, of $\hat{h}$ is given by the vector identities
\[ \hat{\yy}_{j}=s_{j}\yy,\, j=1,\ldots,\ell.\]
Consequently, if \(h\) has the unique stationary point and minimizer $\aaa$, then \(\hat{h}\) has the unique stationary point and minimizer $\hat{\aaa}$ given by
\[ \hat{\aaa}_{j}=s_{j}\aaa,\, j=1,\ldots,\ell. \] 
Since the entropy gradient becomes unbounded at the boundary of the refined state space, 
compactness implies that \(\|\nabla\hat{h}\|\) is bounded away from \(0\) outside every neighborhood of \(\hat{\aaa}\). 
Thus, \(\hat{h}\) satisfies the corresponding assumptions of Theorem \ref{thm:projected_block_measure} whenever \(h\) does. 

Next, we need to understand what the matrices $\hat{B}(\hat{\aaa})$ and $\hat{\cG}(\hat{\aaa})$ corresponding to $B(\aaa)$ and $\cG(\aaa)$ are.
Each $V_j$ forms a new block, so $\hat{B}(\hat{\aaa})$ is a block-diagonal matrix 
of $\ell$ matrix elements, of which the $j$'th element becomes $s_j\mathrm{diag}(\aaa)-\tfrac{1}{s_j}(s_j\aaa)(s_j\aaa)^\top = s_j B(\aaa)$, $j=1,\ldots,\ell$.
Hence
\[\hat{B}(\hat{\aaa})= \mathrm{blockdiag}(s_1 B(\aaa),\ldots,s_\ell B(\aaa)).\]
Taking second derivatives of $\hat{g}$ gives
\[ \hat{\cG}(\hat{\aaa}) = 
\begin{pmatrix} 
\cG(\aaa) & \cG(\aaa) & \cdots & \cG(\aaa) \\ 
\cG(\aaa) & \cG(\aaa) & \cdots & \cG(\aaa) \\ 
\vdots & \vdots & \ddots & \vdots \\ 
\cG(\aaa) & \cG(\aaa) & \cdots & \cG( \aaa) 
\end{pmatrix}. 
\]
Taking the product and suppressing the arguments gives
\[ \hat{B}\hat{\cG} = 
\begin{pmatrix} 
s_1B\cG & s_1B\cG & \cdots & s_1B\cG \\ 
s_2B\cG & s_2B\cG & \cdots & s_2B\cG \\ 
\vdots & \vdots & \ddots & \vdots \\ 
s_\ell B\cG & s_\ell B\cG & \cdots & s_\ell B\cG 
\end{pmatrix}. \]
This $\ell q \times \ell q$-matrix has a spectrum that contains all eigenvalues of $B\cG$,
since if $\vv$ is an eigenvector for $B\cG$, then $[s_1\vv; \ldots;s_\ell\vv]$ is an
eigenvector of $\hat{B}\hat{\cG}$ for the same eigenvalue.
All remaining eigenvalues of $\hat{B}\hat{\cG}$ are zero.
Hence $\lambda_1(\hat{B}\hat{\cG}) = 0 \wedge \lambda_1(B\cG)$, so whenever the smallest eigenvalue of $B\cG$ is nonpositive, it follows from Theorem \ref{thm:projected_block_measure} that
\[\taum(\hat{\cY};\ep) = \frac{1}{2(1+\lambda_1(B\cG))}n\log n + O(n).\]
The analysis extends to where the underlying block measure has more than one block, 
when each block $B_k(\aaa_k)$ in $B(\aaa)$ is divided into $\ell$ sub-blocks of relative proportions $s_{k,1},\ldots,s_{k,\ell}$ within the given block (i.e.\ $\sum_j s_{k,j}=1$).
The preceding entropy argument applies separately within each original block. 
Hence the stationary points of $\hat{h}$ are precisely the proportional refinements of those of \(h\); in particular, its unique stationary point and minimizer $\hat{\aaa}$ is given by the vector identities 
\[ \hat{\aaa}_{k,j}=s_{k,j}\aaa_k, \, k=1,\ldots,d, \, j=1,\ldots,\ell,\]
and $\hat{h}$ satisfies the corresponding assumptions of Theorem \ref{thm:projected_block_measure} whenever $h$ does.
The same matrix calculation as with a single original block applies blockwise to $\hat{\cG}(\hat{\aaa})$, and 
\[\hat{B_k}(\hat{\aaa}) = \mathrm{blockdiag}\left(s_{k,1}B_k(\aaa),\ldots,s_{k,\ell}B_k(\aaa)\right).\]  
The resulting $\hat{B}\hat{\cG}$ has a spectrum that contains all eigenvalues of $B\cG$, plus additional zeros, and the same mixing-time result holds. 

To apply this for lifting from $\cY$ to $\cZ$, divide each block into $q+1$ sub-blocks, or {\em baskets}, where the $j$'th basket of the block consists of the vertices of that block that start with spin $j$.
Then, by the same symmetry as for the extremal measures, mixing for the 
corresponding projection $\hat{\cY}$ occurs precisely when mixing occurs for $\cZ$: at time $t$, conditional on the number of spins of each value within a basket, 
the allocation of vertices in that basket to the spin values is uniform by symmetry.

This is not quite yet a bound of the worst-case mixing time of $\cZ$, though, since the application of Theorem \ref{thm:projected_block_measure} presupposes that the worst-state basket proportions are independent of $n$.
Another complication is that one or more baskets may be empty or near empty at start.
The latter problem, however, is taken care of by the above censoring arguments: at the cost of an extra term $O(n)$ in the mixing-time upper bound, $\cY$ can be regarded as confined to $B(\aaa,\kappa)$ for arbitrarily small fixed $\kappa>0$, which gives baskets of macroscopic size 
in $\hat{\cY}$.
More precisely,
starting from an arbitrary configuration, first run $\cZ$ until $\cY$ enters $B(\aaa,\kappa/2)$. Whp, this takes time $O(n)$.
At that time, create $\hat{\cY}$ by partitioning each block according to the current spins and regard the resulting spin classes as baskets. 
Since $\aaa$ is an interior point of $C$, if $\kappa$ is sufficiently small, every such basket contains at least $cn$ vertices for some $c>0$.
The strong Markov property and the censoring argument therefore reduce the analysis to refined block measures whose basket proportions range over a compact set bounded away from the boundary.

To address the former obstacle, we need a block-measure framework that allows for block proportions that vary with $n$.
What that in turn means for the underlying MFGOU framework, is that 
it has to be extended to drift fields that depend on $n$: $\ff=\ff^{(n)}$.
Suppose that the $\ff^{(n)}$ are uniformly $C^2$ with
\[\ff^{(n)}(\yy) = \cJ^{(n)}(\yy-\yy^{(n)}_0) + O(\n \yy-\yy^{(n)}_0\n^2)\]
where the quadratic remainder is uniform in $n$ and the matrices $\cJ^{(n)}$ are uniformly positive definite. 
If the Lyapunov assumptions also hold uniformly in $n$, then the proofs of Theorem \ref{thm:main_theorem} 
and Theorem \ref{thm:exponential_mixing_time} apply without change to such $n$-dependent families, 
with $\cJ$ replaced by $\cJ^{(n)}$ in Theorem \ref{thm:main_theorem}.
The corresponding $n$-dependent version of Theorem \ref{thm:projected_block_measure} follows with $n$-dependent quantities $g$, $e$, $h$, $B$, $\cH$ and $\cG$.

In the present case, let $u^{(n)}_{k,\sigma}$ denote the starting proportion of all vertices that belong to block $k$ and have spin $\sigma$, and set
\[u_{k,0}^{(n)}=r_k - \sum_{\si=1}^q u^{(n)}_{k,\si}.\]
Suppose that $\uu^{(n)}=(u^{(n)}_{k,\sigma})_{k,\sigma}$ ranges over a compact neighborhood of $\aaa$. 
The corresponding relative spin proportions within block $k$, i.e.\ the basket proportions in the refinement, are $s^{(n)}_{k,\si} = u^{(n)}_{k,\si}/r_k$.
These proportions generate $n$-dependent $\hat{B}^{(n)}$ and $\hat{\cH}^{(n)}$ as above.   
However, as there, if $\lambda_1(B\cG) \leq 0$, then the smallest eigenvalue of their product
agrees with that of $B\cH$, and is thus independent of $n$.
By the preceding arguments and since $\uu^{(n)}$ varies in a compact neighborhood of $\aaa$, the corresponding free energy $\hat{h}^{(n)}$ has a gradient that is bounded away from $0$ uniformly in $n$ on all $B(\hat{\aaa}^{(n)},\kappa)^c$. 
The mixing-time bound that follows from the $n$-dependent version of Theorem \ref{thm:projected_block_measure} then has a leading coefficient that is independent of $n$, and agrees with that of $\cY$.
Before summarizing, notice that since $B\cG$ is similar to $B^{1/2}\cG B^{1/2}$ and $B$ is positive definite, we have $\lambda_1(B\cG) \leq 0$ if and only if $\lambda_1(\cG) \leq 0$.
We arrive at:

\begin{corollary} \label{corollary:block_measure_under_concave_g}
    Under the unique-minimum assumptions of Theorem \ref{thm:projected_block_measure}, suppose in addition that 
    $\lambda_1(\cG) \leq 0$. 
    Then
    \[\taum(\cZ;\ep)=\frac{1}{2\lambda_{1}(B\cH)}n\log n + O(n)
    =\frac{1}{2(1+\lambda_{1}(B\cG))}n\log n + O(n).\]
\end{corollary}

The refinement technique used above is closely related to the basket-wise coupling method of Cuff et al \cite{Cuff}. The block-measure framework provides a particularly direct implementation: the baskets may simply be regarded as additional blocks, after which the lifting follows from the same projected-chain theorem and an explicit spectral calculation.
The condition $\lambda_1(\cG) \leq 0$ is mild and holds for all attractive Ising/Potts model applications, and all mixture model and LDA applications that now follow.

\subsubsection{Mean-field Ising and Potts models} \label{subsub:Ising/Potts}
Let $G=(V,E)$ be an $n$-vertex graph. 
The $q+1$-state Potts model is a probability measure on the set of possible ways to assign each vertex of $G$ one of $q+1$ states (or {\em spins}) $0,1,\ldots,q$;
for each $\zz \in \{0,1,\ldots,q\}^V$
\[\nu(\zz) \propto \exp\left\{-\sum_{\sigma=0}^q M_\sigma n_\sigma(\zz) - \sum_{0 \leq \sigma < \sigma' \leq q} J_{\sigma,\sigma'} e_{\sigma,\sigma'}(\zz)\right\},\]
where the $M_\sigma$'s are real external-field parameters and $J_{\sigma,\sigma'}$'s are nonnegative interaction parameters, respectively. Here $n_\si(\zz)$ denotes the number of $v \in V$ with $z_v=\si$, while 
$e_{\si,\si'}$ denotes the number of edges $\{u,v\}$ for which $z_u=\si$ and $z_v=\si'$.
(This is not the most general formulation possible, but it certainly covers the applications in this paper.)
The Ising model is the special case $q=1$.

In this generality, the model does not fit into the block-measure framework, but with an appropriate block structure of the graph and of the parameters, numerous examples fit into the framework.
We will walk through a number of such examples.

\vspace{-0.4cm}
\paragraph{The Curie-Weiss Ising model}
This is the case where $q=1$ and $G$ is the complete graph. Letting $p_\sigma(\zz)$, $\sigma \in \{0,1\}$ be the proportion of vertices with spin $\sigma$, the model reads
\[\nu(\zz) \propto \exp\left\{-n(\alpha p_1(\zz) + \beta p_1(\zz)(1-p_1(\zz)))\right\},\]
with $\alpha$ being the external field and $\beta \geq 0$ the inverse temperature.
(This corresponds to $M_0=0$, $M_1=\alpha$ and $J_{0,1} = \beta/n$ in the general formulation; this is the standard scaling under which the model has a nondegenerate limit.)
The corresponding projected measure onto $L= [0,1] \cap \tfrac{1}{n}\Z$, writing $y=p_1(\zz)$, is
\[\pi(y) \propto \binom{n}{ny}\exp\left\{-n(\alpha y  + \beta y(1-y)) \right\}.\]

The setup sits squarely as an example of a block measure and its projection,  with $d=1$, $q=1$ and potential $g(y)=\alpha y + \beta y(1-y)$. 
Hence the model falls within the block-measure framework of Theorem \ref{thm:projected_block_measure}. 
The corresponding free energy is
\[h(y)=y\log y + (1-y)\log(1-y) + \alpha y + \beta y(1-y).\]
It is easy to show there is a critical value $\beta_c(\alpha)$, such that there is a unique and nondegenerate local minimizer for $\beta<\beta_c(\alpha)$, and more than one nondegenerate local minimizer for $\beta > \beta_c(\alpha)$.

For $\beta>\beta_c(\alpha)$, Theorem \ref{thm:projected_block_measure} tells us that the mixing time is at least exponential.
For $\beta<\beta_c(\alpha)$, let $a$ be the minimum of $h$, i.e.\ the solution of
\[\log\frac{y}{1-y}+\alpha+\beta(1-2y)=0.\]
Then $\cH=h''(a)= \tfrac{1}{a(1-a)}-2\beta$ and since $g$ is differentiable on $[0,1]$,
Theorem \ref{thm:projected_block_measure} gives for projected Glauber dynamics
\[\taum(\cY;\ep)=\frac{1}{2(1-2\beta a(1-a))}n\log n + O(n),\]
and Corollary \ref{corollary:block_measure_under_concave_g} gives also for the original Glauber dynamics
\[\taum(\cZ;\ep)=\frac{1}{2(1-2\beta a(1-a))}n\log n + O(n).\]
In the case $\alpha=0$, the symmetry around $a=1/2$ is helpful and one can easily compute that $\beta_c=2$,
which also agrees with the point where $h''(1/2)$ switches from positive to negative. 
(This is different when $\alpha \neq 0$, where $\beta_c(\alpha)$ is characterized by the appearance of a new stationary point elsewhere, while $h''(a)$ remains positive.)
Hence for $\beta<2$, up to an $O(n)$ term,
\[\taum(\cY;\ep) = \taum(\cZ;\ep) = \frac{1}{2-\beta}n\log n,\]
and for $\beta>2$, the mixing time is at least exponential.
This recovers (in our parametrization) the cutoff location obtained by Ding, Lubetzky and Peres \cite{DLP} and Levin, Luczak and Peres \cite{LLP}.
\begin{remark}
{\rm I am not aware of a previous statement of the resulting mixing-time formula for $\alpha\neq0$. } 
\end{remark}

\vspace{-0.4cm}
\paragraph{The Curie-Weiss Potts model}
As for the Curie-Weiss Ising model, $G$ is the complete graph, but $q>1$.
It is uncommon to go beyond the standard case of zero external field and a common interaction parameter $\beta$, and we shall not do so here.
We have
\[g(\yy) = \beta\left(\sum_{0 \leq \si < \si' \leq q}y_\si y_{\si'}\right),\]
where we keep in mind that $y_0=1-\sum_{\si=1}^q y_\sigma$, and hence
\[h(\yy) = \sum_{\si=0}^q y_\si\log y_\si + \beta\left(\sum_{0 \leq \si < \si' \leq q}y_\si y_{\si'}\right).\]
The mixing-time behavior of the Curie--Weiss Potts model was analyzed by Cuff et al.\ \cite{Cuff}. (For recent sharp mixing-time asymptotics in the low-temperature regime, see Kim and Lee \cite{KimLee2026}.)

By symmetry, $\aaa=\tfrac{1}{q+1}\one$ is a stationary point of $h$.
There is a critical value
$\beta_c=\beta_c(q)$ such that this is the unique and nondegenerate local minimizer for $\beta<\beta_c$, whereas for $\beta>\beta_c$, there are distinct nondegenerate local minimizers.
We have
\[B(\aaa)=B=\mathrm{diag}(\aaa)-\aaa\aaa^\top = \frac{1}{q+1}\left(I-\frac{1}{q+1}\one\one^\top\right),\]
\[\cG(\aaa)=\cG=-\beta(I+\one\one^\top).\]
Multiplying and recalling that $(\one\one^\top)^2=q\one\one^\top$ gives
\[B\cG=-\frac{\beta}{q+1}I.\]
Hence, for $\beta<\beta_c(q)$,
\[\taum(\cY;\ep)=\frac{1}{2\left(1-\frac{\beta}{q+1}\right)}n\log n + O(n) 
=\frac{q+1}{2(q+1-\beta)}n\log n + O(n),\]
while for $\beta>\beta_c(q)$, $\taum(\cY;\ep) \geq e^{\Omega(n)}$ when $\ep < 1/2$.
This is in agreement with \cite{Cuff}. Corollary \ref{corollary:block_measure_under_concave_g} lifts this back to unprojected Glauber dynamics $\cZ$:
\[\taum(\cZ;\ep)=\frac{q+1}{2(q+1-\beta)}n\log n + O(n).\]

Cases with nonzero external field are also covered by Theorem \ref{thm:projected_block_measure}, but since the value of $\aaa$ is intractable, 
we will not be able to state anything that is essentially more informative than Theorem \ref{thm:projected_block_measure} and Corollary \ref{corollary:block_measure_under_concave_g} already are. 

\vspace{-0.4cm}
\paragraph{The Ising model on complete multipartite graphs}
Let the positive numbers $r_1,\ldots,r_d$ sum to one, and let $G$ be the complete $d$-partite graph on $n$ vertices whose blocks have sizes $nr_k$, $k=1,\ldots,d$.
Thus, each vertex is adjacent to every vertex outside its own block, and there are no intra-block edges.
Consider the Ising model on $G$ with inverse temperature $\beta$ and without external field:
\[\nu(\zz) \propto  e^{-ng(\yy)} := \exp\left\{-\beta n\sum_{k=1}^d\sum_{\ell \neq k} y_k(r_\ell-y_\ell)\right\},\]
where $y_k$ is the proportion of vertices in block $k$ that are assigned spin $1$ in  $\zz$.
The corresponding projected measure on $([0,r_1] \times \ldots \times [0,r_d]) \cap \tfrac{1}{n}\Z^d$ is 
\[\pi(\yy) \propto \left(\prod_{k=1}^d \binom{nr_k}{ny_k}\right)e^{-ng(\yy)}.\]
Its associated free energy is 
\[h(\yy)=\sum_k (y_k\log y_k + (r_k-y_k)\log(r_k-y_k))+g(\yy).\]
This is clearly an example of a block measure and its projection in the above framework with $q=1$.
We analyze this model as follows.
We compute the matrix $B\cG$, whose smallest eigenvalue determines the mixing-time constant via Theorem \ref{thm:projected_block_measure}.
We then evaluate this eigenvalue and treat the special cases $d=2$, equal block sizes, and an extension to regular base graphs.
Finally we summarize the mixing-time behavior in Propositions \ref{prop:complete_d-partite_graphs} and \ref{prop:complete_Gb-partite_graphs}.

By symmetry, the point $\aaa:=\tfrac{1}{2}\bm{r}$ is a stationary point.
For $\beta$ below some critical value $\beta_c>0$, it is also the unique stationary point and the global minimum of $h$, whereas for $\beta>\beta_c$, $h$ has multiple nondegenerate local minimizers, in which case the mixing time is at least exponential.
Assume that $\beta<\beta_c$.
Then, in the above framework, we get $B=\tfrac{1}{4}\mathrm{diag}(\bm{r})$ and 
\[\cG=2\beta(I-\one\one^\top),\]
which results in
\[B\cG=\frac{1}{2}\beta(\mathrm{diag}(\bm{r})-\bm{r}\one^\top).\]
Using the matrix determinant lemma, one obtains the characteristic polynomial
\[\mathrm{det}(\mathrm{diag}(\bm{r})-\bm{r}\one^\top-\lambda I) = \left(1-\sum_{k=1}^d\frac{r_k}{r_k-\lambda}\right) \prod_{k=1}^d (r_k-\lambda),\]
whose smallest genuine root is the unique negative solution $\lambda_*$ of
\[1-\sum_{k=1}^d\frac{r_k}{r_k-\lambda}=0.\]
\begin{remark}
 {\rm All roots of the left-hand factor are genuine eigenvalues of $\mathrm{diag}(\bm{r})-\bm{r}\one^\top$, 
 whereas the roots $r_\ell$ of the right-hand side product are canceled by poles of the left-hand factor whenever they are unique among the $r_k$'s. 
 Indeed, the multiplicity of $r_\ell$ as a root of the characteristic polynomial equals $|\{1 \leq k \leq d:r_k=r_\ell\}|-1$.}
\end{remark}
 
Theorem \ref{thm:projected_block_measure} gives a mixing time 
\[\frac{1}{c} n\log n + O(n)\] 
with $c=2+\lambda_* \beta$.

For $d \geq 5$, no explicit expression for $\lambda_*$ exists in general, while for $d=3,4$ such expressions exist but are rather uninformative.

For $d=2$, $\lambda_* = -\sqrt{r_1r_2}$ and hence 
\[c=2-\sqrt{r_1r_2}\beta.\]

In the case $r_k=1/d$, $\lambda_*=-(d-1)/d$ and thus 
\[c=2-\beta(d-1)/d.\] 
(This agrees well with the Curie-Weiss case, since the larger the $d$, the closer the model is to the Curie-Weiss Ising model.)

\smallskip

When $r_k=1/d$, the same computation extends directly to a broader class of interaction graphs.
Let $G_b=(V_b,E_b)$ be a $\delta$-regular graph on $d$ vertices and construct the $n$-vertex complete $G_b$-partite graph $G$ as follows.
Partition $[n]$ into $d$ equal-sized blocks and create a one-to-one map from the set of blocks to the vertices of $G_b$.
Then let the edge set of $G$ consist of all pairs of vertices whose respective blocks are mapped to neighboring vertices of $G_b$.

For the Ising model on $G$, we get
\[g(\yy) = \beta \sum_{k=1}^d\sum_{\ell:\{k,\ell\} \in E_b} y_k\left(\frac{1}{d}-y_\ell\right).\]
Then $\aaa=\tfrac{1}{2d}\one$ and $B=\tfrac{1}{4d}I$.
For $\cG$, we get $G_{k,\ell} = -2\beta$ if $\{k,\ell\} \in E_b$ and $0$ otherwise.
This makes $B\cG$ into $-\beta\delta/2d$ times a stochastic matrix, and consequently the smallest eigenvalue is $-\beta\delta/2d$.
Theorem \ref{thm:projected_block_measure} applies and gives 
\[c=2-\delta\beta/d.\]

Before stating the two resulting propositions, we make the following useful observation:

\begin{lemma} \label{lemma:H_pd_everywhere}
In each of the setups above, letting $\cH(\yy)$ denote the Hessian of $h$ at an arbitrary point and $\aaa:=\tfrac{1}{2}\bm{r}$, $\cH(\yy)$ is positive definite for all $\yy$ whenever $\cH(\aaa)$ is positive definite.  
\end{lemma}

\begin{proof}
Note first that $\cG(\yy)$ is independent of $\yy$: $\cG(\yy) = \cG(\aaa)$.
Note also that $\cE(\yy)=\mathrm{diag}(r_k/y_k(r_k-y_k))$ and, notably, 
that each diagonal element is minimized when $\yy=\aaa=\tfrac{1}{2}\mathbf{r}$.
Hence $\cH(\yy)$ is the sum of $\cH(\aaa)$ and a positive semidefinite matrix, and is therefore positive definite for every $\yy$ whenever $\cH(\aaa)$ is positive definite.
\end{proof}

In each case, we have identified a critical value $\beta_c$ at which $\mathcal H$ ceases to be positive definite and $h$ develops multiple local minima. Lemma \ref{lemma:H_pd_everywhere}, together with Theorem \ref{thm:projected_block_measure} and Corollary \ref{corollary:block_measure_under_concave_g}, 
therefore shows that $\beta_c$ is also the critical value for the mixing-time transition for $\cY$ as well as $\cZ$.

\smallskip

We summarize these results in the form of two propositions.

\begin{proposition} \label{prop:complete_d-partite_graphs}
    Let $r_1,\ldots,r_d$ be positive numbers whose sum is 1, let $G$ be the complete $d$-partite graph with block sizes $nr_1,\ldots,nr_d$, let $\lambda_*$ be the unique real negative solution of
    \[1-\sum_{k=1}^d\frac{r_k}{r_k-\lambda}=0,\]
    and set $\beta_c(\bm{r}) = -2/\lambda_*$.
    Consider Glauber dynamics $\cZ$ and its corresponding projection $\cY$ for the Ising model on $G$ with inverse temperature $\beta$. Then, if $\beta>\beta_c(\bm{r})$, $\taum(\cZ;\ep) \geq \taum(\cY;\ep) = e^{\Omega(n)}$ for $\ep < 1/2$. 
    If $\beta<\beta_c(\bm{r})$, then for any $0<\ep<1$,
    \[\taum(\cZ;\ep)=\taum(\cY;\ep)=\frac{1}{2+\lambda_*\beta}n\log n + O(n).\]
    In particular: 
    \begin{itemize}
    \item[(i)] If $d=2$, then $\beta_c(\bm{r}) = 2/\sqrt{r_1r_2}$ and for $\beta<\beta_c(\bm{r)}$,
    \[\taum(\cZ;\ep)=\taum(\cY;\ep)=\frac{1}{2-\sqrt{r_1r_2}\beta}n\log n + O(n).\]
    \item[(ii)] If $\bm{r}=\tfrac{1}{d}\one$, then $\beta_c(d):=\beta_c(\bm{r})=2d/(d-1)$ and whenever $\beta<\beta_c(d)$,
    \[\taum(\cZ;\ep)=\taum(\cY;\ep)=\frac{1}{2-\frac{d-1}{d}\beta}n\log n + O(n).\]
    \end{itemize}
\end{proposition}

\begin{proposition} \label{prop:complete_Gb-partite_graphs}
Let $G_b$ be a $\delta$-regular $d$-vertex graph and let $G$ be the complete $n$-vertex $G_b$-partite graph with equal blocks.
Let $\cZ$ be Glauber dynamics for the Ising model with inverse temperature $\beta$ on $G$ and let $\cY$ be its projection.
Then, if $\beta>2d/\delta$, the mixing time of both processes is $e^{\Omega(n)}$ for any $\ep<1/2$. 
If $\beta<2d/\delta$, then for any $0<\ep<1$,
\[\taum(\cZ;\ep)=\taum(\cY;\ep)=\frac{1}{2-\frac{\delta}{d}\beta}n\log n + O(n).\]
    
\end{proposition}

For the Ising model on complete multipartite graphs, cutoff in the high-temperature regime and exponentially slow mixing in the low-temperature regime were established by Kim \cite{Kim2022}. Proposition \ref{prop:complete_d-partite_graphs} recovers the cutoff location through the smallest eigenvalue of the corresponding block matrix.
I am not aware of a previous statement of the extension in Proposition \ref{prop:complete_Gb-partite_graphs} to interaction structures given by general regular base graphs.

\subsubsection{Latent Dirichlet Allocation and hierarchical models} \label{subsubsec:LDA}

The Latent Dirichlet Allocation (LDA) model of Blei, Ng and Jordan \cite{BNJ} is a topic model of iconic status in natural language processing.
LDA and several of its variants remain useful in the humanities and social sciences more than twenty years after its introduction.
LDA is a generative model designed for unsupervised inference of latent themes, or {\em topics}, 
in large corpora of text.
It represents each document as a mixture over $K$ latent topics, each of which is characterized
by a distribution over vocabulary words.

More precisely, consider a corpus of $D$ documents, $d=1,\ldots,D$, each consisting of $L_d$ tokens: $w_{d1},w_{d2},\ldots,w_{dL_d}$,
where each token is a word from a vocabulary of $V$ words: $1,2,\ldots,V$.
Assume that there exist $K$ different topics that are represented in the corpus.
Each topic, $k$, is represented by a probability vector $\bphi_k=(\phi_{k1},\ldots,\phi_{kV})$ over the vocabulary, 
and each document $d$ is represented by a probability 
vector $\btheta_d=(\theta_{d1},\ldots,\theta_{dK})$ over the set of topics.
The corpus is assumed to have been generated as follows: Independently for each $d$ and $j$, a topic $z_{dj}$ is first sampled according to $\theta_d$, after which $w_{dj}$ is sampled from the vocabulary according to $\bphi_{z_{dj}}$.

Hence, for each word in the entire corpus, there is a latent topic label ``behind'' it, and the goal is to infer the topic labels from the text.
Inference here means to sample from the conditional distribution of 
$\zz:=(z_{11},\ldots,z_{DL_D})$ given $\ww:=(w_{11},\ldots,w_{DL_D})$.
If $\Phi=(\bphi_1,\ldots,\bphi_K)$ and $\Theta:=(\btheta_1,\ldots,\btheta_D)$ are known, this is trivial; each topic label is conditionally independent and
\[P(z_{dj}=k|\ww) \propto \theta_{dk}\phi_{kw_{dj}}.\]
However, at least one of $\Phi$ and $\Theta$, and often both, is usually unknown.
In that case, LDA uses Bayesian priors on the unknown parameters.
Fix two vectors: $\balpha \in (0,\iy)^K$ and $\bbeta \in (0,\iy)^V$.
Then in the case of unknown $\Theta$, let each $\theta_d \sim Dir(\balpha)$, where $Dir$ stands for the Dirichlet distribution, independently over $d$, and if
$\Phi$ is unknown, let $\phi_k \sim Dir(\bbeta)$ independently over $k$.
In this way, the LDA framework encompasses several related hierarchical Bayesian models, depending on which of the latent quantities $\zz,\Phi,\Theta$ are treated as random. 
In the sequel, we shall consider:

\begin{itemize}
\item Multinomial mixtures with known class distributions, corresponding to $D=1$ with known $\Phi$.
\item LDA with either $\Theta$ or $\Phi$ known.
\item Full LDA where both $\Theta$ and $\Phi$ are unknown.
\end{itemize}

In each case, we will consider the mixing behavior of Gibbs sampling from the posterior of the topic/class labels, 
which in the first two cases means partially collapsed Gibbs sampling and in the third case collapsed Gibbs sampling 
(where collapsing means integrating out unknown quantities when computing the posterior).
These models form a natural hierarchy of increasing complexity and will allow us to isolate the mechanisms responsible for fast and slow mixing.

\smallskip

To compute the posterior of $\zz$ given $\ww$, for each $(d,v,k) \in [D] \times [V] \times [K]$, 
let $\kappa_{d,v,k}=\kappa_{d,v,k}(\zz)$ be the number of $j \in \{1,\ldots,L_d\}$ such that $z_{dj}=k$ and $w_{dj}=v$.
Use dot-notation for quantities where one or more indices have been summed out, i.e., $\kappa_{\bullet, v,k}=\sum_d \kappa_{d,v,k}$ and analogously with the dot or dots in other index positions.
Let
\begin{itemize}
    \item $N_{dk}=\kappa_{d,\bullet,k}$, the number of tokens in document $d$ assigned topic $k$,
    \item $M_{kv}=\kappa_{\bullet,v,k}$, the total number of tokens in the corpus with word $v$ assigned topic $k$,
    \item $M_k=\kappa_{\bullet,\bullet,k}$, the total number of tokens in the corpus assigned topic $k$.
\end{itemize}
Letting $q$ be the pmf for $(\zz,\ww)$, we get
\[q(\zz,\ww \mid \Theta,\Phi) = \prod_{d=1}^D \prod_{k=1}^K \theta_{dk}^{N_{dk}} \cdot \prod_{k=1}^K\prod_{v=1}^V \phi_{kv}^{M_{kv}}.\]
Hence 

\begin{equation} \label{eq:lda_posterior}
q(\zz \mid \ww) \propto q(\zz,\ww) = \prod_{d=1}^D \E\left[\prod_{k=1}^K \theta_{dk}^{N_{dk}}\right] \cdot \prod_{k=1}^K \E\left[\prod_{v=1}^V \phi_{kv}^{M_{kv}}\right]
\end{equation}
where the expectations are trivial if the corresponding quantity is known and higher moments of Dirichlet random probability vectors if unknown (or of course, albeit usually less tractable, the corresponding higher moments for any other desired prior distribution). 
One must then recall that for $\bX \sim Dir(\gamma_1,\ldots,\gamma_I)$,
\[\E\left[\prod_{i=1}^I X_i^{J_i} \right] =  C\frac{\prod_{i=1}^I\Gamma(\gamma_i+J_i)}{\Gamma\left(\sum_{i=1}^I \gamma_i+J_i\right)}\]
for a factor $C$ independent of the $J_i$'s.
We will be working in the regime $\gamma_i=n\delta_i$, with $n$ the total number of tokens in the corpus, with $V$ and $D$ fixed and the $L_d$'s growing linearly with $n$.
By Stirling's approximation and writing $J_i=nj_i$, this now becomes
\[\E\left[\prod_{i=1}^I X_i^{J_i} \right] \approx C\left(\frac{\prod_{i=1}^I (\delta_i+j_i)^{\delta_i+j_i}}{\left(\sum_{i=1}^I \delta_i+j_i \right)^{\sum_{i=1}^I \delta_i+j_i}}\right)^n.\]
This approximation is not sufficiently sharp to leave the posterior unchanged when substituted into its expression.
Indeed, approximating $\Gamma(an)$ by $(an/e)^{an}$ omits the factor 
\[\sqrt{\frac{2\pi}{an}}\left(1+\frac{1}{12an}+O\left(\frac{1}{n^2}\right)\right),\]
potentially perturbing the ratio of the probabilities of widely separated configurations by a constant factor.
However, single‑site Gibbs transition probabilities are only altered by a relative factor $1+O(1/n)$, uniformly over configurations. Such perturbations are immaterial for the mixing‑time analysis; they correspond to an $O(1/n^2)$ modification of the drift and noise terms and are therefore absorbed in the MFGOU framework (see (\ref{eq:trans_prob_in_terms_of_h}) and the definition of the basic MFGOU process). Consequently, although the stationary distribution is slightly altered, the resulting Markov chain has the same mixing‑time asymptotics.
We will therefore use the Stirling expression to identify the MFGOU drift, with all omitted terms absorbed into the $O(n^{-2})$ drift remainder.

\begin{remark} \label{remark:prior_parameters}
    {\rm (Choice of hyperparameters.) In practical applications of LDA, the Dirichlet hyperparameters $\balpha,\bbeta$ are typically fixed and do not scale with corpus size. 
    In contrast, for the purposes of our mixing-time analysis it is natural to consider the linear-growth regime 
    \[\alpha_k=n\delta_k,\; \beta_v=n\omega_v.\] 
    This is precisely the scaling under which the posterior log-density, after Stirling approximation, acquires a clean mean-field structure of the form $e^{-ng(\yy)}$ with $g$ differentiable on its domain, 
    placing the resulting Gibbs sampler in the MFGOU framework,
    whereas in the fixed hyperparameter setting, differentiability on the entire closed domain fails.
    The main complication is that if started from a boundary state, Gibbs sampling may be slow to enter the interior, while in many cases, starting from an interior point may still allow the MFGOU framework to be applied on a censored domain. 
    When applicable, this will be remarked upon in the applications below.} 
    
\end{remark}

\vspace{-0.4mm}
\paragraph{Multinomial mixture models.}

We begin with the case $D=1$ with known class distributions $\Phi=(\bphi_k)$. Each token is generated by first choosing a class $k \sim \btheta$ and then drawing a word according to $\bphi_k$. 
The class weights $\btheta$ are unknown and endowed with a Dirichlet prior $Dir(\balpha)$.
After applying the Stirling approximation to the Dirichlet moments of $\btheta$, writing
\[y_{v,k}=M_{kv}/n\]
and recalling that $\bdelta=\balpha/n$,
the approximate posterior takes the mean-field block-measure form
\[\nu(\zz) \propto \exp\{-n\,g(\yy)\},\]
where
\begin{align*}
g(\yy) &= -\left(\sum_{k=1}^K (\delta_k+y_{\bullet,k})\log(\delta_k+y_{\bullet,k}) + \sum_{k=1}^K\sum_{v=1}^V y_{v,k}\log\phi_{kv}\right).
\end{align*}
The denominator of the Stirling approximation has been omitted, since it equals $(\ell_1+\sum_k\delta_k)\log (\ell_1+\sum_k\delta_k) = (1+\sum_k\delta_k)\log (1+\sum_k\delta_k)$, and is thus independent of $\yy$.

\smallskip

To apply the block-measure machinery of this section, we analyze the free energy
\[h(\yy) = e(\yy)+g(\yy) = \sum_{k=1}^K\sum_{v=1}^V y_{v,k}\log y_{v,k} + g(\yy).\]
Under the assumption that there is a unique minimizer $\aaa$ of $h$ and that $\n\nabla h \n$ is bounded away from $0$ on each $B(\aaa,\kappa)$, $\kappa>0$, Theorem \ref{thm:projected_block_measure} gives
\[\taum(\cY;\ep) = \mu n\log n + O(n),\]
where $\mu^{-1}=2(1+\lambda_1(B\cG))$,
$B=\mathrm{blockdiag}(B_1,\ldots,B_V)$ and $B_v=\mathrm{diag}(\aaa_v)-\tfrac{1}{r_v}\aaa_v\aaa_v^\top$, where $r_v$ is the proportion of tokens in the document whose observed word is $v$. (Note that $k$ here plays the role of $\sigma$ and $v$ plays the role of $k$ in the general block-measure setup.)
Recall that $\yy$ is $(K-1)V$-dimensional, since 
\[y_{v,K}=r_v-\sum_{k=1}^{K-1} y_{v,k},\]
in accordance with the block-measure setup.

Differentiating $g$ twice and keeping the $y_{v,K}$ restriction in mind, we get for $v,w \in \{1,\ldots,V\}$, $k,\ell \in \{1,\ldots,K-1\}$,
\[\frac{\partial^2}{\partial_{(v,k)}\partial_{(w,\ell)}}g(\yy) = -\frac{1}{S_K}-\frac{1}{S_k}\one_{\{k=\ell\}},\]
where $S_k=\delta_k+y_{\bullet,k}$.
In particular, $\cG$ is negative semidefinite, and 
consequently, Corollary \ref{corollary:block_measure_under_concave_g} shows that the same mixing-time result holds for the unprojected Gibbs sampler $\cZ$.
\[\taum(\cZ;\ep) = \frac{1}{2(1+\lambda_1(B\cG))}n\log n + O(n).\]
It remains to investigate when $h$ satisfies the unimodality conditions. We claim that in fact, $\nabla^2 h(\yy)$ is {\em always uniformly positive definite in $\yy$}, so that the given mixing time result holds for all class distributions $\Phi$ and all hyperparameters $\delta_k>0$.
To that end, write
\[h(\yy) = \sum_{k=1}^K \left( \sum_{v=1}^V y_{v,k}\log y_{v,k} - \left(\delta_k+\sum_{v=1}^V y_{v,k}\right)
\log\left(\delta_k+\sum_{v=1}^V y_{v,k}\right) \right).\]
Regarding $h$ as a function of $KV$ variables restricted to $\one_K^\top\yy_{v,:}=r_v$ for all $v$, the smallest eigenvalue of the restricted Hessian is the minimum over quadratic forms $\ww^\top \hat{\cH}(\yy) \ww$ for the unrestricted Hessian $\hat{\cH}$ over
$\ww=(\ww_{1,:},\ldots,\ww_{V,:}) = (\ww_{:,1},\ldots,\ww_{:,K})$ under $\one_K^\top \ww_{v,:}=0$ and $\n\ww\n=1$.

Pick one of the inner $K$ terms arbitrarily and write for simplicity $\delta=\delta_k$, $\xx=y_{:,k}$, $\uu=\ww_{:,k}$ and $\psi(\xx)$
for the chosen term.
Then
\[\nabla^2 \psi(\xx) = \mathrm{diag}\left(\frac{1}{\xx}\right)-\frac{1}{\delta+s}\one\one^\top,\]
where $s=\sum_{v=1}^V x_v$.
Then for any $\uu \in \RR^V$,
\begin{align*}
\uu^\top (\nabla^2 \psi(\xx)) \uu &= \sum_v \frac{1}{x_v}u_v^2 - \frac{1}{\delta+s}\left(\sum_v u_v\right)^2.
\end{align*}
By Cauchy-Schwarz,
\[\sum_{v=1}^V\frac{u_v^2}{x_v} \geq \frac{1}{s}\left(\sum_{v=1}^V |u_v|\right)^2
= \frac{1}{s}\|\uu\|_1^2,\]
with equality if and only if $\uu$ is proportional to $\xx$. 
Moreover,
\[\left(\sum_{v=1}^V u_v\right)^2\leq \|\uu\|_1^2,\]
with equality if and only if all nonzero components of $\uu$ have the same sign. 
Consequently,
\[\uu^\top \nabla^2 \psi(\xx)\uu \geq \left(\frac{1}{s}-\frac{1}{\delta+s}\right)\|\uu\|_1^2
=\frac{\delta}{s(\delta+s)}\|\uu\|_1^2.\]
Equality holds if and only if $\uu$ is proportional to $\xx$.
Summing over $K$ gives
\[\ww^\top \hat \cH(\yy)\ww \geq \sum_{k=1}^K \frac{\delta_k}{s_k(\delta_k+s_k)} \|\ww_{:,k}\|_1^2 \geq \left(\min_{1\leq k\leq K}\frac{\delta_k}{1+\delta_k}\right) \sum_{k=1}^K\|\ww_{:,k}\|_2^2. \] 
Thus $\hat \cH(\yy)$ is uniformly positive definite. Its restriction to the subspace defined by 
\[\sum_{k=1}^Kw_{v,k}=0,\,\, v=1,\ldots,V\] 
is therefore also uniformly positive definite.


\smallskip



Hence, the given mixing time bounds always hold, whenever $\min_k\delta_k>0$.
In general, $\aaa$ and $B\cH:=B(\aaa)\cH(\aaa)$ are intractable, but there are natural cases that are more amenable. In particular, the formulas simplify when $K=2$, which we assume henceforth. 
Then, with $\yy:=\yy_{:,1}$, we get $\yy_{:,2}=\bm{r}-\yy$. Also take $\delta_1=\delta_2=\delta$.
With $\aaa$ being the minimizer of $h$, we then have 
\[B=\mathrm{diag}\left(a_1-\frac{a_1^2}{r_1},\ldots,a_V-\frac{a_V^2}{r_V}\right),\]
\[\cG= -\left(\frac{1}{\delta+\sum_v a_v}+\frac{1}{1+\delta-\sum_v a_v}\right)\one\one^\top.\]
This gives
\[\lambda_1(B\cG) = - \left(\frac{1}{\delta+s}+\frac{1}{1+\delta-s}\right)\left(s-\sum_v\frac{a_v^2}{r_v}\right),\]
where $s=\sum_v a_v$.
Since $\aaa$ is in the interior, $\lambda_1(B\cG) < 0$.
Moreover,
\[\lambda_1(B\cG) \geq -\frac{(1+2\delta)s(1-s)}{(\delta+s)(1+\delta-s)},\] 
with equality iff $\aaa=s\bm{r}$.
In summary, if $K=2$ and $h$ is minimized at $\aaa$, then up to $O(n)$ terms, with $s=\sum_v a_v$,
\[\taum(\cZ;\ep)=\taum(\cY;\ep) = \mu n\log n,\]
with
\[\mu^{-1} = 2\left(1-\left(\frac{1}{\delta+s}+\frac{1}{1+\delta-s}\right)\left(s-\sum_v\frac{a_v^2}{r_v}\right)\right).\]
It is still difficult to find $\aaa$, but consider the case where the mixture model is well-specified, i.e. there is a true underlying mixture measure $\theta^*$.
In this case, under the generative process,
\[\E[r_v]=\sum_k \theta^*_k\phi_{kv} = \theta^*\phi_{1v}+(1-\theta^*)\phi_{2v},\]
and the law of large numbers assures that the actual $r_v$'s are, with high probability, very close to their expectations for large $n$.
This makes it compelling to consider the ideal setting in which
\[r_v=\sum_k \theta^*_k\phi_{kv} = \theta^*\phi_{1v}+(1-\theta^*)\phi_{2v}.\]
Assuming that $\bphi_1$ and $\bphi_2$ are not identical, the only population allocation corresponding to the generating mixture is $y_v=y^*_v=\theta^*\phi_{1v}$.
For $\delta=0$, one confirms by inserting into the expression for $\nabla h$ that $\yy^*$ is the unique minimizer of $h$.
As $\delta$ increases from $0$, the minimizer moves away continuously from this point at rate $O(\delta)$. This makes it useful to investigate $\lambda_1(B(\yy^*)\cG(\yy^*))$ not only for $\delta=0$.
The given form of $\yy^*$ gives $s=\theta^*$ and after some manipulation, 
\begin{align*} 
\mu^{-1}(\delta) &:= 2\bigl(1+\lambda_1(B(y^*)G(y^*))\bigr) \\
&=\frac{2\left(\delta\bigl(\delta+1-2\theta^*(1-\theta^*)\bigr)+(1+2\delta)\bigl(\theta^*(1-\theta^*)\bigr)^2\sum_{v=1}^V\frac{(\phi_{1v}-\phi_{2v})^2}{r_v}\right)}{(\delta+\theta^*)(\delta+1-\theta^*)}.
\end{align*}
Since $h$ is $C^\iy$ on a neighborhood of $\yy^*$ and its Hessian is strictly positive definite there for $\delta=0$, the minimizer $\aaa(\delta)$ is Lipschitz in $\delta$ for $\delta$ small. It follows that up to an error of $O(n)$,
\[\taum(\cZ;\ep) = \taum(\cY;\ep) = (1+O(\delta))\mu(\delta)n\log n.\]
The second term of $\mu^{-1}(\delta)$ is obviously nonnegative (and zero iff $\bphi_1 = \bphi_2$, which we assumed is not the case).
If we set $\delta=0$, the first term vanishes, but $\mu^{-1}$ is still nonzero and the minimum of $h$ is unique.
Since the preceding argument shows that $\nabla h$ does not vanish on compact interior sets, $\cY$ satisfies the convex hull property, even though PID fails. 
This allows us to censor accordingly and consider the process $\hat{\cY}$ on any compact $K \subseteq \cC^\circ$. 
For $\hat{\cY}$ and the corresponding $\hat{\cZ}$ we get exactly 
\[\taum(\hat{\cZ};\ep) = \taum(\hat{\cY};\ep)+O(n) = \mu(0) n\log n+O(n).\]
Consequently, for either $\cZ$ or $\cY$ started from any point in $K$, the same mixing-time asymptotics hold. 
Only starting points at distance $o(1)$ from the boundary of $\cC$ may introduce an additional delay.

Finally, a few words on the case $\bphi_1=\bphi_2$. In this case there is no way to separate the two topics, i.e., \ there is no way to identify the unknown mixing probability $\theta^*$,
and $h$ is minimized for $\yy=s\bm{r}$ for all $s \in [0,1]$. Hence $h$ is flat on a line in its domain. 
This gives us strong reason to believe that 
the mixing time of $\cZ$ and $\cY$ in this case is $\Theta(n^2)$.

\begin{remark} \label{remark:continuous_mixture}
    {\rm If the multinomial probabilities arise from binning continuous mixture densities $\psi_1,\psi_2$, then as the bins are refined, the curvature term 
    \[\sum_{v=1}^V \frac{(\phi_{1v}-\phi_{2v})^2}{r_v}\]
    converges to 
    \[\int \frac{(\psi_1(v)-\psi_2(v))^2}{\theta^* \psi_1(v)+(1-\theta^*)\psi_2(v)}dv,\]
    provided that the densities are regular.
 The same mixing‑time formula then applies in the continuous case by replacing
 the sum with the integral.
    }
\end{remark}

\vspace{-0.4cm}
\paragraph{LDA with known topic-word distributions.}

The extension from $D=1$ to $D>1$ is straightforward, since $g$ is obtained by summing one term of the same form over the documents.
We have, with document lengths $n\ell_1,\ldots,n\ell_D$, $\sum_d\ell_d=1$,
\[g(\yy) = -\sum_{d=1}^D \left(\sum_{k=1}^K (\delta_k+y_{d,\bullet,k})\log(\delta_k+y_{d,\bullet,k}) + \sum_{k=1}^K\sum_{v=1}^V y_{d,v,k}\log\phi_{kv}\right),\]
\[e(\yy) = \sum_d\sum_v\sum_k y_{d,v,k}\log y_{d,v,k}.\]
Importantly, terms corresponding to different documents do not share variables. Hence the argument from the preceding subsection applies separately to each document and gives 
\[ \lambda_1(\mathcal H(\mathbf y)) \geq \min_{1\leq k\leq K}\frac{\delta_k}{1+\delta_k}>0. \] 
The Hessian $\cG$ is block diagonal across documents, and each block has the same negative-semidefinite form as in the multinomial-mixture model above. 
Thus $\cG$ is negative semidefinite, and Corollary \ref{corollary:block_measure_under_concave_g} applies.
Consequently, the mixing times of $\cY$ and $\cZ$ are both 
\[\mu n\log n+O(n)\] 
whenever $\min_k\delta_k>0$, 
where 
\[ \mu^{-1} = 2\left(1+\lambda_1(B(\mathbf a)\cG(\mathbf a))\right). \]


For $K=2$ with all $\delta_k$ equal, 
\[B=\mathrm{diag}\left(a_{1,1}-\frac{a_{1,1}^2}{r_{1,1}},\ldots,a_{D,V}-\frac{a_{D,V}^2}{r_{D,V}}\right)\]
and $\cG$ is block-diagonal, where block $d$ equals
\[\cG_d = -\left(\frac{1}{\delta+s_d}+\frac{1}{\delta+\ell_d-s_d}\right)\one_V\one_V^\top,\]
where $s_d=\sum_v a_{d,v}$.
This gives
\[\lambda_1(B\cG) = - \max_d\left(\frac{1}{\delta+s_d}+\frac{1}{\delta+\ell_d-s_d}\right)\left(s_d-\sum_v\frac{a_{d,v}^2}{r_{d,v}}\right),\]
and then, as usual $\mu^{-1}=2(1+\lambda_1(B\cG))$.

Returning to general $K$, consider the case in which each document has a true topic-probability vector $\btheta^*_d$. 
For $\delta=0$ and under an appropriate identifiability condition on the $\bphi_k$'s, the unique minimizer of $h$ is $y_{d,v,k}=\ell_{d}\theta^*_{dk}\phi_{kv}$, making it possible to derive explicit expressions for $\mu$ as in the case $D=1$.
For $K=2$, distinct topics are sufficient for identifiability.

\vspace{-0.4cm}
\paragraph{Full LDA.}
When the $\bphi_k$'s are unknown and equipped with Dirichlet priors, 
taking expectations and applying Stirling’s approximation in (\ref{eq:lda_posterior}), as above, gives a posterior of the MFGOU form with
\[h(\yy)=e(\yy)+g(\yy),\] 
where $e$ remains unchanged and
\begin{align*}
g(\yy) &= \sum_{k=1}^K \Big(-\sum_{d=1}^D (\delta_k + y_{d,\bullet,k})\log(\delta_k+y_{d,\bullet,k}) \\
&- \sum_{v=1}^V (\omega_v + y_{\bullet,v,k})\log(\omega_v+y_{\bullet,v,k}) \\
&+ (\omega_\bullet + y_{\bullet,\bullet,k})\log(\omega_\bullet + y_{\bullet,\bullet,k}) \Big).
\end{align*}
The Hessian \(\cG\) need not be negative semidefinite. 
Nevertheless, when \(D \geq2\), $\cG$ is not positive definite: 
Fix two documents \(d_1,d_2\), two vocabulary terms \(v_1,v_2\), and one topic coordinate \(k\), and let \(\xx\) be defined by 
\[ x_{d_1,v_1,k} = x_{d_2,v_2,k} = 1, \,\, x_{d_1,v_2,k} = x_{d_2,v_1,k} = -1, \] 
with all other coordinates equal to \(0\). All document-topic, word-topic, and global topic sums of \(\xx\) vanish. 
Since \(g\) depends on \(\yy\) only through these sums, \(g(\yy+t\xx)=g(\yy)\) for all sufficiently small \(t\). 
Hence 
\[\xx^\top\cG(\yy)\xx=0, \] 
and thus $\cG$ cannot be positive definite.
For $D=1$, this is even cleaner, since then 
$\cG(\yy)$ is the sum of a rank-1 matrix and a diagonal matrix with
negative diagonal coefficients.
Hence Corollary \ref{corollary:block_measure_under_concave_g} applies whenever the remaining assumptions of Theorem \ref{thm:projected_block_measure} are satisfied.


For simplicity, write $\Lambda(x)=x\log x$.
We will focus solely on the case $K=2$.
As usual in this case, write $y_{d,v}=y_{d,v,1}$ (so that $y_{d,v,2}=r_{d,v}-y_{d,v}$). Write $c_v=r_{\bullet,v}$ (so that $c_v$ is for words in the vocabulary what $\ell_d=r_{d,\bullet}$ is for documents). 
Then
\begin{align*}
    g(\yy) &= - \sum_d (\Lambda(\delta_1+y_{d,\bullet})+\Lambda(\delta_2+\ell_d-y_{d,\bullet})) \\
    &- \sum_v (\Lambda(\omega_v+y_{\bullet,v})+\Lambda(\omega_v+c_v-y_{\bullet,v})) \\
    &+ \Lambda(\omega_\bullet+y_{\bullet,\bullet})+\Lambda(\omega_\bullet+1-y_{\bullet,\bullet}).
\end{align*}

The case $D=1$ is special. 
Writing $y_v=y_{1,v}$ and $y_\bullet = y_{1,\bullet}$, the Hessian of $h$ becomes
\[\cH(\yy) = \mathrm{diag}(\bb)+a\one\one^\top,\]
where
\[b_v = b_v(\yy) =\frac{1}{y_v}+\frac{1}{c_v-y_v}-\frac{1}{\omega_v+y_v}-\frac{1}{\omega_v+c_v-y_v}\]
and
\[a= a(\yy) = \frac{1}{\omega_\bullet+y_\bullet}+\frac{1}{\omega_\bullet+1-y_\bullet}-\frac{1}{\delta_1+y_\bullet} - \frac{1}{\delta_2+1-y_\bullet}.\]
For $\delta_k=\omega_v=0$, corresponding to Dirichlet hyperparameters $\alpha_k$ and $\beta_v$ that are positive but independent of $n$, the limiting free energy is $h=0$. 
The projected collapsed Gibbs sampler is then a nearest-neighbor chain on a fixed-dimensional lattice domain of diameter $O(n)$ with no macroscopic drift at interior points. 
One expects a mixing time of $\Theta(n^2)$, but the process exhibits a slow-down near the boundary that needs a separate analysis.

When all $\delta_k$ and $\omega_v$ are positive, $\cH(\yy)$ is positive definite for all $\yy$, and hence $h$ satisfies the assumptions of Theorem \ref{thm:projected_block_measure}, whenever 
\[\inf_{\yy \in \prod_v(0,c_v)}\left(1+a(\yy)\sum_v\frac{1}{b_v(\yy)}\right) > 0, \]
giving a mixing time of order $n\log n$ when true.
This rather opaque condition can be shown to hold for example whenever
$\delta_1,\delta_2 \geq \omega_\bullet$.
Suppose that $\delta_1=\delta_2=\delta$ and that $h$ has a unique stationary point at which its Hessian is positive definite. 
Then by symmetry, the stationary point equals $\yy^*=\bm{c}/2$ and it is the unique minimum of $h$.
At $\yy^*$, we have 
\[ B=\frac14\operatorname{diag}(\mathbf c). \] 
Hence $B\mathcal H(\mathbf y^*)$ is similar to the symmetric matrix 
\[ B^{1/2}\mathcal H(\mathbf y^*)B^{1/2} = \operatorname{diag}\left(\frac{c_vb_v}{4}\right) + \frac{a}{4}\sqrt{\mathbf c}\sqrt{\mathbf c}^{\,\top}. \] 
Thus, by the matrix determinant lemma, its eigenvalues are characterized, up to eigenvalues inherited from repeated diagonal entries,
by the secular equation 
\[ 1+a\sum_{v=1}^V \frac{c_v}{c_vb_v-4\lambda}=0. \] 
Writing $\lambda_*$ for the smallest eigenvalue, we get, for every $\varphi\in(0,1)$, 
\[ \taum(\cY;\varphi) = \taum(\cZ;\varphi) = \frac{1}{2\lambda_*}n\log n+O(n). \]

The free energy $h$ frequently has a unique stationary point for $D=1$ even when $\cH$ is not positive definite everywhere, but examples when it does not are not hard to find.
A complete characterization in terms of the $\delta_k$ and $\omega_v$ for when uniqueness holds appears difficult.

For $D \geq 2$, the mixing time is perhaps generically exponential, see e.g. \cite{slow}, 
since $h$ may have several local minima that are not related solely by topic-label switching. 
However, if we assume that the LDA model is well-specified, i.e. there are probabilities $\theta^*_d$ and probability distributions $\phi^*_k$, $k=1,2$, over the vocabulary such that
\[r_{d,v}=r_{d,\bullet}(\theta^*_d\phi^*_{1v} + (1-\theta^*_d)\phi^*_{2v}),\]
the situation is different.
In this case, it can be shown that the system of equations $\nabla h(\yy)=\bz$ reduces to a two-dimensional problem, 
which for $\omega_v=\delta_k=0$ for all $v,k$ has a continuum of solutions forming a two-dimensional manifold. This phenomenon reflects the well-known nonidentifiability of nonnegative matrix factorization; 
there exists a two-dimensional family of parameter triples $(\theta,\phi_1,\phi_2)$ that give rise to the same $\mathbf{r}$.
This strongly suggests that the mixing time of $\cY$ is $\Theta(n^2)$ in this degenerate case.
Introducing positive hyperparameters $\delta_k$ and $\omega_v$ removes this degeneracy. In that case, there are strong reasons to conjecture that in generic situations 
(for example, when all $\theta_d$ and $\phi_{kv}$ are distinct), the free energy $h$ has either one or two isolated local minima. 
However, the system $\nabla h = \bz$ is highly nonlinear and analytically intractable, and we can only support this picture heuristically and numerically, even in relatively simple instances.
We now describe the mixing-time consequences of the conjectured stationary-point structure and then present numerical evidence for that structure.
Once the stationary-point structure of $h$ is understood, the mixing-time behavior follows directly from Theorems \ref{thm:main_theorem} and \ref{thm:exponential_mixing_time}.

Unlike in the Ising model, there is a fundamental distinction between the symmetric case $\delta_1=\delta_2$ and the non-symmetric case $\delta_1 \neq \delta_2$. 
In the symmetric case, any two local minima occur as mirror images under the symmetry $h(\yy)=h(\mathbf{r}-\yy)$, 
and the point $\yy=\tfrac{1}{2}\mathbf{r}$ is always stationary 
(being a saddle point for small hyperparameters and the unique minimizer for large ones). 
From the perspective of the underlying topic modeling problem, however, the two topics carry no intrinsic labels, and interchanging them does not alter the model. 
It is therefore natural to consider the lumped process obtained by identifying configurations that differ only by a global swap of topic labels. 
If the two symmetric minima are the only minima, then they become a single nondegenerate minimum under this identification, and the exponentially slow transition caused by label switching disappears. 
The central stationary point remains an unstable point of the quotient process, so Theorem \ref{thm:main_theorem} does not apply directly. 
Nevertheless, the process is expected to escape from this point on an $n\log n$ time scale and then contract toward the unique minimizer, strongly suggesting that the quotient process has mixing time of order $n\log n$.
On either side of the symmetry point $\bm{r}/2$, the quotient process may be represented by the original process restricted to that side. Theorem 1.1 therefore applies on compact subsets of the basin of either minimum that are bounded away from $\bm{r}/2$. 


In the non-symmetric case, the two topics are no longer interchangeable; the differing hyperparameters encode a genuine prior asymmetry between them. 
As a consequence, distinct local minima of $h$, when present, are not related by any symmetry, and there is no natural way to quotient them out. 
In this regime, multiple nondegenerate local minima therefore imply that the mixing times of both $\cY$ and $\cZ$ are exponential in $n$.
Of course, even in the symmetric case one may consider the unlumped process as a mathematical object in its own right. 
In that case, the chain will exhibit slow transitions between the symmetric modes, 
but this slow mixing is entirely due to label switching and disappears under the natural quotient.

\smallskip

Thus, in the full LDA setting, the main unresolved issue is the analytic description of the stationary points of $h$, rather than the mixing analysis itself; 
conditional on this landscape information, the conclusions for the unlumped process, and for the quotient process away from the unstable central point, follow from the general MFGOU theory.

Some numerical results follow. In all cases, we performed gradient descent from a large number of starting points, and the results consistently support the conjectured picture: there are either one or two minima.

\begin{example}
    The setting here is $D=V=2$ and
    \[\mathbf{r} = \left[\begin{matrix}
        0.15 & 0.35 \\ 0.3 & 0.2
    \end{matrix}\right].\]
    We fix $\omega_v=0.01$ for all $v$, set $\delta_1=\delta_2=\delta$ and vary $\delta$.
    The numerics reveal that the transition from two minima to a unique minimum at the central point occurs at $\delta_c \approx 0.562$. 
     
    Conditional on the landscape indicated by the numerics, for $\delta<\delta_c$, the mixing time of the unlumped process is exponential in $n$, 
    whereas for $\delta>\delta_c$, it equals 
    \[q_\delta n\log n+O(n)\]
    with 
    \[q_\delta=\frac{1}{2\lambda_1(B(\aaa^*)\cH(\aaa^*))}\] 
    and $\aaa^*=\tfrac{1}{2}\mathbf{r}$.
    For $\delta < \delta_c$, quotienting out the topic labels gives a lumped process with mixing time $q_\delta n\log n + O(n)$,
    provided that the process starts at a fixed positive distance from $\tfrac{1}{2}\bm{r}$. 
    The expression for $q_\delta$ remains the same, but now with $\aaa^*$ replaced by either of the two local minima.

Figure \ref{fig:free_energy_various_delta}
illustrates how the free energy changes with $\delta$, showing in particular the transition from two local minima for small $\delta$ to a single minimum for larger $\delta$.
Table \ref{table:q_delta} shows $q_\delta$ for various $\delta$.

\smallskip

    \begin{figure}
   \begin{center}
   \includegraphics[trim = 0mm 0mm 0mm 0mm, clip, width=1.0\textwidth]{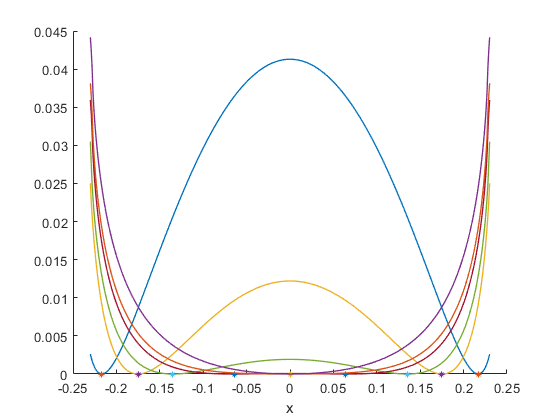}
   \caption{The free energy $h-h_{min}$ evaluated along a one-dimensional slice through the local minima and the central point, for various $\delta$ in the $D=V=2$ case, illustrating the transition from two minima to a single minimum as $\delta$ increases.
   The values used for $\delta$ are $0.001$, $0.1$, $0.3$, $0.5$, $0.6$ and $1$, progressing from $\delta=0.001$ in blue to $\delta=1$ in purple. Respective local minima are marked with asterisks. Recall that $\delta_c \approx 0.56$.} \label{fig:free_energy_various_delta}
   \end{center}
   \end{figure}

\begin{table}[ht]
\centering
\setlength{\tabcolsep}{5pt}
\small
\begin{tabular}{r rrrrrrrrrr}
\toprule
$\delta$ & $0.001$ & $0.01$ & $0.1$ & $0.3$ & $0.5$ & $0.55$ & $0.56\,\,|$ & $0.57$ & $0.6$ & $1$ \\
$q_\delta$ & $29.3$ & $7.9$ & $4.28$ & $10.5$ & $61.5$ & $460$ & $2361\,\,|$ & $610$ & $189$ & $29.6$ \\
\bottomrule
\end{tabular}
\caption{Values of the mixing time constant $q_\delta$ for different values of $\delta$. 
Above $\delta_c$, this is the mixing-time constant of both the projected process and the full Gibbs sampler. 
Below $\delta_c$, it is the mixing-time constant of the quotient process for initial states bounded away from the unstable central point. The vertical bar marks the transition between these two regimes.}
\label{table:q_delta}
\normalsize
\end{table}

\end{example}

A few examples of higher values of $D$ and $V$ for well-specified LDA have also been evaluated numerically, and they support the same picture: two symmetric local minima for $\delta<\delta_c$ and
a single minimum for $\delta \geq \delta_c$ for some critical $\delta_c>0$.

\subsection{Stochastic gradient descent} \label{subsec:SGD}

  Consider a statistical model of the following form
  \[y=\varphi(\xx^T\btheta)+\ep,\]
  where $\xx \in \RR^p$ is a feature vector associated with the observation $y$, $\btheta \in \bTheta$ is a parameter vector and $\ep$ is an error term that has zero mean and finite variance $\si^2$, independent of $\xx$. 
  Given data $(\xx_i,y_i)$, $i=1,\ldots,D$, one wants to optimize $\btheta$ based on these data. 
  To estimate $\btheta$, one specifies a loss function, 
  $\ell:\bTheta \times \RR^p \times \RR \ra \RR$, assumed to be continuously differentiable in $\theta$, with $\nabla_{\btheta}\ell(\btheta,\xx,y)$ jointly continuous in $(\btheta,\xx,y)$.
  For an observation $(\xx_i,y_i)$, the quantity $\ell(\btheta;\xx_i,y_i)$ is the corresponding loss. 
  From this, one derives the total loss
  \[\Leb(\btheta):= \frac1D\sum_{i=1}^{D}\ell(\btheta;\xx_i,y_i),\]
  which is then the minimization objective.
  
  Usually explicit minimization is difficult, and some numerical method must be employed. Among such methods, gradient descent is arguably the most well-known. 
  The method is iterative: starting from an initial value $\btheta_0$, one updates according to
  \[\btheta_{t+1} = \btheta_t-\eta\nabla_{\btheta} \Leb(\btheta_t),\]
  where the learning rate $\eta>0$ is a small fixed positive number. 
  This recursion is then run for a sufficiently long time. 
  
  However, in most modern machine learning applications, the data set and the dimension of $\btheta$ are so large that it is too costly to compute the full gradient for each step. One standard remedy is instead to use stochastic gradient descent (SGD), where a small positive integer $k$ is fixed and at each step of the recursion, one picks a mini-batch, i.e., a uniformly chosen random subset $S \subset [D]$ of size $k$, and updates according to
  \[\btheta_{t+1} = \btheta_t-\eta\nabla_{\btheta} \Leb^S(\btheta_t),\]
  where $\Leb^S(\btheta) = \frac1k\sum_{i \in S}\ell(\btheta;x_i,y_i)$. Often $k$ is simply set to be $1$, which is the case we will consider here. 
  Let $\cU=\{\bU_t\}:=\{\btheta_t\}_{t=0}^\iy$ denote the resulting process. 
  Then $\cU$ forms a Markov process on $\bTheta$ and under suitable recurrence assumptions (e.g.\ if $\bTheta$ is bounded) admits a stationary distribution concentrated around $\hat{\btheta}:=\text{argmin}\Leb(\btheta)$. 
  Constant-step-size SGD has previously been studied from a Markov-chain perspective by Dieuleveut, Durmus and Bach \cite{DieuleveutDurmusBach2020}. 
  More recently, Merad and Ga\"iffas \cite{MeradGaiffas2025} established convergence to an invariant distribution in total variation and Wasserstein distance under strong convexity and regularity assumptions. 
  Related total-variation mixing bounds for projected Langevin iterations, including nonsmooth and weakly smooth potentials, were obtained by Bravo, Flores-Mella and Guzmán \cite{BravoFloresMellaGuzman2025}.
  Our purpose here is complementary: the MFGOU framework identifies the leading total-variation mixing-time asymptotics and their dependence on the Hessian at the minimum.

  The step size of $\cU$ is $-\eta \nabla_{\btheta} \ell(\btheta;\bX,Y)$, where $(\bX,Y)$ is chosen according to the empirical distribution
  of the data.
  If $D$ is large, then this is close to the true underlying data generating distribution. 
  We therefore consider the idealized infinite-data version, in which $(\bX,Y)$ is sampled independently from the data-generating distribution at each step.
  For the remainder of the general discussion, assume that $(X,Y)$ has bounded support.
  Let
  \[g(\btheta)=\E[\ell(\btheta;\bX,Y)],\]
  where the expectation is with respect to the random choice of $(\bX,Y)$, and assume that $g$ is three times continuously differentiable and that differentiation under the expectation is justified, so that
  \[\nabla g(\btheta) = \E[\nabla_{\btheta}\ell(\btheta;\bX,Y)].\]
  Then
  \[\E[\bU_{t+1}-\bU_t|\bU_t=\btheta] = -\eta \nabla g(\btheta).\]
  Set $\eta=\tfrac{1}{2n}$.
  The resulting increment variables then become
  \[\bxi(\btheta) = -\frac{1}{2}\nabla_{\btheta} \ell(\btheta;\bX,Y),\]
  whose law is independent of $n$. Under suitable local boundedness, regularity, and nondegeneracy assumptions on this law, 
  the resulting processes have the local structure of an MFGOU family with $\ff(\btheta)=\nabla g(\btheta)$.
  
  However, if, as is commonly the case, $\nabla \ell$ is unbounded globally in $\btheta$, then the common compact support assumption of the $\bxi(\btheta)$'s fails.
  On the other hand, if we restrict attention to a compact domain $K_2 \subseteq \RR^p$,
  joint continuity of $\nabla \ell$ and compact support of $(\bX,Y)$ give a common compact support for $\bxi(\btheta)$, $\btheta \in K_2$.
  
  Suppose that $K_0,K_1,K_2$ are compact sublevel sets of $g$ such that $K_0 \subset K_1^\circ \subset K_1 \subset K_2^\circ$, and assume that
  $\nabla g$ only vanishes at a unique minimizer $\btheta^*$ of $g$, the Hessian $\cG(\btheta^*)=\nabla^2 g(\btheta^*)$ is positive definite, and
   \[\n\nabla g(\btheta)\n^2 \geq c((g(\btheta)-g(\btheta^*)) \wedge 1)\]
  holds for some $c>0$ (which holds e.g.\ if $\n\nabla g\n$ is bounded away from $0$ on all $B(\btheta^*,\kappa)^c$).
  Then if $\bU_0 \in K_0$, Lemma \ref{lemma:to_small_box_in_linear_time} shows that the process whp stays in $K_1$ for an exponential time, and the censoring mechanism of Section \ref{subsec:to_small_box_in_linear_time} applies with $V(\btheta)=g(\btheta)-g(\btheta^*)$.
  Applying the censoring construction to each $\cU^{(n)}$, the resulting processes $\cU^{(n,r)}$, $n \in \cN$, confined to $K_2,$ form an MFGOU family.
  Since $\cU$ and $\cU^{(r)}$ agree until $\cU$ leaves $K_1$, an event that has exponentially small probability on the relevant time scale, we may identify the two processes for the purpose of estimating the mixing time.
  This is also a natural restriction in practice, where the initial parameter values are confined to some prescribed bounded region, and fixed-learning-rate SGD may itself be unstable when stochastic gradients grow too rapidly at infinity.
  
  The Lyapunov hypotheses of Theorem \ref{thm:main_theorem} hold by Remark \ref{remark:various_things_on_MFGOU_assumptions}(i), and it follows that 
  \[\taum(\cU;\ep,K_0)=qn\log n+O(n) = \frac{q}{2\eta}\log(\eta^{-1})+O(\eta^{-1})\]
  with $q=1/\lambda_1(\cG(\btheta^*))$.
  
  If $g$ has an additional nondegenerate local minimum, the mixing time is exponential.
  In relation to this, metastability and escape-time behavior for nonconvex SGD have been studied through diffusion approximations by Hu et al.\ \cite{HuLiLiLiu2019}.
  
  \smallskip

  \begin{example}
      {\bf Linear and robust regression:} Consider the standard multiple linear regression model without an intercept
      \[y=\xx^\top\bbeta^*+\ep\]
      where $\ep$ is a symmetric absolutely continuous error term independent of $\bX$. 
      Assume in addition that 
      \begin{itemize}
          \item $\bX$ is absolutely continuous, 
          \item $\ep$ has compact support (which in the well-specified case follows from the compact support of $(\bX,Y)$), 
          \item the densities of $\ep$ and $\bX$ are bounded away from $0$ in a neighborhood of $\bz$ (in $\RR$ and $\RR^p$, respectively),
          \item the law of $\ep$ is Lipschitz under translation in total variation (which holds, for example, whenever $\ep$ has a density of bounded variation).
      \end{itemize}

      Define the loss function
      \[\ell(\bbeta;\xx,y) = \rho(y-\xx^\top\bbeta)\]
      and $g(\bbeta) = \E[\rho(Y-\bX^\top\bbeta)]$, where $\rho(x)$
      is an even, nonnegative, continuously differentiable function that is increasing in $|x|$ and satisfies $\rho'\neq 0$ almost everywhere.
      
      Write $\bbeta^*$ for the minimizer of $g$, so that $Y=\bX^\top\bbeta^*+\ep$ whenever the model is well-specified, and assume that $g$ is $C^3$. 
      Assume that $g$ is unimodal with $\n\nabla g\n$ bounded away from $0$ as above and $\nabla^2g(\bbeta^*)$ positive definite. We next verify that the stochastic-gradient increments satisfy the local regularity conditions of the MFGOU framework. Note that the SGD increments in this application are
      \[\bxi(\bbeta) = \frac{1}{2}\bX\rho'(\ep+\bX^\top(\bbeta^*-\bbeta)).\]
      \begin{itemize}
      \item Since $(\bX,\ep)$ is absolutely continuous, so is $(\bX,\ep+\bX^\top(\bbeta^*-\bbeta))$.
      Conditional on $\ep+\bX^\top(\bbeta^*-\bbeta)=r$, the law of the increments $\bxi(\bbeta)=\tfrac{1}{2}\bX\rho'(r)$
      is absolutely continuous for almost every $r$, and hence so is the unconditional law of $\xi(\bbeta)$.
      
      \item Since $\bX$ and $\ep$ have compact support, the increments have common compact support for $\bbeta$ in any compact domain.
      \item Since the law of $\ep$ is Lipschitz under translation in total variation, the law of $\bxi(\bbeta)$ is Lipschitz in total variation as a function of $\bbeta$.
      \item Since the law at $\bbeta^*$ is symmetric and the densities of $\ep$ and $\bX$ are bounded away from $0$ around $\bz$, a common symmetric absolutely continuous component of uniformly positive mass can be extracted in a neighborhood of $\bbeta^*$.
      \end{itemize}
      Consequently, the processes have the local structure required of an MFGOU family, and the corresponding censored processes form an MFGOU family. Hence, for the initial region $K_0$ considered above, 
      \[\taum(\cU;\varphi,K_0)=\frac{q}{2\eta}\log(\eta^{-1})+O(\eta^{-1})\] 
      for every $0<\varphi<1$, where $q^{-1} = \lambda_1(\nabla^2 g(\bbeta^*))$.

      For concrete results, consider $\ep \sim N(0,\si^2)$ and $\bX \sim N(0,\Sigma)$ and the loss function $\rho(x)=|x|^c$ for some $c \in (1,2]$.
      Then 
      \[g(\bbeta)=\E[|\ep+\bX^\top(\bbeta^*-\bbeta)|^c] = (\si^2+(\bbeta^*-\bbeta)^\top\Sigma(\bbeta^*-\bbeta))^{c/2}\mu_c,\]
      where $\mu_c=\E[|Z|^c]$ for a standard normal $Z$.
      Differentiating twice and setting $\bbeta=\bbeta^*$ gives
      \[\nabla^{2}g(\bbeta^*) = c\mu_c\si^{c-2}\Sigma.\]
      We cannot apply the MFGOU mixing-time result directly, since $\ep$ and $\bX$ have unbounded support. However, let $\ep_R$ and $\bX_R$ have the
      respective conditional laws of $\epsilon$ given $|\epsilon|\leq R$ and of $\bX$ given $\n\bX\n \le R$. 
      Let $g_R$ denote the corresponding expected loss. Dominated convergence then gives
      \[\lim_{R \ra \iy} \nabla^2 g_R(\bbeta^*) = \nabla^2 g(\bbeta^*).\]
      For each fixed $R$, the corresponding censored SGD processes form an MFGOU family.
      It follows that
      \[\taum(\cU_R;\varphi,K_0) = \frac{q_R}{2\eta}\log (\eta^{-1}) + O_R(\eta^{-1}),\]
      where
      \[\lim_{R \ra \iy}q_R = \frac{\sigma^{2-c}}{c\mu_c \lambda_1(\Sigma)}.\]
      
      The main reason for considering losses other than squared loss, however, is to perform robust regression in settings where the model may be misspecified or the data may contain outliers.
      In such settings, one typically seeks a loss that is more forgiving of a few large residuals than squared loss, which penalizes such observations heavily.
      Examples covered by the preceding discussion include the $|x|^c$-loss, $1 < c\leq 2$, Huber-type losses, and Welsch loss. 
     The absolute loss $\rho(x)=|x|$ is not differentiable at zero. 
     This causes no immediate difficulty when the residual has an absolutely continuous distribution, 
     since the stochastic gradient then exists almost surely. 
     What is required for the MFGOU analysis is sufficient smoothness of the expected loss $g$. 
     Redescending losses such as Tukey’s loss may instead produce an atom at zero in the stochastic-gradient law, in which case the semicontinuous decomposition must be verified separately.

      The well-specified representation is not essential. In a misspecified model, the argument continues to apply with $\bbeta^*$ defined as the minimizer of the population loss, provided that the stochastic-gradient laws satisfy the same local assumptions near $\bbeta^*$.

      Unfortunately, the Hessian of $g$ at $\bbeta^*$ can rarely be evaluated analytically except in specially constructed examples, and even then substantial work may be required.
      One can easily construct examples exhibiting phase transitions between fast and slow mixing as hyperparameters vary, but we do not pursue such examples here.
      
  \end{example}

  \begin{example}
      Consider logistic regression without intercept, where the $y_i$'s are binary and 
      \[\Pro(Y=1|\mathbf{X}=\xx) = 1-\Pro(Y=0|\bX=\xx) = \sigma(\xx^T\bbeta^*):=\frac{1}{1+e^{-\xx^T\bbeta^*}}.\]
      The loss function is taken to be negative log-likelihood, i.e.,
      \[\ell(\bbeta;\xx_i,y_i)= y_i\log(1+e^{-\xx_i^T\bbeta})+(1-y_i)\log(1+e^{\xx_i^T\bbeta}).\]
      The loss is convex in $\bbeta$, and if the second-moment matrix $\Sigma:=\E[\bX\bX^\top]$ is positive definite, then the expected Hessian $\nabla^2 g(\bbeta)$ is positive definite for every $\bbeta$.
      Hence $g$ is strictly convex, and assuming that a minimizer $\bbeta^*$ exists, it is unique,
      The corresponding SGD increments are
      \[\bxi(\bbeta)=\frac{1}{2}\bX(Y-\sigma(\bX^\top\bbeta)).\]
      Assume that $\bX$ is absolutely continuous, with density bounded away from zero on a neighborhood of the origin. 
      The lower bound on the density near the origin ensures that the increment laws contain a common symmetric absolutely continuous component of uniformly positive mass.
      Assume also that, for $\bbeta$ in a neighborhood of $\bbeta^*$, the SGD increment laws are Lipschitz in total variation as functions of $\bbeta$.
      Under these assumptions, the corresponding censored processes form an MFGOU family on $K_2$.
      
      Computing the first and second derivatives, and using dominated convergence to differentiate under the expectation, gives
      \[\nabla^2 g(\bbeta^*) = \E\left[\frac{\bX\bX^T}{(1+e^{-\bX^T\bbeta^*})(1+e^{\bX^T\bbeta^*})}\right].\]
      The mixing-time result applies with 
      \[ \tau_{\mathrm{mix}}(\cU;\varphi,K_0) = \frac{1}{2\eta\lambda_1(\nabla^2 g(\bbeta^*))} \log(\eta^{-1}) + O(\eta^{-1}). \]
      In general, $\nabla^2 g(\bbeta^*)$ cannot be computed explicitly. 
      However, in the special case of $\bbeta^*=\bz$, the denominator equals $4$ and hence $\nabla^2 g(\bbeta^*)=\frac{1}{4}\Sigma$ and this simplifies to
      \[ \tau_{\mathrm{mix}}(\cU;\varphi,K_0) = \frac{2}{\eta\lambda_1(\Sigma)} \log(\eta^{-1}) + O(\eta^{-1}). \]
      For a concrete example, suppose that $\bbeta^*=\bz$ and the model is well-specified, i.e.\ $Y$ is an independent fair coin flip. 
      Then the increment law at the minimizer is the symmetric mixture of $\bX/4$ and $-\bX/4$. 
      Thus, if $\bX$ has compact support and a sufficiently regular density bounded away from zero near the origin, the local continuous MFGOU conditions hold.
    \end{example}

    \subsection{Hamiltonian Monte Carlo} \label{subsec:HMC}

    Hamiltonian Monte Carlo is a method for sampling from high-dimensional continuous distributions
      \[H(\xx) \propto e^{-nh(\xx)},\, \xx \in \RR^m,\]
      where $n$ is a large number and $h:\RR^m \ra [0,\iy)$. 
      A typical example is a Bayesian posterior based on $n$ observations, in which case $nh$ is the sum of the negative log-likelihood and the negative log-prior.
      
      We assume that $h$ is three times continuously differentiable and unimodal, that $\nabla h$ vanishes only at the minimum
      and is bounded away from zero on the complement of every neighborhood of the minimum, and that the Hessian of $h$ at the minimum is positive definite.
      Without loss of generality, we assume that $h$ is minimized at the origin and $h(\bz)=0$.
      
     When $m$ is large, HMC can mix faster than corresponding random-walk Metropolis--Hastings or Gibbs samplers for unimodal targets. Although rigorous mixing-time results are available in several important settings, they mostly concern asymptotic regimes different from the one considered here.

      Without going into the details of the physics that inspires it, let us now define HMC. 
      The idea is to sample from the $2m$-variate product density 
      \[\psi(\xx,\yy)=e^{-nh(\xx) - \yy^T\yy/2},\]
      which is, up to normalization, the joint density of $(\bX,\bY)$, where $\bX \sim H$ and $\bY \sim N(\bz,I_m)$ is independent of $\bX$. 
      The auxiliary process $\cY=(\bY(t))_{t \geq 0}$ plays the role of a Hamiltonian momentum. Concretely, HMC goes as follows: begin by choosing a starting position $\bX(0)$ and $\bZ_0 \sim N(\bz,I_m)$. Then follow the differential equations
\[\bX'(t)=\bY(t),\, \bY'(t) = -n\nabla h(\bX(t)),\, \bY(0)=\bZ_0\]
up to some predefined time $T$.
After that, reset $\cY$ to a new independent $\bZ_1 \sim N(0,I_m)$ and follow the differential equations for another time period of length $T$, starting from $(\bX(T),\bZ_1)$.
Then iterate this procedure for the desired number of steps.

We restrict attention to the frequent-refreshment regime, in which the interval between momentum refreshments is of order $1/n$.

In practice, this system of ODEs generally cannot be solved explicitly, so a numerical integrator must be used. The standard numerical integrator for HMC is the leapfrog integrator. We consider a single leapfrog step between momentum refreshments and do not apply a Metropolis correction for the discretization error.
In the present frequent-refreshment setting, the time between refreshments is sufficiently short that a single leapfrog step closely approximates the exact Hamiltonian motion over that interval.

Quantitative total-variation mixing bounds for unadjusted HMC, together with bounds on the discrepancy between its invariant distribution and the target distribution, were obtained by Bou-Rabee and Eberle \cite{BouRabeeEberle2023}. 
Their work primarily addresses dependence on dimension, accuracy and discretization.
Rapid-mixing bounds for ideal HMC and its leapfrog implementation on strongly log-concave targets were obtained by Mangoubi and Smith \cite{MangoubiSmith2021,MangoubiSmith2019}. 
Their results focus on dimension and condition-number dependence for trajectories of appreciable length.
The regime considered here is different: the dimension is fixed, the target becomes increasingly concentrated as $n$ grows, and the momentum is refreshed after a single leapfrog step of length $a/n$. Our aim is to identify the leading mixing-time asymptotics in this regime.
The leapfrog algorithm is carried out as
\begin{align*}
    & \bY(T/2) = \bZ_0-\frac{nT}{2}\nabla h(\bX(0)), \\
    & \bX(T) = \bX(0)+T\bY(T/2), \\
    & \bY(T) = \bY(T/2) - \frac{nT}{2} \nabla h(\bX(T)).
\end{align*}
For the $\cX$ process, this gives simply that
\[\bX(T) = \bX(0)-\frac{nT^2}{2}\nabla h(\bX(0))+T\bZ_0.\]
Taking $T=a/n$ for a constant $a$, this reads 
\[\bX(a/n) = \bX(0)-\frac{a^2}{2n}\nabla h(\bX(0))+\frac{a}{n}\bZ_0.\]
Defining the discrete-time process $\cU=\{\bU_k\}_{k=0}^\iy$ by $\bU_k=\bX(ka/n)$ then gives
\[\bU_{k+1}=\bU_k-\frac{a^2}{2n}\nabla h(\bU_k)+\frac{a}{n}\bZ_k.\]
Thus, after one leapfrog step and momentum refreshment, the position chain takes the form of an Euler-type Langevin recursion. We nevertheless retain the HMC terminology to emphasize its derivation from the leapfrog dynamics.

The resulting processes do not quite form an MFGOU family according to the definition, since $\nabla h$ may be unbounded and since the Gaussian random variables $\bZ_k$ have unbounded support. 
For the second of these problems, we replace the $\bZ_k$'s with their truncated versions $\bZ_{R,k}$ with the conditional law of $\bZ$ given $\n\bZ\n \leq R$ for an arbitrary fixed $R<\iy$, 
and consider the resulting process $\cU_R$.
For each fixed $R$, the truncated Gaussian law is symmetric, absolutely continuous, independent of $n$,
and Lipschitz under translation in total variation.
Since the increment law is a smoothly varying translate of the truncated Gaussian law and is symmetric at the origin, 
a fixed common symmetric absolutely continuous component of uniformly positive mass can be extracted in a neighborhood of the origin.
Hence the local continuous MFGOU regularity conditions hold in a neighborhood of the origin.

For the first obstacle, we apply the censoring construction of Section \ref{subsec:to_small_box_in_linear_time} on an arbitrarily large compact sublevel set $K_2$ of $h$. 
This allows $\cU_R$, started in a compact $K_0\subset K_2^\circ$, to be treated as an MFGOU process confined to $K_2$, up to an exponentially small error on the relevant time scale.

With drift field $\ff(\uu)=a^2 \nabla h(\uu)$, 
and using Remark \ref{remark:various_things_on_MFGOU_assumptions} (i) to confirm the Lyapunov conditions,
Theorem \ref{thm:main_theorem} yields
\[\taum(\cU_R;\ep,K_0) = \frac{1}{a^2\lambda_{1}(\cH)}n \log n + O_R(n).\]
The mixing time here is with respect to the stationary distribution of the truncated unadjusted chain $\cU_R$, 
which is not the target distribution proportional to $e^{-nh}$. 
For fixed $R$, truncation changes the covariance of the noise and may therefore produce a discrepancy on the equilibrium scale, but this discrepancy is expected to vanish as $R \to \iy$. 
The remaining discretization bias of the untruncated chain is separate, and, in the quadratic case, is of relative order $O(1/n)$.
The result concerns mixing to the stationary distribution of $\cU_R$, not approximation of the target distribution.

We expect that allowing Gaussian increments, despite their unbounded support, 
would require only minor modifications of the MFGOU analysis, particularly in the present setting. 
Over the relevant time scale, the probability that the norm of any increment exceeds $\log n/n$ is superpolynomially small. 
We therefore expect the same leading mixing-time asymptotics for the untruncated process, 
though possibly with the lower-order $O(n)$ term replaced by $o(n\log n)$.

\begin{remark} \label{remark:m=m(n)}
    {\rm In the quadratic case $h(\xx)=\tfrac{1}{2}\xx^T\cH\xx$, the recursion
\[\bU_{k+1}=\left(I-\frac{a^2}{2n}\cH\right)\bU_k+\frac{a}{n}\bZ_k,\]
$\bZ_k \sim N(\bz,I_m)$,
is explicitly solvable. In particular, one can compute the covariance structure exactly and analyze the dependence of the mixing time on both the dimension $m$ and the initial state. 
When $m=m(n)$ grows with $n$, this yields an additional logarithmic dependence on $m$, suggesting mixing times of order $n\log(mn)$. 
Although this observation arises here in the context of HMC, it indicates a natural direction for the general MFGOU theory: allowing $m=m(n)$. 
The quadratic case suggests that such an extension should remain valid whenever the nonlinear drift is sufficiently close to its linearization near equilibrium, although quantifying the required degree of closeness appears to require significant care.}
\end{remark}

\section{Discussion and open problems} \label{sec:discussion_and_open_problems}

The results of this paper leave several natural directions for further work. The most obvious one concerns the landscape of the free energy for full LDA. 
The MFGOU theory reduces the mixing analysis to the stationary-point structure of this landscape, but a complete analytic description of that structure remains open. 
When the Dirichlet hyperparameters are symmetric, we conjecture, subject to suitable identifiability and nondegeneracy conditions, that the following are equivalent for sufficiently small hyperparameters: the only local minima are two minima related by topic-label switching, and the observed word-frequency matrix admits a genuine underlying LDA factorization in terms of document-specific topic weights and two topic-word distributions.
It would be valuable to formulate and prove this conjecture precisely, to describe how the landscape changes as the hyperparameters increase, 
and to establish the mixing time of the quotient process when the central stationary point is unstable. 
Such results would complete the mixing-time analysis of the full LDA models considered here.

A second direction is to allow the dimension $m$ of the MFGOU process to grow with $n$. The present theory keeps $m$ fixed, and its constants may depend on $m$. 
As observed in Remark \ref{remark:m=m(n)}, the quadratic case can be analyzed explicitly and suggests an additional logarithmic dependence on the dimension. 
Closely related high-dimensional cutoff phenomena for Ornstein--Uhlenbeck processes have been analyzed by Boursier, Chafa\"i and Labb\'e \cite{BoursierChafaiLabbe2023}.
A growing-dimensional theory would require uniform versions of the variance, mean-tracking and coupling estimates, together with suitable control of the nonlinear remainder. 
Determining conditions under which the fixed-dimensional mixing-time formula remains valid, possibly with an additional dependence on $m$, is a natural next step. 

A more ambitious project concerns degenerate stationary points of the drift field. 
The present theory treats nondegenerate attracting zeros of the drift and gives exponentially slow mixing when there is more than one such attractor. 
At a critical parameter value, however, the Jacobian may have zero eigenvalues, and the leading restoring force is then determined by higher-order terms. 
One should consequently expect mixing-time scales different from $n\log n$, depending on the order and geometry of the degeneracy. 
A general theory covering higher-order minima, saddles and unstable stationary points would encompass critical phenomena in mean-field spin systems, 
escape from the unstable central point in quotient LDA, and analogous transitions in stochastic optimization. 
Developing such a theory appears to require substantially different techniques from those used here.

\bigskip

\end{document}